\documentclass[10pt,a4paper]{article}
\usepackage[utf8]{inputenc}
\usepackage[american]{babel}
\usepackage[all]{xy}
\usepackage{enumitem}
\usepackage[dvipsnames]{xcolor}
\usepackage{graphicx}
\usepackage{stmaryrd}
\usepackage{amsfonts,amstext,amssymb,amsmath,amsthm,pspicture,bigints}
\usepackage{fullpage}
\usepackage{moreverb}
\usepackage{pifont}
\usepackage{pst-all}
\usepackage{tikz}
\usepackage[left=2cm,right=2cm,top=2cm,bottom=2cm]{geometry}
\usepackage{esint}
\usepackage{fourier-orns}
\usepackage{mathrsfs}
\usepackage{array}

\usepackage{hyperref}
\newcommand{\ii }{{\rm i} }
\usepackage{authblk}
\newcolumntype{C}[1]{>{\centering\arraybackslash}b{#1}}
\newcolumntype{R}[1]{>{\raggedleft\arraybackslash}b{#1}}
\newcolumntype{L}[1]{>{\raggedright\arraybackslash}b{#1}}
\newcolumntype{M}[1]{>{\centering}m{#1}}
\newtheorem{theo}{Theorem}[section]
\newtheorem{defin}{Definition}[section]
\newtheorem{lem}{Lemma}[section]

\newtheorem{remark}{Remark}[section]

\usepackage{bbm}
\numberwithin{equation}{section}

\usepackage{subfigure}

\usepackage{pgfplots}
\usepackage{amsfonts,amsmath}
\usepackage{tikz-3dplot}
\pgfplotsset{compat=newest}
\usetikzlibrary{calc}
\usetikzlibrary{quotes,angles}
\usepackage{pgfplots}
\pgfplotsset{compat=1.12}
\usetikzlibrary{math}

\date{}
\title{Local and global bifurcation of electron-states} 
	
\author{Emeric Roulley\thanks{International School for Advanced Studies (SISSA), Via Bonomea 265, 34136, Trieste, Italy.\\
		E-mail address : eroulley@sissa.it}}
\begin{document}
	\maketitle
	\begin{abstract}
		We study the bifurcation of traveling periodic electron layers, that we call \textit{electron-states}, from symmetric and asymmetric flat velocity strips in the phase space, for the one dimensional Vlasov-Poisson equation with space periodic condition. The boundaries of the constructed solutions are real-analytic and in uniform translation at the same speed in the space direction. These structures are obtained applying Crandall-Rabinowitz's Theorem using either the velocity or geometrical quantities related to the size of the strip as bifurcation parameters. In the first case, we can prove for any fixed symmetry, the emergence of a pair of branches and the local bifurcation diagram has a hyperbolic structure. In the symmetric situation, we find, for any large enough symmetry, one branch whose orientation close to the stationary solution depends on the sign of the prescribed speed of translation. As for the asymmetric case, we find either a countable or a finite number of bifurcation curves according to some constraints related to the prescribed speed of translation. The pitchfork (subcritical or supercritical) bifurcation is also described in this case. Finally, we briefly discuss the global continuation of these branches.
	\end{abstract}
	\tableofcontents
	\section{Intoduction}
	We present here the equation studied in this work which is a kinetic model in dimension one with space periodic boundary conditions. Then, we discuss a particular class of weak solutions to this equation called \textit{electron layers} that are renormalized characteristic functions of time and space dependent velocity domains. The velocity flat strips provide stationary solutions and we present some pertubative existence results of time periodic solutions close to these equilibrium states.
	\subsection{One dimensional Vlasov-Poisson equation and patches of electrons}
	We consider the 1D Vlasov-Poisson equation with space 1-periodic boundary condition
	\begin{equation}\label{VP eq}
		\partial_{t}f(t,x,v)+v\,\partial_{x}f(t,x,v)-E(t,x)\,\partial_{v}f(t,x,v)=0,\qquad(t,x,v)\in\mathbb{R}_{+}\times\mathbb{T}\times\mathbb{R}.
	\end{equation}
	Here $\mathbb{T}=\mathbb{R}/\mathbb{Z}$ denotes the flat torus that we liken to the segment $[0,1]$ where $0$ and $1$ are identified. The equation \eqref{VP eq} is a model that can be found in \cite[Chap. 13]{BM02} or \cite{D87}. It describes a collisionless neutral plasma composed with ions and electrons. The ions' significant inertia enables us to consider them as a neutralizing uniform background field. The unknown $f(t,x,v)$ represents the density of electrons traveling with speed $v$ at position $x$ and time $t.$ We make the assumption that the plasma properties are one-dimensional. Hence, the transport is unidirectional and the problem is simplified to one space dimension. In particular, the particle motion is only influenced by induced electrostatic forces and therefore we disregard electromagnetic interactions. The electric field $E$ is associated to the electric potential $\boldsymbol{\varphi}$ as follows
	\begin{equation}\label{def boldvarphi}
		E(t,x)=\partial_{x}\boldsymbol{\varphi}(t,x),\qquad\partial_{xx}\boldsymbol{\varphi}(t,x)=1-\int_{\mathbb{R}}f(t,x,v)dv.
	\end{equation}
	According to \eqref{def boldvarphi} and Taylor formula, the periodic constraint
	$$E(t,0)=E(t,1)$$
	is equivalent to the neutrality condition
	\begin{equation}\label{neutre cond}
		\int_{0}^{1}\int_{\mathbb{R}}f(t,x,v)dxdv=1.
	\end{equation}
	Observe that in the problem the quantity of interest is $\partial_{x}\boldsymbol{\varphi}$. Then $\boldsymbol{\varphi}$ is defined up to a time dependent additive constant that we can choose in order to impose, for any time, a zero space average condition for $\boldsymbol{\varphi}.$ As a consequence, introducing the inverse Laplace operator $\partial_{xx}^{-1}$ defined as follows
	$$\forall j\in\mathbb{Z}^*,\quad\partial_{xx}^{-1}\mathbf{e}_j=\frac{-\mathbf{e}_j}{4\pi^{2}j^2},\qquad\mathbf{e}_{j}(x)\triangleq e^{2\ii\pi jx},$$
	we get from \eqref{def boldvarphi}-\eqref{neutre cond}
	\begin{equation}\label{inv Lap phi}
		\boldsymbol{\varphi}(t,x)=\partial_{xx}^{-1}\left(1-\int_{\mathbb{R}}f(t,x,v)dv\right).
	\end{equation}
	The equation \eqref{VP eq} can recast as an active scalar equation. To this aim, we see the phase space $\mathbb{T}\times\mathbb{R}$ as a cylinder manifold embedded in $\mathbb{R}^3$ with radius $r=1$ and with vertical axis soul. The identification can be done through the local chart
	$$\begin{array}{rcl}
		(0,1)\times\mathbb{R} & \rightarrow & \mathbb{R}^3\\
		(x,v) & \mapsto & \big(\cos(2\pi x),\sin(2\pi x),v\big).
	\end{array}$$
At any point $(x,v)\in\mathbb{T}\times\mathbb{R},$ the tangent plane $T_{(x,v)}(\mathbb{T}\times\mathbb{R})\equiv\mathbb{R}^2$ admits the orthonormal basis (with the classical identification vector/directional derivative)
$$\mathtt{e}_{x}\triangleq\partial_{x},\qquad\mathtt{e}_{v}\triangleq\partial_{v}.$$
For any function $\mathtt{g}:\mathbb{T}\times\mathbb{R}\rightarrow\mathbb{R},$ the gradient is given by
$$\nabla_{x,v}\mathtt{g}(x,v)=\partial_{x}\mathtt{g}(x,v)\mathtt{e}_{x}+\partial_{v}\mathtt{g}(x,v)\mathtt{e}_{v}.$$
The orthogonal gradient is obtained by a rotation of angle $\tfrac{\pi}{2}$
$$\nabla_{x,v}^{\perp}\triangleq\mathtt{J}_{x,v}\nabla_{x,v},\qquad\underset{(\mathtt{e}_{x},\mathtt{e}_{v})}{\textnormal{Mat}}(\mathtt{J}_{x,v})=\begin{pmatrix}
	0 & -1\\
	1 & 0
\end{pmatrix}.$$
Consider the velocity field
$$\mathbf{v}:\begin{array}[t]{rcl}
	\mathbb{T}\times\mathbb{R} & \rightarrow & \displaystyle T(\mathbb{T}\times\mathbb{R})\triangleq\bigcup_{(x,v)\in\mathbb{T}\times\mathbb{R}}T_{(x,v)}(\mathbb{T}\times\mathbb{R})\\
	(x,v) & \mapsto & v\mathtt{e}_{x}-E(t,x)\mathtt{e}_{v},
\end{array}$$
which is divergence-free
\begin{equation}\label{div=0}
	\textnormal{div}_{x,v}\mathbf{v}(t,x,v)=\partial_{x}(v)+\partial_{v}\big(-E(t,x)\big)=0.
\end{equation}
More precisely, we can write
\begin{equation}\label{vel pot}
	\mathbf{v}(t,x,v)=-\nabla_{x,v}^{\perp}\boldsymbol{\Psi}(t,x,v),\qquad\boldsymbol{\Psi}(t,x,v)\triangleq\frac{v^2}{2}+\boldsymbol{\varphi}(t,x).
\end{equation}
Then, the equation \eqref{VP eq} becomes
\begin{equation}\label{active eq}
		\partial_{t}f(t,x,v)+\Big\langle\mathbf{v}(t,x,v)\,,\nabla_{x,v}f(t,x,v)\Big\rangle_{T_{(x,v)}(\mathbb{T}\times\mathbb{R})}=0,
	\end{equation}
where the scalar product defined by
	$$\Big\langle\alpha(x,v)\mathtt{e}_{x}+\beta(x,v)\mathtt{e}_{v}\,,\gamma(x,v)\mathtt{e}_{x}+\delta(x,v)\mathtt{e}_{v}\Big\rangle_{T_{(x,v)}(\mathbb{T}\times\mathbb{R})}\triangleq \alpha(x,v)\gamma(x,v)+\beta(x,v)\delta(x,v).$$
	The global existence of classical solutions to \eqref{VP eq} was discussed by Cottet-Raviart \cite{CR84} and the existence of periodic mild solutions has been studied by Bostan-Poupaud \cite{BP00}. In his thesis \cite[Thm. 2.1.1]{D87}, Dziurzynski proved that any initial datum $f_0\in L^{\infty}(\mathbb{T}\times\mathbb{R})$ with compact support satisfying \eqref{neutre cond} generates a unique global in time weak solution $f\in L^{\infty}\big([0,\infty),L^{\infty}(\mathbb{T}\times\mathbb{R})\big)$ which is Lagrangian, namely 
	$$f(t,x,v)=f_0\big(X_t^{-1}(x,v)\big),$$
	where $X_t$ is the flow map associated with the velocity field $\mathbf{v}$ and given by
	$$\partial_{t}X_{t}(x,v)=\mathbf{v}\big(t,X_{t}(x,v)\big),\qquad X_{0}(x,v)=(x,v).$$
	In particular, if we consider a bounded initial domain $\Omega_0$ and set $\Omega_t\triangleq X_t(\Omega_0)$, then
	$$f(t,x,v)=\frac{1}{|\Omega_t|}\mathbf{1}_{\Omega_t}(x,v)$$
	is a weak solution called \textit{patch of electrons.} In addition, the divergence-free condition \eqref{div=0} implies the conservation of the area, that is $|\Omega_t|=|\Omega_0|.$ We mention that patches of electrons do not physically model a concentration of electrons since the domain is in the phase space $\mathbb{T}\times\mathbb{R}.$ Patches of electrons are the kinetic equivalent of the vortex patches in fluid mechanics. We refer the reader to \cite{BM02,Y63} for a general introduction to the vortex patch dynamics. In the sequel, we shall work with a subclass of patches of electrons called \textit{electron layers} and defined as follows. Consider an initial condition with strip-shaped domain
	$$f_0(x,v)=\frac{1}{|S_0|}\mathbf{1}_{S_0}(x,v),\qquad S_0\triangleq\big\{(x,v)\in\mathbb{T}\times\mathbb{R},\quad\textnormal{s.t.}\quad v_{-}^{0}(x)< v< v_{+}^{0}(x)\big\}$$
	associated with two initial periodic profiles $v_{\pm}^{0}.$ Then, the corresponding weak solution writes
	$$f(t,x,v)=\frac{1}{|S_t|}\mathbf{1}_{S_t}(x,v),\qquad S_t\triangleq X_{t}(S_0).$$
	At later time $t>0$, the domain $S_t$ is still a velocity strip, that is one can find two periodic profiles $x\mapsto v_{\pm}(t,x)$ such that
	$$S_t=\big\{(x,v)\in\mathbb{T}\times\mathbb{R}\quad\textnormal{s.t.}\quad v_{-}(t,x)< v< v_{+}(t,x)\big\}.$$
	With these notations, the area-preserving condition writes
	\begin{equation}\label{area cond0}
		\int_{0}^{1}\big[v_{+}(t,x)-v_{-}(t,x)\big]dx=|S_t|=|S_0|=\int_{0}^{1}\big[v_{+}^{0}(x)-v_{-}^{0}(x)\big]dx.
	\end{equation}
Dziurzynski also showed in \cite{D87} the global in time persistence for the $C^1$ regularity of the boundary $\partial S_t.$ In addition, he numerically exposed possible folding formation. In this latter case, he proved possible loss of $C^3$-smoothness in finite time and excluded the formation of cusps. 
	In the next lemma, we provide a family of electron layer stationary solutions parametrized by two real numbers $a<b$ related to the geometry of the patch.
	\begin{lem}\label{lem trivial solutions flat strip}
		For any $(a,b)\in\mathbb{R}^2$ with $a<b,$ the initial profile
		\begin{equation}\label{stationary sol}
			f_0(x,v)=\frac{1}{b-a}\mathbf{1}_{S_{\textnormal{\tiny{flat}}}(a,b)}(x,v),\qquad S_{\textnormal{\tiny{flat}}}(a,b)\triangleq\mathbb{T}\times[a,b]
		\end{equation}
		generates a stationary electron layer. 
	\end{lem} 
	\begin{proof}
		Let $a<b$ and consider a function $f$ in the form
		$$f(t,x,v)=f_0(x,v)=\frac{1}{b-a}\mathbf{1}_{a\leqslant v\leqslant b}.$$
		Then one has $\partial_{t}f=\partial_{x}f=0.$ Besides, the identity \eqref{inv Lap phi} together with the structure of $f$ and the neutrality condition \eqref{neutre cond} imply $\boldsymbol{\varphi}=0.$ Thus $E=0$ and $f$ solves \eqref{VP eq}.
	\end{proof}
	\begin{remark}
		More generally, the previous proof shows that any function of the variable $v$ only is a stationary solution of \eqref{VP eq}. This is a classical result in kinetic theory.
	\end{remark}
	\subsection{Perturbative approach for the electron layer dynamics}
	The scope of this subsection is to obtain the equations of motion for a general electron layer. Then, introducing small deformations of the flat strip $S_{\textnormal{\tiny{flat}}}(a,b)$ with $a<b$, we prove that they are solutions to a system of two coupled quasilinear transport equations with linear coupling, see \eqref{system rpm}.\\

	Due to the transport structure \eqref{active eq}, the dynamics is entirely characterized by the evolution of the boundaries. We provide here the complete derivation of the contour dynamics equations following the general computations in \cite[Sec. 3.1]{HMV15} but adapted to our notations. We denote $\Gamma_0^+$ and $\Gamma_0^-$ the two boundaries of the initial strip $S_0.$ They can be seen as the zero level sets of two $C^1$ regular functions $g_0^+$ and $g_0^-$ from $\mathbb{T}\times\mathbb{R}$ into $\mathbb{R},$ namely
	$$\Gamma_0^\pm=\big\{(x,v)\in\mathbb{T}\times\mathbb{R}\quad\textnormal{s.t.}\quad g_0^{\pm}(x,v)=0\big\},\qquad\forall(x,v)\in\mathbb{T}\times\mathbb{R},\quad\nabla_{x,v}\,g_0^\pm(x,v)\neq0.$$
	We set
	\begin{equation}\label{level set funct}
		g_{\pm}(t,x,v)\triangleq g_0^\pm\big(X_t^{-1}(x,v)\big),\qquad\textnormal{i.e.}\qquad g_{\pm}\big(t,X_t(x,v)\big)\triangleq g_0^\pm(x,v).
	\end{equation}
	By construction, for any time $t$, the boundaries $\Gamma_t^{+}$ and $\Gamma_t^-$ of $S_t$ are  the zero level sets of the functions $g_{+}(t,\cdot,\cdot)$ and $g_{-}(t,\cdot,\cdot),$ respectively
	$$\Gamma_t^{\pm}\triangleq X_t\big(\Gamma_0^{\pm}\big)=\big\{(x,v)\in\mathbb{T}\times\mathbb{R}\quad\textnormal{s.t.}\quad g_{\pm}(t,x,v)=0\big\}.$$
	Differentiating in time the relation \eqref{level set funct}, we get
	\begin{align*}
		0&=\partial_{t}g_{\pm}\big(t,X_t(x,v)\big)+\Big\langle\partial_tX_t(x,v)\,,\nabla_{x,v}\,g_{\pm}\big(t,X_t(x,v)\big)\Big\rangle_{T_{(x,v)}(\mathbb{T}\times\mathbb{R})}\\
		&=\partial_{t}g_{\pm}\big(t,X_t(x,v)\big)+\Big\langle\mathbf{v}\big(t,X_t(x,v)\big)\,,\nabla_{x,v}\,g_{\pm}\big(t,X_t(x,v)\big)\Big\rangle_{T_{(x,v)}(\mathbb{T}\times\mathbb{R})}.
	\end{align*}
	Now, we consider a parametrization $z_{\pm}(t,\cdot):\mathbb{T}\rightarrow\mathbb{T}\times\mathbb{R}$ of the boundary $\Gamma_t^{\pm}$, then
	$$\partial_tg_{\pm}\big(t,z_{\pm}(t,x)\big)+\Big\langle\partial_tz_{\pm}(t,x)\,,\nabla_{x,v}\,g_{\pm}\big(t,z_{\pm}(t,x)\big)\Big\rangle_{T_{z_{\pm}(t,x)}(\mathbb{T}\times\mathbb{R})}=0.$$
	In particular,
	\begin{equation}\label{CDE VP}
		\Big\langle\partial_tz_{\pm}(t,x)-\mathbf{v}\big(t,z_{\pm}(t,x)\big)\,,\nabla_{x,v}\,g_{\pm}\big(t,z_{\pm}(t,x)\big)\Big\rangle_{T_{z_{\pm}(t,x)}(\mathbb{T}\times\mathbb{R})}=0.
	\end{equation}
	Now, by construction, the vectors $\nabla_{x,v}\,g_{\pm}\big(t,z_{\pm}(t,x)\big)$ and $\partial_{x}z_{\pm}(t,x)$ are respectively orthogonal and transversal to $\Gamma_t^{\pm}$ inside $T_{z_{\pm}(t,x)}(\mathbb{T}\times\mathbb{R}).$ Hence, we can write $$\nabla_{x,v}\,g_{\pm}\big(t,z_{\pm}(t,x)\big)=\alpha\mathtt{J}_{x,v}\partial_{x}z_{\pm}(t,x),\qquad\alpha\in\mathbb{R}.$$
	As a consequence, the identity \eqref{CDE VP} becomes
	\begin{equation}\label{cde00}
		\Big\langle\partial_tz_{\pm}(t,x)\,,\mathtt{J}_{x,v}\partial_{x}z_{\pm}(t,x)\Big\rangle_{T_{z_{\pm}(t,x)}(\mathbb{T}\times\mathbb{R})}=\Big\langle\mathbf{v}\big(t,z_{\pm}(t,x)\big)\,,\mathtt{J}_{x,v}\partial_{x}z_{\pm}(t,x)\Big\rangle_{T_{z_{\pm}(t,x)}(\mathbb{T}\times\mathbb{R})}.
	\end{equation}
	Therefore, using \eqref{vel pot} and the fact that $\mathtt{J}_{x,v}$ is orthogonal for the scalar product on the tangent plane, we obtain
	\begin{align}\label{dxpsi}
		\partial_{x}\Big(\boldsymbol{\Psi}\big(t,z_{\pm}(t,x)\big)\Big)&=\Big\langle\nabla_{x,v}\boldsymbol{\Psi}\big(t,z_{\pm}(t,x)\big)\,,\partial_{x}z_{\pm}(t,x)\Big\rangle_{T_{z_{\pm}(t,x)}(\mathbb{T}\times\mathbb{R})}\nonumber\\
		&=\Big\langle\mathtt{J}_{x,v}\nabla_{x,v}\boldsymbol{\Psi}\big(t,z_{\pm}(t,x)\big)\,,\mathtt{J}_{x,v}\partial_{x}z_{\pm}(t,x)\Big\rangle_{T_{z_{\pm}(t,x)}(\mathbb{T}\times\mathbb{R})}\nonumber\\
		&=\Big\langle\nabla_{x,v}^{\perp}\boldsymbol{\Psi}\big(t,z_{\pm}(t,x)\big)\,,\mathtt{J}_{x,v}\partial_{x}z_{\pm}(t,x)\Big\rangle_{T_{z_{\pm}(t,x)}(\mathbb{T}\times\mathbb{R})}\nonumber\\
		&=-\Big\langle\mathbf{v}\big(t,z_{\pm}(t,x)\big)\,,\mathtt{J}_{x,v}\partial_{x}z_{\pm}(t,x)\Big\rangle_{T_{z_{\pm}(t,x)}(\mathbb{T}\times\mathbb{R})}.
	\end{align}
	Combining \eqref{cde00} and \eqref{dxpsi}, we deduce the following equations
	\begin{equation}\label{vp VP}
		\Big\langle\partial_tz_{\pm}(t,x)\,,\mathtt{J}_{x,v}\partial_{x}z_{\pm}(t,x)\Big\rangle_{T_{z_{\pm}(t,x)}(\mathbb{T}\times\mathbb{R})}=-\partial_{x}\Big(\boldsymbol{\Psi}\big(t,z_{\pm}(t,x)\big)\Big).
	\end{equation}
	We fix $(a,b)\in\mathbb{R}^2$ with $a<b$. We consider an initial domain $S_{0}$ close to the flat strip $S_{\textnormal{\tiny{flat}}}(a,b)$ defined in \eqref{stationary sol} and with the same area 
	$$|S_0|=|S_{\textnormal{\tiny{flat}}}(a,b)|=b-a.$$
	We denote
	\begin{equation}\label{patch f}
		f(t,x,v)=\frac{1}{b-a}\mathbf{1}_{S_t}(x,v),\qquad S_t=\big\{(x,v)\in\mathbb{T}\times\mathbb{R}\quad\textnormal{s.t.}\quad v_{-}(t,x)< v< v_{+}(t,x)\big\}
	\end{equation}
	the corresponding electron layer weak solution of \eqref{VP eq}. The area condition \eqref{area cond0} writes in this context
	\begin{equation}\label{area cond}
		\int_{0}^{1}\big[v_{+}(t,x)-v_{-}(t,x)\big]dx=b-a.
	\end{equation}
	We take as an ansatz
	\begin{equation}\label{zpm vpm}
		\begin{cases}
			z_{+}(t,x)=\big(x,b+ r_{+}(t,x)\big),\qquad\textnormal{i.e.}\qquad v_{+}(t,x)=b+ r_{+}(t,x),\\
			z_{-}(t,x)=\big(x,a+ r_{-}(t,x)\big),\qquad\textnormal{i.e.}\qquad v_{-}(t,x)=a+r_{-}(t,x).
		\end{cases}
	\end{equation}
	
	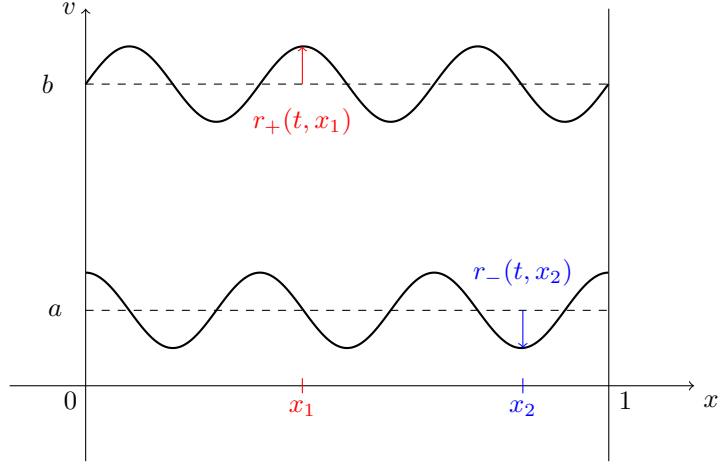
\begin{figure}[!h]
		\begin{center}
			\begin{tikzpicture}
				\draw[->](-1,0)--(8,0)
				node[below right] {$x$};
				\draw[->](0,-1)--(0,5)
				node[left] {$v$};
				\draw[black,dashed] (0,1)--(6.88,1);
				\draw[black,dashed] (0,4)--(6.88,4);
				\draw[black] (6.88,-1)--(6.88,5);
				\node at (-0.4,1) {$a$};
				\node at (-0.5,4) {$b$};
				\node at (-0.2,-0.2) {$0$};
				\node at (7.1,-0.2) {$1$};
				\draw[domain=0:6.88,thick, black,samples=500] plot [variable=\t] (\t,{4+0.5*sin(50*pi*\t)});
				\draw[domain=0:6.88,thick, black,samples=500] plot [variable=\t] (\t,{1+0.5*cos(50*pi*\t)});
				\draw[red] (2.85,0.1)--(2.85,-0.1);
				\node at (2.85,-0.3) {$\textcolor{red}{x_1}$};
				\node at (2.85,3.5) {$\textcolor{red}{r_+(t,x_1)}$};
				\draw[->,red](2.85,4)--(2.85,4.5);
				\draw[blue] (5.75,0.1)--(5.75,-0.1);
				\node at (5.75,-0.3) {$\textcolor{blue}{x_2}$};
				\node at (5.75,1.5) {$\textcolor{blue}{r_-(t,x_2)}$};
				\draw[->,blue](5.75,1)--(5.75,0.5);
			\end{tikzpicture}
		\end{center}
		\caption{Perturbation of the flat strip $S_{\textnormal{\tiny{flat}}}(a,b)$.}
	\end{figure}
	
	\noindent Observe that the area conservation condition \eqref{area cond} writes 
	\begin{equation}\label{area cond-2}
		\int_{0}^{1}\big[r_{+}(t,x)-r_{-}(t,x)\big]dx=0.
	\end{equation} 
	Now, on one hand
	\begin{align}\label{left hs}
		\Big\langle\partial_tz_{\pm}(t,x)\,,\mathtt{J}_{x,v}\partial_{x}z_{\pm}(t,x)\Big\rangle_{T_{z_{\pm}(t,x)}(\mathbb{T}\times\mathbb{R})}&=\Big\langle\partial_tr_{\pm}(t,x)\mathtt{e}_{v}\,,\mathtt{J}_{x,v}\big(\mathtt{e}_{x}+\partial_{x}r_{\pm}\mathtt{e}_{v}\big)\Big\rangle_{T_{z_{\pm}(t,x)}(\mathbb{T}\times\mathbb{R})}\nonumber\\
		&=\Big\langle\partial_tr_{\pm}(t,x)\mathtt{e}_{v}\,,\mathtt{e}_{v}-\partial_{x}r_{\pm}\mathtt{e}_{x}\Big\rangle_{T_{z_{\pm}(t,x)}(\mathbb{T}\times\mathbb{R})}\nonumber\\
		&=\partial_{t}r_{\pm}(t,x).
	\end{align}
	On the other hand, using \eqref{vel pot},
	\begin{equation}\label{Psi od zpm}
		\begin{cases}
			\boldsymbol{\Psi}\big(t,z_{+}(t,x)\big)=\frac{1}{2}\big(b+ r_{+}(t,x)\big)^2+\boldsymbol{\varphi}(t,x),\\
			\boldsymbol{\Psi}\big(t,z_{-}(t,x)\big)=\frac{1}{2}\big(a+ r_{-}(t,x)\big)^2+\boldsymbol{\varphi}(t,x).
		\end{cases}
	\end{equation}
	In addition, from \eqref{inv Lap phi} and \eqref{patch f}, we can write
	\begin{align}\label{boldvarphi}
		\boldsymbol{\varphi}(t,x)&=\partial_{xx}^{-1}\left(1-\int_{v_{-}(t,x)}^{v_{+}(t,x)}\frac{1}{b-a}dv\right)\nonumber\\
		&=\partial_{xx}^{-1}\Big(1-\tfrac{1}{b-a}\big(v_{+}(t,x)-v_{-}(t,x)\big)\Big)\nonumber\\
		&=-\tfrac{1}{b-a}\partial_{xx}^{-1}\big(r_{+}(t,x)-r_{-}(t,x)\big).
	\end{align}
	Inserting \eqref{left hs}, \eqref{Psi od zpm} and \eqref{boldvarphi} into \eqref{vp VP}, we end up with the following system
	\begin{equation}\label{syst}
		\begin{cases}
			\partial_{t}r_{+}(t,x)=-\partial_{x}\Big(\tfrac{1}{2}\big( b+r_{+}(t,x)\big)^2-\tfrac{1}{b-a}\partial_{xx}^{-1}r_{+}(t,x)+\tfrac{1}{b-a}\partial_{xx}^{-1}r_{-}(t,x)\Big),\vspace{0.2cm}\\
			\partial_{t}r_{-}(t,x)=-\partial_{x}\Big(\tfrac{1}{2}\big(a+ r_{-}(t,x)\big)^2-\tfrac{1}{b-a}\partial_{xx}^{-1}r_{+}(t,x)+\tfrac{1}{b-a}\partial_{xx}^{-1}r_{-}(t,x)\Big),
		\end{cases}
	\end{equation}
	which can recast in the following form
	\begin{equation}\label{system rpm}
		\begin{cases}
			\partial_{t}r_{+}(t,x)+\big( r_{+}(t,x)+b\big)\partial_{x}r_+(t,x)-\tfrac{1}{b-a}\partial_{x}^{-1}r_{+}(t,x)+\tfrac{1}{b-a}\partial_{x}^{-1}r_{-}(t,x)=0,\vspace{0.2cm}\\
			\partial_{t}r_{-}(t,x)+\big( r_{-}(t,x)+a\big)\partial_{x}r_-(t,x)-\tfrac{1}{b-a}\partial_{x}^{-1}r_{+}(t,x)+\tfrac{1}{b-a}\partial_{x}^{-1}r_{-}(t,x)=0,
		\end{cases}
	\end{equation}
where
$$\partial_{x}^{-1}\mathbf{e}_{j}\triangleq\frac{\mathbf{e}_j}{2\pi\ii j},$$
or equivalently, in real notations,
\begin{equation}\label{inv dx}
	\forall j\in\mathbb{N}^*,\qquad\partial_{x}^{-1}\cos(2\pi jx)\triangleq\frac{\sin(2\pi jx)}{2\pi j},\qquad\partial_{x}^{-1}\sin(2\pi jx)\triangleq-\frac{\cos(2\pi jx)}{2\pi j}\cdot
\end{equation}
	This is a system of two coupled quasilinear transport equations where the coupling is linear. Notice that if one tries to impose the constraint $r_+=r_-$, compatible with \eqref{area cond-2}, in order tp reduce the study to a scalar equation, then one finds only the trivial solution $r_+=r_-=\textnormal{Cte}$, i.e. the flat strip, as a solution. Remark that \eqref{syst} implies
	$$\partial_{t}\int_{0}^{1}r_{+}(t,x)dx=\partial_{t}\int_{0}^{1}r_{-}(t,x)dx=0.$$
	Then, we can impose
	\begin{equation}\label{zero average}
		\int_{0}^{1}r_{+}(t,x)dx=\int_{0}^{1}r_{-}(t,x)dx=0,
	\end{equation}
	which is compatible with \eqref{area cond-2}.

	\subsection{Main results}
	Here, we expose our main results. Our inspiration naturally comes from the fluid mechanics where uniformly rotating vortex patch solutions were obtained for various models, see \cite{B82,CCG16,DHR19,GHR23,GHM22,GHM22-1,HH15,HHH16,HHHM15,HMW20,HM16,HM16-2,HXX22,R17,R21,R22}. In the planar case, such solutions are called \textit{V-states} (for "Vortex states") according to the terminology introduced by Deem and Zabusky \cite{DZ78}. The second inspiration is borrowed (for instance) to the bifurcation of traveling waves for water-waves \cite{CSV16,LC25,N21,S47,S26}. In honor of Deem and Zabusky terminology, we give the following definition.
	\begin{defin}\label{def E-states}\textbf{(E-states or Electron-states)}
		Let $c\in\mathbb{R}$ and $\mathbf{m}\in\mathbb{N}^*.$ We say that an electron layer solution to \eqref{VP eq} is
		\begin{enumerate}[label=\textbullet]
			\item \textit{$\mathbf{m}$-symmetric} if
			$$v_{\pm}\big(t,x+\tfrac{1}{\mathbf{m}}\big)=v_{\pm}(t,x),\qquad\textnormal{i.e.}\qquad r_{\pm}\big(t,x+\tfrac{1}{\mathbf{m}}\big)=r_{\pm}(t,x).$$
			\item an \textit{E(lectron)-state with velocity $c$} if there exist 1-periodic profiles $\check{v}_{\pm}$, i.e. $\check{r}_{\pm}$, such that
			$$v_{\pm}(t,x)=\check{v}_{\pm}(x-ct),\qquad\textnormal{i.e.}\qquad r_{\pm}(t,x)=\check{r}_{\pm}(x-ct).$$
			The \textit{$\mathbf{m}$-symmetry} condition for an E-state becomes 
			$$\check{v}_{\pm}\big(x+\tfrac{1}{\mathbf{m}}\big)=\check{v}_{\pm}(x),\qquad\textnormal{i.e.}\qquad \check{r}_{\pm}\big(x+\tfrac{1}{\mathbf{m}}\big)=\check{r}_{\pm}(x).$$
		\end{enumerate}
	\end{defin}
	Observe that the E-states are traveling periodic solutions for which the boundaries look stationary in a moving frame in translation with speed $c$ in the $x$-direction. In what follows, we prove the emergence of E-states with analytic boundary. In particular, they do not present folding formation so there is no interaction with Dziurzynski's results. Moreover, for an electron-state, the electric field is time periodic. Since we work with non-smooth solutions close to non-smooth equilibria (patches), there is also no contradiction with the classical theory of Landau damping \cite{GI22,MV11}. We shall now state our main theorem.
	\begin{theo}\label{thm E-states}\textbf{(Local bifurcation of E-states)}\\
		The one dimensonal Vlasov-Poisson equation \eqref{VP eq} admits the following implicit solutions.
		\begin{enumerate}[label=(\roman*)]
			\item Let $a<b$ and $\mathbf{m}\in\mathbb{N}^*.$ There exist two local curves 
			$$\mathscr{C}_{\textnormal{\tiny{local}}}^{\pm,\mathbf{m}}(a,b)\triangleq\Big\{\left(c_{\mathbf{m}}^{\pm}(\mathtt{s},a,b),\check{r}_{\mathbf{m}}^{\pm}(\mathtt{s},a,b)\right),\quad|\mathtt{s}|<\delta\Big\},\qquad\delta>0$$ 
			corresponding to $\mathbf{m}$-symmetric E-states bifurcating from the flat strip $S_{\textnormal{\tiny{flat}}}(a,b)$ defined in \eqref{stationary sol} and admitting the expansion
			$$c_{\mathbf{m}}^{\pm}(\mathtt{s},a,b)\underset{\mathtt{s}\to0}{=}c_{\mathbf{m}}^{\pm}(a,b)+\mathtt{s}^2c_{\mathbf{m},2}^{\pm}(a,b)+O(\mathtt{s}^3),$$
			with
			$$c_{\mathbf{m}}^{\pm}(a,b)\triangleq\frac{a+b}{2}\pm\sqrt{\frac{\pi^2\mathbf{m}^2(b-a)^2+1}{4\pi^2\mathbf{m}^2}},\qquad c_{\mathbf{m},2}^{+}(a,b)>0,\qquad c_{\mathbf{m},2}^{-}(a,b)<0$$
			and
			$$\check{r}_{\mathbf{m}}^{\pm}(\mathtt{s},a,b)(x)\underset{\mathtt{s}\to0}{=}\mathtt{s}\begin{pmatrix}
				2\pi\mathbf{m}\big(a-c_{\mathbf{m}}^{\pm}(a,b)\big)-\tfrac{1}{2\pi\mathbf{m}(b-a)}\\
				-\tfrac{1}{2\pi\mathbf{m}(b-a)}
			\end{pmatrix}\cos(2\pi\mathbf{m}x)+O(\mathtt{s}^2).$$
			Both bifurcations are of pitchfork-type and the bifurcation diagram admits (locally close to the trivial line) a "hyperbolic" structure as represented in the following figure
			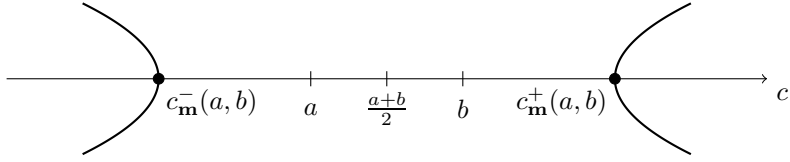
\begin{figure}[!h]
				\begin{center}
					\begin{tikzpicture}
						\draw[->] (-5,0)--(5,0)
						node[below right] {$c$};
						\draw[black] (0,0.1)--(0,-0.1);
						\node at (0,-0.4) {$\frac{a+b}{2}$};
						\draw[black] (-1,0.1)--(-1,-0.1);
						\node at (-1,-0.4) {$a$};
						\draw[black] (1,0.1)--(1,-0.1);
						\node at (1,-0.4) {$b$};
						\node at (-2.3,-0.3) {$c_\mathbf{m}^{-}(a,b)$};
						\node at (2.3,-0.3) {$c_\mathbf{m}^{+}(a,b)$};
						\filldraw [black] (-3,0) circle (2pt);
						\filldraw [black] (3,0) circle (2pt);
						\draw[domain=3:4,thick, black,samples=500] plot [variable=\t] (\t,{sqrt(\t-3)});
						\draw[domain=3:4,thick, black,samples=500] plot [variable=\t] (\t,{-sqrt(\t-3)});
						\draw[domain=-4:-3,thick, black,samples=500] plot [variable=\t] (\t,{sqrt(abs(3+\t))});
						\draw[domain=-4:-3,thick, black,samples=500] plot [variable=\t] (\t,{-sqrt(abs(3+\t))});
					\end{tikzpicture}
				\end{center}
				\caption{Representation of the velocity bifurcation diagram with "hyperbolic" structure.}\label{fig hyperbolic}
			\end{figure}
		\item Let $(a,c)\in\mathbb{R}^2.$
		We denote, for any $p\in\mathbb{R}^*$,
		$$N_1(p)\triangleq1+N_2(p),\qquad N_2(p)\triangleq \big\lfloor\tfrac{1}{2\pi|p|}\big\rfloor\cdot$$
		\begin{enumerate}[label=(\alph*)]
			\item Assume $a<c.$ Then, for any $\mathbf{m}\in\mathbb{N}^*$ with $\mathbf{m}\geqslant N_1(c-a)$, there exists a local curve
			$$\mathscr{C}_{\textnormal{\tiny{local}}}^{\mathbf{m}}(a,c)\triangleq\Big\{\big(b_{\mathbf{m}}(\mathtt{s},a,c),\check{r}_{\mathbf{m}}(\mathtt{s},a,c)\big),\quad|\mathtt{s}|<\delta\Big\},\qquad\delta>0$$
			corresponding to $\mathbf{m}$-symmetric E-states with velocity $c$ bifurcating from the flat strip $S_{\textnormal{\tiny{flat}}}\big(a,b_{\mathbf{m}}(a,c)\big)$ and admitting the following expansion
			$$b_{\mathbf{m}}(\mathtt{s},a,c)\underset{\mathtt{s}\to0}{=}b_{\mathbf{m}}(a,c)+\mathtt{s}^2b_{\mathbf{m},2}(a,c)+O(\mathtt{s}^3),$$
			with
			$$b_{\mathbf{m}}(a,c)\triangleq c+\frac{1}{4\pi^2\mathbf{m}^2(a-c)},\qquad b_{\mathbf{m},2}(a,c)<0\textnormal{ (subcritical bifurcation)}$$
			and
			$$\check{r}_{\mathbf{m}}(\mathtt{s},a,c)(x)\underset{\mathtt{s}\to0}{=}\mathtt{s}\begin{pmatrix}
				2\pi\mathbf{m}(a-c)-\tfrac{1}{2\pi\mathbf{m}\big(b_{\mathbf{m}}(a,c)-a\big)}\\
				-\tfrac{1}{2\pi\mathbf{m}\big(b_{\mathbf{m}}(a,c)-a\big)}
			\end{pmatrix}\cos(2\pi\mathbf{m}x)+O(\mathtt{s}^2).$$
			\item Assume $c<a<c+\tfrac{1}{2\pi}.$ Then, for any $\mathbf{m}\in\big\llbracket 1,N_2(a-c)\big\rrbracket$, there exists a local curve $\mathscr{C}_{\textnormal{\tiny{local}}}^{\mathbf{m}}(a,c)$ of $\mathbf{m}$-symmetric E-states with velocity $c$ bifucating from the flat strip $S_{\textnormal{\tiny{flat}}}\big(a,b_{\mathbf{m}}(a,c)\big)$ as above but with $b_{\mathbf{m},2}(a,c)>0$ (supercritical bifurcation).
		\end{enumerate}
		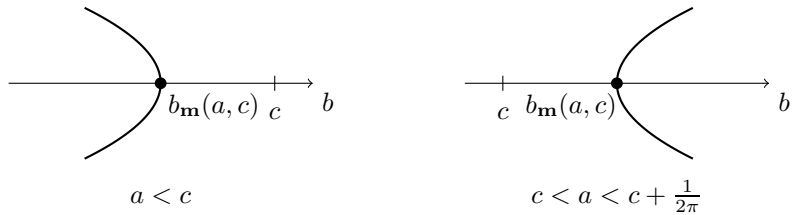
\begin{figure}[!h]
		\begin{center}
			\begin{tikzpicture}
				\draw[->] (-5,0)--(-1,0)
				node[below right] {$b$};
				\draw[->] (1,0)--(5,0)
				node[below right] {$b$};
				\draw[black] (-1.5,0.1)--(-1.5,-0.1);
				\node at (-1.5,-0.4) {$c$};
				\draw[black] (1.5,0.1)--(1.5,-0.1);
				\node at (1.5,-0.4) {$c$};
				\node at (-2.3,-0.3) {$b_\mathbf{m}(a,c)$};
				\node at (-3,-1.5) {$a<c$};
				\node at (2.4,-0.3) {$b_\mathbf{m}(a,c)$};
				\node at (3,-1.5) {$c<a<c+\tfrac{1}{2\pi}$};
				\filldraw [black] (-3,0) circle (2pt);
				\filldraw [black] (3,0) circle (2pt);
				\draw[domain=3:4,thick, black,samples=500] plot [variable=\t] (\t,{sqrt(\t-3)});
				\draw[domain=3:4,thick, black,samples=500] plot [variable=\t] (\t,{-sqrt(\t-3)});
				\draw[domain=-4:-3,thick, black,samples=500] plot [variable=\t] (\t,{sqrt(abs(3+\t))});
				\draw[domain=-4:-3,thick, black,samples=500] plot [variable=\t] (\t,{-sqrt(abs(3+\t))});
			\end{tikzpicture}
		\end{center}
	\caption{Representation of $b$-local bifurcation diagram close to asymmetric flat strips.}
	\end{figure}
	\item Let $(b,c)\in\mathbb{R}^2.$
	\begin{enumerate}[label=(\alph*)]
		\item Assume $b>c.$ Then, for any $\mathbf{m}\in\mathbb{N}^*$ with $\mathbf{m}\geqslant N_1(b-c)$, there exists a local curve
		$$\mathscr{C}_{\textnormal{\tiny{local}}}^{\mathbf{m}}(b,c)\triangleq\Big\{\big(a_{\mathbf{m}}(\mathtt{s},b,c),\check{r}_{\mathbf{m}}(\mathtt{s},b,c)\big),\quad|\mathtt{s}|<\delta\Big\},\qquad\delta>0$$
		corresponding to $\mathbf{m}$-symmetric E-states with velocity $c$ bifurcating from the flat strip $S_{\textnormal{\tiny{flat}}}\big(a_{\mathbf{m}}(b,c),b\big)$ and admitting the following expansion
		$$a_{\mathbf{m}}(\mathtt{s},b,c)\underset{\mathtt{s}\to0}{=}a_{\mathbf{m}}(b,c)+\mathtt{s}^2a_{\mathbf{m},2}(b,c)+O(\mathtt{s}^3),$$
		with
		$$a_{\mathbf{m}}(b,c)\triangleq c+\frac{1}{4\pi^2\mathbf{m}^2(b-c)},\qquad a_{\mathbf{m},2}(b,c)>0\textnormal{ (supercritical bifurcation)}$$
		and
		$$\check{r}_{\mathbf{m}}(\mathtt{s},b,c)(x)\underset{\mathtt{s}\to0}{=}\mathtt{s}\begin{pmatrix}
			2\pi\mathbf{m}\big(a_{\mathbf{m}}(b,c)-c\big)-\tfrac{1}{2\pi\mathbf{m}\big(b-a_{\mathbf{m}}(b,c)\big)}\\
			-\tfrac{1}{2\pi\mathbf{m}\big(b-a_{\mathbf{m}}(b,c)\big)}
		\end{pmatrix}\cos(2\pi\mathbf{m}x)+O(\mathtt{s}^2).$$
		\item Assume $c-\tfrac{1}{2\pi}<b<c.$ Then, for any $\mathbf{m}\in\big\llbracket 1,N_2(c-b)\big\rrbracket$, there exists a local curve $\mathscr{C}_{\textnormal{\tiny{local}}}^{\mathbf{m}}(b,c)$ of $\mathbf{m}$-symmetric E-states with velocity $c$ bifurcating from the flat strip $S_{\textnormal{\tiny{flat}}}\big(a_{\mathbf{m}}(b,c),b\big)$ as above but with $a_{\mathbf{m},2}(b,c)<0$ (subcritical bifurcation).
	\end{enumerate}
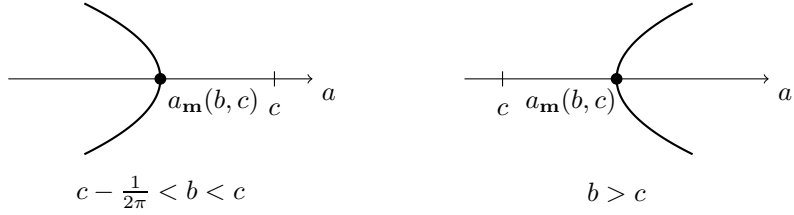
\begin{figure}[!h]
	\begin{center}
		\begin{tikzpicture}
			\draw[->] (-5,0)--(-1,0)
			node[below right] {$a$};
			\draw[->] (1,0)--(5,0)
			node[below right] {$a$};
			\draw[black] (-1.5,0.1)--(-1.5,-0.1);
			\node at (-1.5,-0.4) {$c$};
			\draw[black] (1.5,0.1)--(1.5,-0.1);
			\node at (1.5,-0.4) {$c$};
			\node at (-2.3,-0.3) {$a_\mathbf{m}(b,c)$};
			\node at (-3,-1.5) {$c-\tfrac{1}{2\pi}<b<c$};
			\node at (2.4,-0.3) {$a_\mathbf{m}(b,c)$};
			\node at (3,-1.5) {$b>c$};
			\filldraw [black] (-3,0) circle (2pt);
			\filldraw [black] (3,0) circle (2pt);
			\draw[domain=3:4,thick, black,samples=500] plot [variable=\t] (\t,{sqrt(\t-3)});
			\draw[domain=3:4,thick, black,samples=500] plot [variable=\t] (\t,{-sqrt(\t-3)});
			\draw[domain=-4:-3,thick, black,samples=500] plot [variable=\t] (\t,{sqrt(abs(3+\t))});
			\draw[domain=-4:-3,thick, black,samples=500] plot [variable=\t] (\t,{-sqrt(abs(3+\t))});
		\end{tikzpicture}
	\end{center}
\caption{Representation of $a$-local bifurcation diagram close to asymmetric flat strips.}
\end{figure}
			\item Let $c\in\mathbb{R}^*.$ Then, for any $\mathbf{m}\in\mathbb{N}^*$ with $\mathbf{m}\geqslant N_1(c),$ there exists a local curve
			$$\mathscr{C}_{\textnormal{\tiny{local}}}^{\mathbf{m}}(c)\triangleq\Big\{\big(a_{\mathbf{m}}(\mathtt{s},c),\check{r}_{\mathbf{m}}(\mathtt{s},c)\big),\quad|\mathtt{s}|<\delta\Big\},\qquad\delta>0$$
			corresponding to $\mathbf{m}$-symmetric E-states with velocity $c$ bifurcating from the symmetric flat strip\\ $S_{\textnormal{\tiny{flat}}}\big(-a_{\mathbf{m}}(c),a_{\mathbf{m}}(c)\big)$ and admitting the following expansion
			$$a_{\mathbf{m}}(\mathtt{s},c)\underset{\mathtt{s}\to0}{=}a_{\mathbf{m}}(c)+\mathtt{s}^2a_{\mathbf{m},2}(c)+O(\mathtt{s}^3),$$
			with
			$$a_{\mathbf{m}}(c)\triangleq\sqrt{\frac{4\pi^{2}\mathbf{m}^2c^2-1}{4\pi^2\mathbf{m}^2}},\qquad \begin{cases}
				a_{\mathbf{m},2}(c)<0, & \textnormal{if }c>0,\\
				a_{\mathbf{m},2}(c)>0, & \textnormal{if }c<0
			\end{cases}$$
			and
			$$\check{r}_{\mathbf{m}}(\mathtt{s},c)(x)\underset{\mathtt{s}\to0}{=}\mathtt{s}\begin{pmatrix}
				2\pi\mathbf{m}\big(a_{\mathbf{m}}(c)-c\big)-\tfrac{1}{4\pi\mathbf{m}a_{\mathbf{m}}(c)}\\
				-\tfrac{1}{4\pi\mathbf{m}a_{\mathbf{m}}(c)}
			\end{pmatrix}\cos(2\pi\mathbf{m}x)+O(\mathtt{s}^2).$$
			\begin{figure}[!h]
				\begin{center}
					\begin{tikzpicture}
						\draw[->] (-5,0)--(-1,0)
						node[below right] {$a$};
						\draw[->] (1,0)--(5,0)
						node[below right] {$a$};
						\draw[black] (-4.5,0.1)--(-4.5,-0.1);
						\node at (-4.5,-0.4) {$0$};
						\draw[black] (1.5,0.1)--(1.5,-0.1);
						\node at (1.5,-0.4) {$0$};
						\node at (-2.3,-0.3) {$a_\mathbf{m}(c)$};
						\node at (-3,-1.5) {$c>0$};
						\node at (2.4,-0.3) {$a_\mathbf{m}(c)$};
						\node at (3,-1.5) {$c<0$};
						\filldraw [black] (-3,0) circle (2pt);
						\filldraw [black] (3,0) circle (2pt);
						\draw[domain=3:4,thick, black,samples=500] plot [variable=\t] (\t,{sqrt(\t-3)});
						\draw[domain=3:4,thick, black,samples=500] plot [variable=\t] (\t,{-sqrt(\t-3)});
						\draw[domain=-4:-3,thick, black,samples=500] plot [variable=\t] (\t,{sqrt(abs(3+\t))});
						\draw[domain=-4:-3,thick, black,samples=500] plot [variable=\t] (\t,{-sqrt(abs(3+\t))});
					\end{tikzpicture}
				\end{center}
				\caption{Representation of the area local bifurcation diagram close to symmetric flat strips.}
			\end{figure}
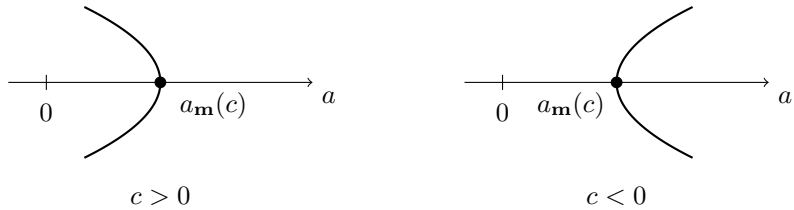
		\end{enumerate}
		In addition, each one of the previous bifurcations occurs at any level of Sobolev-analytic regularity $H^{s,\sigma}$ for $\sigma>0$ and  $s\geqslant1.$
	\end{theo}
	
To prove Theorem \ref{thm E-states} we use Crandall-Rabinowitz-Shi's Theorem \ref{thm CR+S} for the construction of the local curves and study the pitchfork phenomenon. Then, we apply Buffoni-Toland Theorem \ref{thm BT} to globally extend the branches and get the following result.

\begin{theo}\label{thm global E-states}\textbf{(Global bifurcation of E-states)}
	All the bifurcations of Theorem \ref{thm E-states} are global in $H^{s,\sigma}$ for $s>\tfrac{3}{2}$ and $\sigma>0.$ More precisely, 
	\begin{enumerate}[label=(\roman*)]
		\item Let $a<b$ and $\mathbf{m}\in\mathbb{N}^*$. Fix $\kappa\in\{+,-\}.$ Then, there exist two global curves
		$$\mathscr{C}_{\textnormal{\tiny{global}}}^{\kappa,\mathbf{m}}(a,b)\triangleq\Big\{\left(c_{\mathbf{m}}^{\kappa}(\mathtt{s},a,b),\check{r}_{\mathbf{m}}^{\kappa}(\mathtt{s},a,b)\right),\quad\mathtt{s}\in\mathbb{R}\Big\}$$
		corresponding to $\mathbf{m}$-symmetric E-states and extending the local curves $\mathscr{C}_{\textnormal{\tiny{local}}}^{\kappa,\mathbf{m}}(a,b)$ given by Theorem \ref{thm E-states}-(i). Moreover, the curve $\mathscr{C}_{\textnormal{\tiny{global}}}^{\kappa,\mathbf{m}}(a,b)$ admits locally around each of its points a real-analytic reparametrization. In addition, one has the following alternatives
		\begin{itemize}
			\item [$(A1)$] There exist $T_{\mathbf{m}}^{\kappa}(a,b)>0$ such that
			$$\forall\mathtt{s}\in\mathbb{R},\quad c_{\mathbf{m}}^{\kappa}\big(\mathtt{s}+T_{\mathbf{m}}^{\kappa}(a,b),a,b\big)=c_{\mathbf{m}}^{\kappa}(\mathtt{s},a,b)\qquad\textnormal{and}\qquad \check{r}_{\mathbf{m}}^{\kappa}\big(\mathtt{s}+T_{\mathbf{m}}^{\kappa}(a,b),a,b\big)=\check{r}_{\mathbf{m}}^{\kappa}(\mathtt{s},a,b).$$
			\item [$(A2)$] One of the following limits occurs (possibly simultaneously)
			\begin{enumerate}[label=\textbullet]
				\item (Blow-up) $\displaystyle\lim_{\mathtt{s}\to\pm\infty}\frac{1}{1+\left|c_{\mathbf{m}}^{\kappa}(\mathtt{s},a,b)\right|+\|\check{r}_{\mathbf{m}}^{\kappa}(\mathtt{s},a,b)\|_{s,\sigma}}=0.$
				\item (Collision of the boundaries) $\displaystyle\lim_{\mathtt{s}\to\pm\infty}\min_{x\in\mathbb{T}}\left|\big(\check{r}_{\mathbf{m}}^{\kappa}\big)_{+}(\mathtt{s},a,b)(x)-\big(\check{r}_{\mathbf{m}}^{\kappa}\big)_{-}(\mathtt{s},a,b)(x)+b-a\right|=0.$
				\item (Degeneracy $+$) $\displaystyle\lim_{\mathtt{s}\to\pm\infty}\min_{x\in\mathbb{T}}\left|\big(\check{r}_{\mathbf{m}}^{\kappa}\big)_{+}(\mathtt{s},a,b)(x)+b-c_{\mathbf{m}}^{\pm}(\mathtt{s},a,b)\right|=0.$
				\item (Degeneracy $-$) $\displaystyle\lim_{\mathtt{s}\to\pm\infty}\min_{x\in\mathbb{T}}\left|\big(\check{r}_{\mathbf{m}}^{\kappa}\big)_{-}(\mathtt{s},a,b)(x)+a-c_{\mathbf{m}}^{\pm}(\mathtt{s},a,b)\right|=0.$
			\end{enumerate}
		\end{itemize}
	\item Let $(a,c)\in\mathbb{R}^2$,  $\mathbf{m}\in\mathbb{N}^*$ with $a<c$ and $\mathbf{m}\geqslant N_1(c-a)$ (resp. $c<a<c+\tfrac{1}{2\pi}$ and $\mathbf{m}\in\llbracket 1, N_2(a-c)\rrbracket$). Then, there exists a global curve
	$$\mathscr{C}_{\textnormal{\tiny{global}}}^{\mathbf{m}}(a,c)\triangleq\Big\{\big(b_{\mathbf{m}}(\mathtt{s},a,c),\check{r}_{\mathbf{m}}(\mathtt{s},a,c)\big),\quad\mathtt{s}\in\mathbb{R}\Big\}$$
	corresponding to $\mathbf{m}$-symmetric E-states and extending the local curves $\mathscr{C}_{\textnormal{\tiny{local}}}^{\mathbf{m}}(a,c)$ given by Theorem \ref{thm E-states}-(ii). Moreover, the curve $\mathscr{C}_{\textnormal{\tiny{global}}}^{\mathbf{m}}(a,c)$ admits locally around each of its points a real-analytic reparametrization. In addition, one has the following alternatives
	\begin{itemize}
		\item [$(A1)$] There exist $T_{\mathbf{m}}(a,c)>0$ such that
		$$\forall\mathtt{s}\in\mathbb{R},\quad b_{\mathbf{m}}\big(\mathtt{s}+T_{\mathbf{m}}(a,c),a,c\big)=b_{\mathbf{m}}(\mathtt{s},a,c)\qquad\textnormal{and}\qquad \check{r}_{\mathbf{m}}\big(\mathtt{s}+T_{\mathbf{m}}(a,c),a,c\big)=\check{r}_{\mathbf{m}}(\mathtt{s},a,c).$$
		\item [$(A2)$] One of the following limits occurs (possibly simultaneously)
		\begin{enumerate}[label=\textbullet]
			\item (Blow-up) $\displaystyle\lim_{\mathtt{s}\to\pm\infty}\frac{1}{1+\left|b_{\mathbf{m}}(\mathtt{s},a,c)\right|+\|\check{r}_{\mathbf{m}}(\mathtt{s},a,c)\|_{s,\sigma}}=0.$
			\item (Collision of the boundaries) $\displaystyle\lim_{\mathtt{s}\to\pm\infty}\min_{x\in\mathbb{T}}\left|\big(\check{r}_{\mathbf{m}}\big)_{+}(\mathtt{s},a,c)(x)-\big(\check{r}_{\mathbf{m}}\big)_{-}(\mathtt{s},a,c)(x)+b_{\mathbf{m}}(\mathtt{s},a,c)-a\right|=~0.$
			\item (Degeneracy $+$) $\displaystyle\lim_{\mathtt{s}\to\pm\infty}\min_{x\in\mathbb{T}}\left|\big(\check{r}_{\mathbf{m}}\big)_{+}(\mathtt{s},a,c)(x)+b_{\mathbf{m}}(\mathtt{s},a,c)-c\right|=0.$
			\item (Degeneracy $-$) $\displaystyle\lim_{\mathtt{s}\to\pm\infty}\min_{x\in\mathbb{T}}\left|\big(\check{r}_{\mathbf{m}}\big)_{-}(\mathtt{s},a,c)(x)+a-c\right|=0.$
		\end{enumerate}
	\end{itemize}
	\item Let $(b,c)\in\mathbb{R}^2$,  $\mathbf{m}\in\mathbb{N}^*$ with $b>c$ and $\mathbf{m}\geqslant N_1(b-c)$ (resp. $c-\tfrac{1}{2\pi}<b<c$ and $\mathbf{m}\in\llbracket 1, N_2(c-b)\rrbracket$). Then, there exists a global curve
$$\mathscr{C}_{\textnormal{\tiny{global}}}^{\mathbf{m}}(b,c)\triangleq\Big\{\big(a_{\mathbf{m}}(\mathtt{s},b,c),\check{r}_{\mathbf{m}}(\mathtt{s},b,c)\big),\quad\mathtt{s}\in\mathbb{R}\Big\}$$
corresponding to $\mathbf{m}$-symmetric E-states and extending the local curves $\mathscr{C}_{\textnormal{\tiny{local}}}^{\mathbf{m}}(b,c)$ given by Theorem \ref{thm E-states}-(iii). Moreover, the curve $\mathscr{C}_{\textnormal{\tiny{global}}}^{\mathbf{m}}(b,c)$ admits locally around each of its points a real-analytic reparametrization. In addition, one has the following alternatives
\begin{itemize}
	\item [$(A1)$] There exist $T_{\mathbf{m}}(b,c)>0$ such that
	$$\forall\mathtt{s}\in\mathbb{R},\quad a_{\mathbf{m}}\big(\mathtt{s}+T_{\mathbf{m}}(b,c),b,c\big)=a_{\mathbf{m}}(\mathtt{s},b,c)\qquad\textnormal{and}\qquad \check{r}_{\mathbf{m}}\big(\mathtt{s}+T_{\mathbf{m}}(b,c),b,c\big)=\check{r}_{\mathbf{m}}(\mathtt{s},b,c).$$
	\item [$(A2)$] One of the following limits occurs (possibly simultaneously)
	\begin{enumerate}[label=\textbullet]
		\item (Blow-up) $\displaystyle\lim_{\mathtt{s}\to\pm\infty}\frac{1}{1+\left|a_{\mathbf{m}}(\mathtt{s},b,c)\right|+\|\check{r}_{\mathbf{m}}(\mathtt{s},b,c)\|_{s,\sigma}}=0.$
		\item (Collision of the boundaries) $\displaystyle\lim_{\mathtt{s}\to\pm\infty}\min_{x\in\mathbb{T}}\left|\big(\check{r}_{\mathbf{m}}\big)_{+}(\mathtt{s},b,c)(x)-\big(\check{r}_{\mathbf{m}}\big)_{-}(\mathtt{s},b,c)(x)+b-a_{\mathbf{m}}(\mathtt{s},b,c)\right|=~0.$
		\item (Degeneracy $+$) $\displaystyle\lim_{\mathtt{s}\to\pm\infty}\min_{x\in\mathbb{T}}\left|\big(\check{r}_{\mathbf{m}}\big)_{+}(\mathtt{s},b,c)(x)+b-c\right|=0.$
		\item (Degeneracy $-$) $\displaystyle\lim_{\mathtt{s}\to\pm\infty}\min_{x\in\mathbb{T}}\left|\big(\check{r}_{\mathbf{m}}\big)_{-}(\mathtt{s},b,c)(x)+a_{\mathbf{m}}(\mathtt{s},b,c)-c\right|=0.$
	\end{enumerate}
\end{itemize}
\item Let $c\in\mathbb{R}^*$ and  $\mathbf{m}\in\mathbb{N}^*$ with $\mathbf{m}\geqslant N_1(c)$. Then, there exists a global curve
$$\mathscr{C}_{\textnormal{\tiny{global}}}^{\mathbf{m}}(c)\triangleq\Big\{\big(a_{\mathbf{m}}(\mathtt{s},c),\check{r}_{\mathbf{m}}(\mathtt{s},c)\big),\quad\mathtt{s}\in\mathbb{R}\Big\}$$
corresponding to $\mathbf{m}$-symmetric E-states and extending the local curves $\mathscr{C}_{\textnormal{\tiny{local}}}^{\mathbf{m}}(c)$ given by Theorem \ref{thm E-states}-(iv). Moreover, the curve $\mathscr{C}_{\textnormal{\tiny{global}}}^{\mathbf{m}}(c)$ admits locally around each of its points a real-analytic reparametrization. In addition, one has the following alternatives
\begin{itemize}
	\item [$(A1)$] There exist $T_{\mathbf{m}}(c)>0$ such that
	$$\forall\mathtt{s}\in\mathbb{R},\quad a_{\mathbf{m}}\big(\mathtt{s}+T_{\mathbf{m}}(c),c\big)=a_{\mathbf{m}}(\mathtt{s},c)\qquad\textnormal{and}\qquad \check{r}_{\mathbf{m}}\big(\mathtt{s}+T_{\mathbf{m}}(c),c\big)=\check{r}_{\mathbf{m}}(\mathtt{s},c).$$
	\item [$(A2)$] One of the following limits occurs (possibly simultaneously)
	\begin{enumerate}[label=\textbullet]
		\item (Blow-up) $\displaystyle\lim_{\mathtt{s}\to\pm\infty}\frac{1}{1+a_{\mathbf{m}}(\mathtt{s},c)+\|\check{r}_{\mathbf{m}}(\mathtt{s},c)\|_{s,\sigma}}=0.$
		\item (Collision of the boundaries) $\displaystyle\lim_{\mathtt{s}\to\pm\infty}\min_{x\in\mathbb{T}}\left|\big(\check{r}_{\mathbf{m}}\big)_{+}(\mathtt{s},c)(x)-\big(\check{r}_{\mathbf{m}}\big)_{-}(\mathtt{s},c)(x)+2a_{\mathbf{m}}(\mathtt{s},c)\right|=0.$
		\item (Degeneracy $+$) $\displaystyle\lim_{\mathtt{s}\to\pm\infty}\min_{x\in\mathbb{T}}\left|\big(\check{r}_{\mathbf{m}}\big)_{+}(\mathtt{s},c)(x)+a_{\mathbf{m}}(\mathtt{s},c)-c\right|=0.$
		\item (Degeneracy $-$) $\displaystyle\lim_{\mathtt{s}\to\pm\infty}\min_{x\in\mathbb{T}}\left|\big(\check{r}_{\mathbf{m}}\big)_{-}(\mathtt{s},c)(x)-a_{\mathbf{m}}(\mathtt{s},c)-c\right|=0.$
	\end{enumerate}
\end{itemize}
	\end{enumerate}
\end{theo}

\begin{remark} We first make the following remarks concerning Theorem \ref{thm E-states}.
	\begin{enumerate}
		\item The "hyperbolic" structure of the bifurcation diagram (Figure \ref{fig hyperbolic}) is in contrast with the Eulerian case for doubly-connected patches \cite{HHMV16,HR17} where we have an "elliptic" situation.
		\item The fact of having a finite number of bifurcation points in the asymmetric case is also very interesting. This rarely happens in the fluid patch class, see \cite[Thm. 1.1-(i)]{GHM23}.
		\item A precise expression of $c_{\mathbf{m},2}^{\pm}(a,b)$, $b_{\mathbf{m},2}(a,c)$, $a_{\mathbf{m},2}(b,c)$ and $a_{\mathbf{m},2}(c)$ will be given along Section \ref{sec local}. More generally, one can obtain a general expansion of the solutions by an algorithmic procedure inserting an a priori unknown expansion into the equations and solving the resulting system order by order with respect to the parametrization parameter $\mathtt{s}.$
		\item For the items (ii) and (iii), the case $c=0$ (corresponding to stationary solutions) can be reached for suitable ranges of $a$ or $b$. This makes echo to \cite{GS19}.
	\end{enumerate}
Now, let us discuss the conclusions of Theorem \ref{thm global E-states}.
\begin{enumerate}
	\item Conversely to \cite{HR17} and similarly to \cite{HMW20}, we expect that the loop alternative $(A1)$ does not occur. This may require a more refined analysis introducing suitable nodal conditions and reformulating the study as a Riemann-Hilbert problem.
	\item As we shall see in Section \ref{sec quali}, for a true solution of the system \eqref{system rpm per} (corresponding to E-states), the "Degeneracy $\pm$" alternative happens for a critical point of $\check{r}_{\mp}.$ Hence, for the limiting E-states (end of the branch), we expect the formation of corners (as conjectured in the case of V-states).
\end{enumerate}
\end{remark}
	\noindent\textbf{Acknoledgment :} This work has been supported by PRIN 2020XB3EFL, "Hamiltonain and Dispersive PDEs". I would like to thank Frédéric Rousset who mentioned the notion of patches of electrons (that i didn't know at that time) during a talk in a conference in Lyon. His remark led me to think about this work and upcoming ones. Thank you also to Slim Ibrahim for providing me the manuscript \cite{D87} and to Alberto Maspero for pointing the valuable reference \cite{LS22}. Finally, I thank Massimiliano Berti and Taoufik Hmidi for stimulating discussions when writing this document. 
	\section{Local construction}\label{sec local}
	This section is devoted to the local construction of the branches. For that purpose, we reformulate the problem by looking for the zeros of a nonlinear time independant functional. The flat strips correspond to a trivial line of roots when only one of the free parameters varies. Then, we implement in a systematic way the Crandall-Rabinowitz Theorem \ref{thm CR+S} by looking at the kernel, range and transversality conditions. We also investigate the pitchfork bifurcation property of the constructed branches by computing the condition given by Shi in \cite{S99}, see also Theorem \ref{thm CR+S}.\\
	
	We shall look for solutions of \eqref{system rpm} in the form
	$$r_{\pm}(t,x)=\check{r}_{\pm}(x-ct), \qquad c\in\mathbb{R},\qquad\check{r}_{\pm}\in L^{2}(\mathbb{T}).$$
	With this ansatz, the system \eqref{system rpm} becomes
	\begin{equation}\label{system rpm per}
		\begin{array}{l}
			\forall x\in\mathbb{T},\quad F(a,b,c,\check{r}_+,\check{r}_-)(x)=0,\qquad F\triangleq(F_+\,,\,F_-),\vspace{0.2cm}\\
			F_+(a,b,c,\check{r}_+,\check{r}_-)(x)\triangleq \big(\check{r}_{+}(x)+b-c\big)\partial_{x}\check{r}_{+}(x)-\tfrac{1}{b-a}\partial_{x}^{-1}\check{r}_{+}(x)+\tfrac{1}{b-a}\partial_{x}^{-1}\check{r}_{-}(x),\vspace{0.2cm}\\
			F_-(a,b,c,\check{r}_+,\check{r}_-)(x)\triangleq \big(\check{r}_{-}(x)+a-c\big)\partial_{x}\check{r}_{-}(x)-\tfrac{1}{b-a}\partial_{x}^{-1}\check{r}_{+}(x)+\tfrac{1}{b-a}\partial_{x}^{-1}\check{r}_{-}(x).
		\end{array}
	\end{equation}
	One readily has
	\begin{equation}\label{trivial line a-c}
		\forall a<b,\quad\forall c\in\mathbb{R},\quad F(a,b,c,0,0)=0.
	\end{equation}
	This identity provides, either for fixed couple $(a,b),$ $(a,c)$ or $(b,c)$ a line of trivial solutions corresponding to the flat strip(s), see Lemma \ref{lem trivial solutions flat strip}. We shall find non-trivial solutions of \eqref{system rpm per} by implementing the local bifurcation theory through the use of Crandall-Rabinowitz Theorem \ref{thm CR+S}. We shall work with the following Sobolev-analytic function spaces defined for any $s,\sigma\geqslant0$ and $\mathbf{m}\in\mathbb{N}^*$ by
	\begin{align*}
		X_{\mathbf{m}}^{s,\sigma}&\triangleq\Big\{f=(f_+,f_-),\qquad\forall x\in\mathbb{T},\quad f_{\pm}(x)=\sum_{j=1}^{\infty}f_{j}^{\pm}\cos(2\pi j\mathbf{m}x),\quad f_{j}^{\pm}\in\mathbb{R},\quad\sum_{j=1}^{\infty}|2\pi j|^{2s}|f_{j}^{\pm}|^2e^{4\pi\sigma|j|}<\infty\Big\},\\
		Y_{\mathbf{m}}^{s,\sigma}&\triangleq\Big\{g=(g_+,g_-),\qquad \forall x\in\mathbb{T},\quad g_{\pm}(x)=\sum_{j=1}^{\infty}g_{j}^{\pm}\sin(2\pi j\mathbf{m}x),\quad g_{j}^{\pm}\in\mathbb{R},\quad\sum_{j=1}^{\infty}|2\pi j|^{2s}|g_{j}^{\pm}|^2e^{4\pi\sigma|j|}<\infty\Big\}.
	\end{align*}
	Both spaces are endowed with the norm 
	$$\|(u_+,u_-)\|_{s,\sigma}\triangleq\|u_+\|_{s,\sigma}+\|u_-\|_{s,\sigma},\qquad\|u_{\pm}\|_{s,\sigma}\triangleq\left(\sum_{j=1}^{\infty}(2\pi j)^{2s}(u_{j}^{\pm})^2e^{4\pi\sigma j}\right)^{\frac{1}{2}}.$$
	The average is taken equal to zero in accordance with \eqref{zero average}. We refer the reader for instance to \cite{ABZ22} for a nice introduction to the general Sobolev-analytic spaces and there properties. We mention that the Sobolev-analytic scale $(H^{s,\sigma})_{s\geqslant 0,\sigma\geqslant0}$ behaves like the classical Sobolev scale $(H^s)_{s\geqslant0}$ and admits the same properties with respect to the Sobolev regularity parameter $s$ (interpolation, product and composition laws, compact embeddings etc...). One can easily check from the structure \eqref{system rpm per}, using in particular the classical formula
	\begin{equation}\label{sicosisi}
		\forall(u,v)\in\mathbb{R}^2,\quad\sin(u)\cos(v)=\tfrac{1}{2}\big(\sin(u+v)+\sin(u-v)\big),
	\end{equation}
	that, denoting $\mathbb{S}\triangleq\{(x,y)\in\mathbb{R}^2\quad\textnormal{s.t.}\quad x<y\}$,
	\begin{equation}\label{regularity F}
		\textnormal{for any }\sigma>0\textnormal{ and }s\geqslant 1,\textnormal{ the function }F:\mathbb{S}\times\mathbb{R}\times X_{\mathbf{m}}^{s,\sigma}\rightarrow Y_{\mathbf{m}}^{s-1,\sigma}\textnormal{ is well-defined and analytic.}
	\end{equation}
	Indeed, the analyticity results for that \eqref{system rpm per} only involves linear or quadratic terms. From now on, we fix $s\geqslant 1$ and $\sigma>0$. The linearized operator at $(\check{r}_{+},\check{r}_{-})=(0,0)$ in the direction $(h_+,h_-)$ is
	\begin{align}\label{lin op}
		d_{(\check{r}_{+},\check{r}_{-})}F(a,b,c,0,0)[h_+,h_-]=I_{(0,0)}[h_+,h_-]+K_{(0,0)}[h_+,h_-],
	\end{align}
	with
	\begin{align}\label{lin op 2}
		I_{(0,0)}\triangleq\begin{pmatrix}
			(b-c)\partial_{x} & 0\\
			0 & (a-c)\partial_{x}
		\end{pmatrix},\qquad
		K_{(0,0)}\triangleq\tfrac{1}{b-a}\begin{pmatrix}
			-\partial_{x}^{-1} & \partial_{x}^{-1}\\
			-\partial_{x}^{-1} & \partial_{x}^{-1}
		\end{pmatrix}.
	\end{align}
	If $c\not\in\{a,b\},$ then $I_{(0,0)}:X_{\mathbf{m}}^{s,\sigma}\rightarrow Y_{\mathbf{m}}^{s-1,\sigma}$ is an isomorphism. In addition, $K_{(0,0)}:X_{\mathbf{m}}^{s,\sigma}\rightarrow Y_{\mathbf{m}}^{s+1,\sigma}$ is continuous and by Rellich-type Theorem, we deduce that $K_{(0,0)}:X_{\mathbf{m}}^{s,\sigma}\rightarrow Y_{\mathbf{m}}^{s-1,\sigma}$ is a compact operator. Therefore, 
	\begin{equation}\label{Fredholmness}
		\textnormal{for }c\not\in\{a,b\},\,\sigma>0\textnormal{ and }s\geqslant 1,\,\, d_{(\check{r}_{+},\check{r}_{-})}F(a,b,c,0,0):X_{\mathbf{m}}^{s,\sigma}\rightarrow Y_{\mathbf{m}}^{s-1,\sigma}\textnormal{ is a zero index Fredholm operator.}
	\end{equation}
	Moreover, this latter admits the following Fourier representation : for given $(h_+,h_-)\in X_{\mathbf{m}}^{s,\sigma}$ in the form
	$$h_{\pm}(x)=\sum_{j=1}^{\infty}h_{j}^{\pm}\cos(2\pi j\mathbf{m}x),\qquad h_{j}^{\pm}\in\mathbb{R},$$
	we have
	\begin{equation}\label{linear per}
		d_{(\check{r}_{+},\check{r}_{-})}F(a,b,c,0,0)[h_+,h_-](x)=-\sum_{j=1}^{\infty}M_{j\mathbf{m}}(a,b,c)\begin{pmatrix}
			h_{j}^{+}\\
			h_{j}^{-}
		\end{pmatrix}\sin(2\pi j\mathbf{m}x),
	\end{equation}
	where
	\begin{equation}\label{def Mj per}
		M_{j}(a,b,c)\triangleq \begin{pmatrix}
			2\pi j(b-c)+\tfrac{1}{2\pi j(b-a)} & -\tfrac{1}{2\pi j(b-a)}\\
			\tfrac{1}{2\pi j(b-a)} & 2\pi j(a-c)-\tfrac{1}{2\pi j(b-a)}
		\end{pmatrix}.
	\end{equation}
	In order to apply Crandall-Rabinowitz Theorem, we shall look for the singularity of the matrices $M_{j}(a,b,c)$ giving rise in turn to a non-trivial kernel for the linearized operator. For any $j\in\mathbb{N}^*,$ the determinant of $M_{j}(a,b,c)$ is
	\begin{equation}\label{determinant Mj per}
		\Delta_{j}(a,b,c)\triangleq\det\big(M_{j}(a,b,c)\big)=4\pi^{2}j^{2}(b-c)(a-c)-1.
	\end{equation}
	We shall now dissociate the analysis whenever two of the three parameters $a,$ $b$ or $c$ are fixed, which leads to the following subsections. We also discuss the symmetric case $(a,b)=(-a,a),$ with $a>0$ which must be treated separately.
	\subsection{Velocity bifurcation}\label{sec vel bif}
	In this subsection, we fix $a<b$ and study the bifurcation with respect to the parameter $c\in\mathbb{R}.$ First, we can rewrite the determinant $\Delta_{j}(a,b,c)$ in \eqref{determinant Mj per} as a polynomial of degree two in the variable $c$ as follows
	\begin{equation}\label{det reec}
		\Delta_{j}(a,b,c)=4\pi^2j^2c^2-4\pi^2j^2(a+b)c+4\pi^2j^2ab-1.
	\end{equation}
	The associated discriminant is
	\begin{align*}
		\delta_{j}(a,b)&\triangleq 16\pi^4j^4(a+b)^2+16\pi^2j^2(1-4\pi^2j^2ab)\\
		&=16\pi^2j^2\big(\pi^2j^2(b-a)^2+1\big)>0.
	\end{align*}
	We deduce that
	\begin{equation}\label{def cj}
		\Delta_{j}(a,b,c)=0\quad\Leftrightarrow\quad c=c_j^{\pm}(a,b)\triangleq\frac{a+b}{2}\pm\sqrt{\frac{\pi^2j^2(b-a)^2+1}{4\pi^2j^2}}\cdot
	\end{equation}
	$\blacktriangleright$ \textbf{One dimensional kernel condition :} The sequence $\big(c_j^{+}(a,b)\big)_{j\in\mathbb{N}^*}$ is decreasing and tends to $b$ when $j\to\infty$ whereas the sequence $\big(c_j^{-}(a,b)\big)_{j\in\mathbb{N}^*}$ is increasing and tends to $a$ when $j\to\infty$. As a consequence, for any fixed $\mathbf{m}\in\mathbb{N}^*,$
	$$\Delta_{\mathbf{m}}\big(a,b,c_{\mathbf{m}}^{\pm}(a,b)\big)=0\qquad\textnormal{and}\qquad\forall j\in\mathbb{N}\setminus\{0,1\},\quad\Delta_{j\mathbf{m}}\big(a,b,c_{\mathbf{m}}^{\pm}(a,b)\big)\neq 0.$$
	Actually, from \eqref{determinant Mj per}, we deduce
\begin{align}\label{sign det1}
	\forall j\in\mathbb{N}\setminus\{0,1\},\quad\Delta_{j\mathbf{m}}\big(a,b,c_{\mathbf{m}}^{\pm}(a,b)\big)&=4\pi^2\mathbf{m}^2j^2\big(b-c_{\mathbf{m}}^{\pm}(a,b)\big)\big(a-c_{\mathbf{m}}^{\pm}(a,b)\big)-1\nonumber\\
	&=j^2-1>0.
\end{align}
	Thus, the kernel of $d_{(\check{r}_{+},\check{r}_{-})}F\big(a,b,c_{\mathbf{m}}^{\pm}(a,b),0,0\big)$ is one dimensional, more precisely
	\begin{equation}\label{ker-a}
		\begin{array}{l}
			\ker\Big(d_{(\check{r}_{+},\check{r}_{-})}F\big(a,b,c_{\mathbf{m}}^{\pm}(a,b),0,0\big)\Big)=\mathtt{span}\big( \check{r}_{0,\mathbf{m},a,b}^{\pm}\big),\vspace{0.2cm}\\
			\check{r}_{0,\mathbf{m},a,b}^{\pm}(x)\triangleq\begin{pmatrix}
				2\pi\mathbf{m}\big(a-c_{\mathbf{m}}^{\pm}(a,b)\big)-\tfrac{1}{2\pi\mathbf{m}(b-a)}\\
				-\tfrac{1}{2\pi\mathbf{m}(b-a)}
			\end{pmatrix}\cos(2\pi \mathbf{m}x).
		\end{array}
	\end{equation}
	$\blacktriangleright$ \textbf{Range condition :} Notice that the monotonicity and the convergence of the sequences $\big(c_{j}^{\pm}(a,b)\big)_{j\in\mathbb{N}^*}$ give that $c_{\mathbf{m}}^{\pm}(a,b)\not\in\{a,b\}.$ So the Fredholmness property \eqref{Fredholmness} together with the previous point imply that the range $\mathcal{R}\Big(d_{(\check{r}_{+},\check{r}_{-})}F\big(a,b,c_{\mathbf{m}}^{\pm}(a,b),0,0\big)\Big)$ is closed and of codimension one in $Y_{\mathbf{m}}^{s-1,\sigma}.$ To check, later on, the transversality, we shall prove that
	\begin{equation}\label{range-a}
		\begin{array}{l}
			\mathcal{R}\Big(d_{(\check{r}_{+},\check{r}_{-})}F\big(a,b,c_{\mathbf{m}}^{\pm}(a,b),0,0\big)\Big)=V^{\perp},\vspace{0.2cm}\\
			V\triangleq\mathtt{span}\big(g_{0,\mathbf{m},a,b}^{\pm}\big),\qquad g_{0,\mathbf{m},a,b}^{\pm}(x)\triangleq\begin{pmatrix}
				2\pi\mathbf{m}\big(a-c_{\mathbf{m}}^{\pm}(a,b)\big)-\tfrac{1}{2\pi\mathbf{m}(b-a)}\\
				\tfrac{1}{2\pi\mathbf{m}(b-a)}
			\end{pmatrix}\sin(2\pi\mathbf{m}x),
		\end{array}
	\end{equation}
	where the orthogonal is understood in the sense of the scalar product defined on $Y_\mathbf{m}^{s-1,\sigma}$ as follows: for any $g=(g_+,g_-)\in Y_{\mathbf{m}}^{s-1,\sigma}$ and $\widetilde{g}=(\widetilde{g}_+,\widetilde{g}_-)\in Y_{\mathbf{m}}^{s-1,\sigma}$ writing
	\begin{equation}\label{form g}
		g_{\pm}(x)=\sum_{j=1}^{\infty}g_{j}^{\pm}\sin(2\pi j\mathbf{m}x),\qquad \widetilde{g}_{\pm}(x)=\sum_{j=1}^{\infty}\widetilde{g}_{j}^{\,\pm}\sin(2\pi j\mathbf{m}x),\qquad g_j^{\pm},\widetilde{g}_j^{\pm}\in\mathbb{R},
	\end{equation}
	the scalar product $\langle g,\widetilde{g}\rangle$ of $g$ and $\widetilde{g}$ is given by
	\begin{equation}\label{scalar product per}
		\big\langle g,\widetilde{g}\big\rangle\triangleq\int_{0}^{1}\big(g_+(x)\widetilde{g}_+(x)+g_-(x)\widetilde{g}_-(x)\big)dx=\tfrac{1}{2}\sum_{j=1}^{\infty}\big(g_{j}^{+}\,\widetilde{g}_{j}^{\,+}+g_{j}^{-}\,\widetilde{g}_{j}^{\,-}\big)=\tfrac{1}{2}\sum_{j=1}^{\infty}\left\langle \begin{pmatrix}
			g_{j}^{+}\\
			g_{j}^{-}
		\end{pmatrix}\,,\,\begin{pmatrix}
			\widetilde{g}_{j}^{\,+}\\
			\widetilde{g}_{j}^{\,-}
		\end{pmatrix}\right\rangle_{\mathbb{R}^2}.
	\end{equation}
	Let us now prove \eqref{range-a}. First remark that, by construction, $V$ is a subspace of $Y_{\mathbf{m}}^{s-1,\sigma}$ of dimension one. Therefore, we can apply the orthogonal supplementary theorem in the pre-Hilbertian vector space $\left(Y_{\mathbf{m}}^{s-1,\sigma},\langle\cdot,\cdot\rangle\right)$ to get
	$$Y_{\mathbf{m}}^{s-1,\sigma}=V\overset{\perp}{\oplus}V^{\perp}.$$
	This proves that $V^{\perp}$ is a subspace of $Y_{\mathbf{m}}^{s-1,\sigma}$ of codimension one. Since the range is also of codimension one in $Y_{\mathbf{m}}^{s-1,\sigma}$, it suffices to prove one inclusion is order to get the equality \eqref{range-a}. For
	$$g(x)=\sum_{j=1}^{\infty}M_{j\mathbf{m}}\big(a,b,c_{\mathbf{m}}^{\pm}(a,b)\big)\begin{pmatrix}
		h_{j}^{+}\\
		h_{j}^{-}
	\end{pmatrix}\sin(2\pi j\mathbf{m}x)\in\mathcal{R}\Big(d_{(\check{r}_{+},\check{r}_{-})}F\big(a,b,c_{\mathbf{m}}^{\pm}(a,b),0,0\big)\Big),$$
	we have
	\begin{align*}
		\left\langle g,g_{0,\mathbf{m},a,b}^{\pm}\right\rangle&=\tfrac{1}{2}\left\langle M_{\mathbf{m}}\big(a,b,c_{\mathbf{m}}^{\pm}(a,b)\big)\begin{pmatrix}
			h_{j}^{+}\\
			h_{j}^{-}
		\end{pmatrix},\begin{pmatrix}
			2\pi\mathbf{m}\big(a-c_{\mathbf{m}}^{\pm}(a,b)\big)-\tfrac{1}{2\pi\mathbf{m}(b-a)}\\
			\tfrac{1}{2\pi\mathbf{m}(b-a)}
		\end{pmatrix}\right\rangle_{\mathbb{R}^{2}}\\
		&=\tfrac{1}{2}\left\langle \begin{pmatrix}
			h_{j}^{+}\\
			h_{j}^{-}
		\end{pmatrix},M_{\mathbf{m}}^{\top}\big(a,b,c_{\mathbf{m}}^{\pm}(a,b)\big)\begin{pmatrix}
			2\pi\mathbf{m}\big(a-c_{\mathbf{m}}^{\pm}(a,b)\big)-\tfrac{1}{2\pi\mathbf{m}(b-a)}\\
			\tfrac{1}{2\pi\mathbf{m}(b-a)}
		\end{pmatrix}\right\rangle_{\mathbb{R}^{2}}\\
		&=0.
	\end{align*}
	Here and in the sequel, the notation $M^{\top}$ refers to the transpose of the matrix $M.$ The last equality is true because by construction
	$$\begin{pmatrix}
		2\pi\mathbf{m}\big(a-c_{\mathbf{m}}^{\pm}(a,b)\big)-\tfrac{1}{2\pi\mathbf{m}(b-a)}\\
		\tfrac{1}{2\pi\mathbf{m}(b-a)}
	\end{pmatrix}\in\ker\Big(M_{\mathbf{m}}^{\top}\big(a,b,c_{\mathbf{m}}^{\pm}(a,b)\big)\Big).$$
	This proves 
	$$\mathcal{R}\Big(d_{(\check{r}_{+},\check{r}_{-})}F\big(a,b,c_{\mathbf{m}}^{\pm}(a,b),0,0\big)\Big)\subset V^{\perp}.$$
	Hence, by applying Lemma \ref{lem codim1}, we conclude \eqref{range-a}. Now, still for later puposes, we shall also show that
	\begin{equation}\label{range-ker-a}
		\mathcal{R}\Big(d_{(\check{r}_{+},\check{r}_{-})}F\big(a,b,c_{\mathbf{m}}^{\pm}(a,b),0,0\big)\Big)=\ker(l),
	\end{equation}
	with $l\in\left(Y_{\mathbf{m}}^{s-1,\sigma}\right)^*$ defined by
	$$\forall g\in Y_{\mathbf{m}}^{s-1,\sigma},\quad l(g)\triangleq\left\langle g,g_{0,\mathbf{m},a,b}^{\pm}\right\rangle.$$
	Actually, the identity \eqref{range-ker-a} is just a reformulation of \eqref{range-a} and the only thing to prove is that $l\in\left(Y_{\mathbf{m}}^{s-1,\sigma}\right)^*.$ For this aim, we denote $\|\cdot\|\triangleq\sqrt{\langle\cdot,\cdot\rangle}$ the norm associated with the scalar product $\langle\cdot,\cdot\rangle.$ Then, for any $g=(g_+,g_-)\in Y_{\mathbf{m}}^{s-1,\sigma}$ in the form \eqref{form g} (recall that $s\geqslant1$ and $\sigma>0$), we have 
	\begin{align*}
		\|g\|^2&=\tfrac{1}{2}\sum_{j=1}^{\infty}\Big[\left(g_j^+\right)^2+\left(g_j^-\right)^2\Big]\\
		&\leqslant\tfrac{1}{2}\sum_{j=1}^{\infty}\Big[(2\pi j)^{2(s-1)}\left(g_j^+\right)^2e^{4\pi\sigma j}+(2\pi j)^{2(s-1)}\left(g_j^-\right)^2e^{4\pi\sigma j}\Big]\\
		&\leqslant\tfrac{1}{2}\big(\|g_+\|_{s-1,\sigma}^2+\|g_-\|_{s-1,\sigma}^2\big),
	\end{align*}
from which we deduce
\begin{align*}
	\|g\|&\leqslant\tfrac{1}{\sqrt{2}}\big(\|g_+\|_{s-1,\sigma}^2+\|g_-\|_{s-1,\sigma}^2\big)^{\frac{1}{2}}\\
	&\leqslant\tfrac{1}{\sqrt{2}}\big(\|g_+\|_{s-1,\sigma}+\|g_-\|_{s-1,\sigma}\big)\\
	&=\tfrac{1}{\sqrt{2}}\|g\|_{s-1,\sigma}.
\end{align*}
In the second inequality above, we have used the classical relation
$$\forall(a,b)\in(\mathbb{R}_+)^2,\quad\sqrt{a+b}\leqslant\sqrt{a}+\sqrt{b}.$$
Therefore, by Cauchy-Schwarz inequality, we get
	\begin{align*}
		|l(g)|&=\left|\left\langle g,g_{0,\mathbf{m},a,b}^{\pm}\right\rangle\right|\\
		&\leqslant\left\|g_{0,\mathbf{m},a,b}^{\pm}\right\|\|g\|\\
		&\leqslant\tfrac{1}{\sqrt{2}}\left\|g_{0,\mathbf{m},a,b}^{\pm}\right\|\|g\|_{s-1,\sigma}.
	\end{align*}
	This proves the claim.\\
	$\blacktriangleright$ \textbf{Transversality condition :} From \eqref{lin op}-\eqref{lin op 2}, we have
	$$\partial_{c}d_{(\check{r}_+,\check{r}_-)}F(a,b,c,0,0)=\begin{pmatrix}
		-\partial_{x} & 0\\
		0 & -\partial_{x}
	\end{pmatrix}.$$
	Hence,
	$$\partial_{c}d_{(\check{r}_{+},\check{r}_{-})}F\big(a,b,c_{\mathbf{m}}^{\pm}(a,b),0,0\big)[\check{r}_{0,\mathbf{m},a,b}^{\pm}](x)=2\pi\mathbf{m}\begin{pmatrix}
		2\pi\mathbf{m}\big(a-c_{\mathbf{m}}^{\pm}(a,b)\big)-\tfrac{1}{2\pi\mathbf{m}(b-a)}\\
		-\tfrac{1}{2\pi\mathbf{m}(b-a)}
	\end{pmatrix}\sin(2\pi \mathbf{m}x).$$
	Straightforward computations lead to
	\begin{equation}\label{trans exp cpm}
		\left\langle \partial_{c}d_{(\check{r}_{+},\check{r}_{-})}F\big(a,b,c_{\mathbf{m}}^{\pm}(a,b),0,0\big)[\check{r}_{0,\mathbf{m},a,b}^{\pm}]\,,\,g_{0,\mathbf{m},a,b}^{\pm}\right\rangle=2\pi\mathbf{m}\Big[a-c_{\mathbf{m}}^{\pm}(a,b)\Big]\Big[2\pi^{2}\mathbf{m}^2\big(a-c_{\mathbf{m}}^{\pm}(a,b)\big)-\tfrac{1}{b-a}\Big].
	\end{equation}
	Since $c_{\mathbf{m}}^{+}(a,b)>b>a,$ then 
	$$a-c_{\mathbf{m}}^{+}(a,b)<0\qquad\textnormal{and}\qquad 2\pi^{2}\mathbf{m}^2\big(a-c_{\mathbf{m}}^{+}(a,b)\big)-\tfrac{1}{b-a}<0.$$
	Thus,
	\begin{equation}\label{sgn cm-cond}
		\left\langle \partial_{c}d_{(\check{r}_{+},\check{r}_{-})}F\big(a,b,c_{\mathbf{m}}^{+}(a,b),0,0\big)[\check{r}_{0,\mathbf{m},a,b}^{+}]\,,\,g_{0,\mathbf{m},a,b}^{+}\right\rangle>0.
	\end{equation}
	In addition, one can easily check from \eqref{def cj} that
	$$a>c_{\mathbf{m}}^{-}(a,b)>a-\frac{1}{2\pi^2\mathbf{m}^2(b-a)}\cdot$$
	Therefore,
	\begin{equation}\label{sgn cm+cond}
		\left\langle \partial_{c}d_{(\check{r}_{+},\check{r}_{-})}F\big(a,b,c_{\mathbf{m}}^{-}(a,b),0,0\big)[\check{r}_{0,\mathbf{m},a,b}^{-}]\,,\,g_{0,\mathbf{m},a,b}^{-}\right\rangle<0.
	\end{equation}
	In particular, in both cases, we have
	$$\left\langle \partial_{c}d_{(\check{r}_{+},\check{r}_{-})}F\big(a,b,c_{\mathbf{m}}^{\pm}(a,b),0,0\big)[\check{r}_{0,\mathbf{m},a,b}^{\pm}]\,,\,g_{0,\mathbf{m},a,b}^{\pm}\right\rangle\neq 0,$$
	which is, in view of \eqref{range-ker-a}, equivalent to
	\begin{equation}\label{transversality-a}
		\partial_{c}d_{(\check{r}_{+},\check{r}_{-})}F\big(a,b,c_{\mathbf{m}}^{\pm}(a,b),0,0\big)[\check{r}_{0,\mathbf{m},a,b}^{\pm}]\not\in\mathcal{R}\Big(d_{(\check{r}_{+},\check{r}_{-})}F\big(a,b,c_{\mathbf{m}}^{\pm}(a,b),0,0\big)\Big).
	\end{equation}
	Finally, \eqref{trivial line a-c}, \eqref{regularity F}, \eqref{Fredholmness}, \eqref{ker-a} and \eqref{transversality-a} allow to apply the Crandall-Rabinowitz Theorem \ref{thm CR+S} which concludes the existence part of Theorem \ref{thm E-states}-(i). In the sequel, we denote
	\begin{equation}\label{def local curve Cpm}
		\mathscr{C}_{\textnormal{\tiny{local}}}^{\pm,\mathbf{m}}(a,b):\mathtt{s}\in(-\delta,\delta)\mapsto \Big(c_{\mathbf{m}}^{\pm}(\mathtt{s},a,b)\,,\,\check{r}_{\mathbf{m}}^{\pm}(\mathtt{s},a,b)\Big)\in\mathbb{R}\times X_{\mathbf{m}}^{s,\sigma},\qquad\delta>0
	\end{equation}
	the corresponding (real-analytic) local curve which satisfies
	\begin{equation}\label{frst orders exp cpm}
		c_{\mathbf{m}}^{\pm}(0,a,b)=c_{\mathbf{m}}^{\pm}(a,b),\qquad\frac{d}{d\mathtt{s}}\Big(\check{r}_{\mathbf{m}}^{\pm}(\mathtt{s},a,b)\Big)\Big|_{\mathtt{s}=0}=\check{r}_{0,\mathbf{m},a,b}^{\pm}.
	\end{equation}
	$\blacktriangleright$ \textbf{Pitchfork-type bifurcation :} According to Theorem \ref{thm CR+S}, we first need to prove that
	\begin{equation}\label{pitchforck cond}
		d_{(\check{r}_+,\check{r}_-)}^{2}F\big(a,b,c_{\mathbf{m}}^{\pm}(a,b),0,0\big)[\check{r}_{0,\mathbf{m},a,b}^{\pm},\check{r}_{0,\mathbf{m},a,b}^{\pm}]\in\mathcal{R}\Big(d_{(\check{r}_{+},\check{r}_{-})}F\big(a,b,c_{\mathbf{m}}^{\pm}(a,b),0,0\big)\Big).
	\end{equation}
	One readily has from \eqref{system rpm per},
	\begin{equation}\label{diff 2 and 3}
		d_{(\check{r}_+,\check{r}_-)}^{2}F_{\pm} \big(a,b,c,0,0\big)[(h_+,h_-),(\widetilde{h}_+,\widetilde{h}_-)]=\partial_{x}\big(h_{\pm}\widetilde{h}_{\pm}\big),\qquad d_{(\check{r}_+,\check{r}_-)}^{3}F(a,b,c,0,0)=0.
	\end{equation}
	Consequently, using the identity \eqref{sicosisi}, we find
	\begin{align}\label{d2Frr}
		&d_{(\check{r}_+,\check{r}_-)}^{2}F\big(a,b,c_{\mathbf{m}}^{\pm}(a,b),0,0\big)[\check{r}_{0,\mathbf{m},a,b}^{\pm},\check{r}_{0,\mathbf{m},a,b}^{\pm}]\nonumber\\
		&=\begin{pmatrix}
			\Big(2\pi\mathbf{m}\big(a-c_{\mathbf{m}}^{\pm}(a,b)\big)-\tfrac{1}{2\pi\mathbf{m}(b-a)}\Big)^2\\
			\tfrac{1}{4\pi^2\mathbf{m}^2(b-a)^2}
		\end{pmatrix}\partial_{x}\Big(\cos^2(2\pi\mathbf{m}x)\Big)\nonumber\\
		&=-4\pi\mathbf{m}\begin{pmatrix}
			\Big(2\pi\mathbf{m}\big(a-c_{\mathbf{m}}^{\pm}(a,b)\big)-\tfrac{1}{2\pi\mathbf{m}(b-a)}\Big)^2\\
			\tfrac{1}{4\pi^2\mathbf{m}^2(b-a)^2}
		\end{pmatrix}\sin(2\pi\mathbf{m}x)\cos(2\pi\mathbf{m}x)\nonumber\\
		&=-2\pi\mathbf{m}\begin{pmatrix}
			\Big(2\pi\mathbf{m}\big(a-c_{\mathbf{m}}^{\pm}(a,b)\big)-\tfrac{1}{2\pi\mathbf{m}(b-a)}\Big)^2\\
			\tfrac{1}{4\pi^2\mathbf{m}^2(b-a)^2}
		\end{pmatrix}\sin(4\pi\mathbf{m}x).
	\end{align}
	The orthogonality of the family $\big(x\mapsto\sin(2\pi j\mathbf{m}x)\big)_{j\in\mathbb{N}^*}$ with respect to the scalar product \eqref{scalar product per} implies
	$$\left\langle d_{(\check{r}_+,\check{r}_-)}^{2}F\big(a,b,c_{\mathbf{m}}^{\pm}(a,b),0,0\big)[\check{r}_{0,\mathbf{m},a,b}^{\pm},\check{r}_{0,\mathbf{m},a,b}^{\pm}]\,,\,g_{0,\mathbf{m},a,b}^{\pm}\right\rangle=0,$$
	which is equivalent to \eqref{pitchforck cond} according to \eqref{range-ker-a}. In addition, one can easily check from \eqref{linear per} and \eqref{d2Frr} that
	\begin{equation}\label{def th-ab}
		\theta_{0,\mathbf{m},a,b}^{\pm}(x)\triangleq2\pi\mathbf{m}M_{2\mathbf{m}}^{-1}\big(a,b,c_{\mathbf{m}}^{\pm}(a,b)\big)\begin{pmatrix}
			\Big(2\pi\mathbf{m}\big(a-c_{\mathbf{m}}^{\pm}(a,b)\big)-\tfrac{1}{2\pi\mathbf{m}(b-a)}\Big)^2\\
			\tfrac{1}{4\pi^2\mathbf{m}^2(b-a)^2}
		\end{pmatrix}\cos(4\pi\mathbf{m}x)
	\end{equation}
	satisfies
	$$d_{(\check{r}_{+},\check{r}_{-})}F\big(a,b,c_{\mathbf{m}}^{\pm}(a,b),0,0\big)[\theta_{0,\mathbf{m},a,b}^{\pm}]=d_{(\check{r}_+,\check{r}_-)}^{2}F\big(a,b,c_{\mathbf{m}}^{\pm}(a,b),0,0\big)[\check{r}_{0,\mathbf{m},a,b}^{\pm},\check{r}_{0,\mathbf{m},a,b}^{\pm}].$$
	Then, according to Theorem \ref{thm CR+S}, \eqref{pitchforck cond} and \eqref{diff 2 and 3}, we have
	\begin{equation}\label{vanishing frist order cpm}
		\frac{d}{d\mathtt{s}}\Big(c_{\mathbf{m}}^{\pm}(\mathtt{s},a,b)\Big)\Big|_{\mathtt{s}=0}=0
	\end{equation}
	and
	\begin{equation}\label{param curv diff0}
		\frac{d^2}{d\mathtt{s}^2}\Big(c_{\mathbf{m}}^{\pm}(\mathtt{s},a,b)\Big)\Big|_{\mathtt{s}=0}=\frac{\left\langle d_{(\check{r}_{+},\check{r}_{-})}^2F\big(a,b,c_{\mathbf{m}}^{\pm}(a,b),0,0\big)[\check{r}_{0,\mathbf{m},a,b}^{\pm},\theta_{0,\mathbf{m},a,b}^{\pm}]\,,\,g_{0,\mathbf{m},a,b}^{\pm}\right\rangle}{\left\langle \partial_c d_{(\check{r}_{+},\check{r}_{-})}F\big(a,b,c_{\mathbf{m}}^{\pm}(a,b),0,0\big)[\check{r}_{0,\mathbf{m},a,b}^{\pm}]\,,\,g_{0,\mathbf{m},a,b}^{\pm}\right\rangle}\cdot
	\end{equation}
	Now, to conclude the pitchfork bifurcation, it remains to prove that
	$$\left\langle d_{(\check{r}_{+},\check{r}_{-})}^2F\big(a,b,c_{\mathbf{m}}^{\pm}(a,b),0,0\big)[\check{r}_{0,\mathbf{m},a,b}^{\pm},\theta_{0,\mathbf{m},a,b}^{\pm}]\,,\,g_{0,\mathbf{m},a,b}^{\pm}\right\rangle\neq0$$
	and more precisely, to study the sign of the previous quantity in order to obtain the (local) direction of the branch. Looking at \eqref{def Mj per}, we denote for any $j\in\mathbb{N}^*,$
	\begin{equation}\label{mat not}
		M_{j}\big(a,b,c\big)\triangleq\begin{pmatrix}
			\alpha_{j}\big(a,b,c\big)& -\gamma_{j}(a,b)\\
			\gamma_{j}(a,b) & \beta_{j}(a,b,c)
		\end{pmatrix}.
	\end{equation}
	With this notation, we have
	$$\check{r}_{0,\mathbf{m},a,b}^{\pm}(x)=\begin{pmatrix}
		\beta_{\mathbf{m}}\big(a,b,c_{\mathbf{m}}^{\pm}(a,b)\big)\\ 
		-\gamma_{\mathbf{m}}(a,b)
	\end{pmatrix}\cos(2\pi\mathbf{m}x),\qquad g_{0,\mathbf{m},a,b}^{\pm}(x)=\begin{pmatrix}
		\beta_{\mathbf{m}}\big(a,b,c_{\mathbf{m}}^{\pm}(a,b)\big)\\
		\gamma_{\mathbf{m}}(a,b)
	\end{pmatrix}\sin(2\pi\mathbf{m}x)$$
	and, using \eqref{sign det1} and \eqref{def th-ab},
	\begin{align*}
		\theta_{0,\mathbf{m},a,b}^{\pm}(x)&=\frac{2\pi\mathbf{m}}{3}\begin{pmatrix}
			\beta_{2\mathbf{m}}\big(a,b,c_{\mathbf{m}}^{\pm}(a,b)\big) & \gamma_{2\mathbf{m}}(a,b)\\
			-\gamma_{2\mathbf{m}}(a,b) & \alpha_{2\mathbf{m}}\big(a,b,c_{\mathbf{m}}^{\pm}(a,b)\big)
		\end{pmatrix}\begin{pmatrix}
			\beta_{\mathbf{m}}^2\big(a,b,c_{\mathbf{m}}^{\pm}(a,b)\big)\\
			\gamma_{\mathbf{m}}^2(a,b)
		\end{pmatrix}\cos(4\pi\mathbf{m}x)\vspace{0.2cm}\\
		&=\frac{2\pi\mathbf{m}}{3}\begin{pmatrix}
			\beta_{2\mathbf{m}}\big(a,b,c_{\mathbf{m}}^{\pm}(a,b)\big)\beta_{\mathbf{m}}^2\big(a,b,c_{\mathbf{m}}^{\pm}(a,b)\big)+\gamma_{2\mathbf{m}}(a,b)\gamma_{\mathbf{m}}^2(a,b)\vspace{0.2cm}\\
			-\gamma_{2\mathbf{m}}(a,b)\beta_{\mathbf{m}}^2\big(a,b,c_{\mathbf{m}}^{\pm}(a,b)\big)+\alpha_{2\mathbf{m}}\big(a,b,c_{\mathbf{m}}^{\pm}(a,b)\big)\gamma_{\mathbf{m}}^2(a,b)
		\end{pmatrix}\cos(4\pi\mathbf{m}x).
	\end{align*}
	Therefore, by \eqref{diff 2 and 3}, we infer
	\begin{align*}
		&d_{(\check{r}_{+},\check{r}_{-})}^2F\big(a,b,c_{\mathbf{m}}^{\pm}(a,b),0,0\big)[\check{r}_{0,\mathbf{m},a,b}^{\pm},\theta_{0,\mathbf{m},a,b}^{\pm}]\\
		&=\tfrac{2\pi\mathbf{m}\partial_{x}\big(\cos(4\pi\mathbf{m}x)\cos(2\pi\mathbf{m}x)\big)}{3}\begin{pmatrix}
			\beta_{2\mathbf{m}}\big(a,b,c_{\mathbf{m}}^{\pm}(a,b)\big)\beta_{\mathbf{m}}^3\big(a,b,c_{\mathbf{m}}^{\pm}(a,b)\big)+\gamma_{2\mathbf{m}}(a,b)\gamma_{\mathbf{m}}^2(a,b)\beta_{\mathbf{m}}\big(a,b,c_{\mathbf{m}}^{\pm}(a,b)\big)\\
			-\alpha_{2\mathbf{m}}\big(a,b,c_{\mathbf{m}}^{\pm}(a,b)\big)\gamma_{\mathbf{m}}^3(a,b)+\gamma_{2\mathbf{m}}(a,b)\gamma_{\mathbf{m}}(a,b)\beta_{\mathbf{m}}^2\big(a,b,c_{\mathbf{m}}^{\pm}(a,b)\big)
		\end{pmatrix}.
	\end{align*}
	Now, from the classical relation
	$$\forall(u,v)\in\mathbb{R}^2,\quad\cos(u)\cos(v)=\tfrac{1}{2}\big(\cos(u+v)+\cos(u-v)\big),$$
	we obtain
	$$\partial_{x}\Big(\cos(4\pi\mathbf{m}x)\cos(2\pi\mathbf{m}x)\Big)=-3\pi\mathbf{m}\sin(6\pi\mathbf{m}x)-\pi\mathbf{m}\sin(2\pi\mathbf{m}x).$$
	Together with the orthogonality of the family $\big(x\mapsto\sin(2\pi j\mathbf{m}x)\big)_{j\in\mathbb{N}^*}$ with respect to the scalar product \eqref{scalar product per}, we get
	\begin{equation}\label{pitch-term1}
		\left\langle d_{(\check{r}_{+},\check{r}_{-})}^2F\big(a,b,c_{\mathbf{m}}^{\pm}(a,b),0,0\big)[\check{r}_{0,\mathbf{m},a,b}^{\pm},\theta_{0,\mathbf{m},a,b}^{\pm}]\,,\,g_{0,\mathbf{m},a,b}^{\pm}\right\rangle=\frac{\pi^2\mathbf{m}^2}{3}\mathtt{f}^{\pm}(\mathbf{m},a,b),
	\end{equation}
	where $\mathtt{f}^{\pm}$ is a well-defined continuous fonction on $[1,\infty)\times\mathbb{S}$ given by
	$$\mathtt{f}^{\pm}(z,a,b)\triangleq \alpha_{2z}\big(a,b,c_{z}^{\pm}(a,b)\big)\gamma_{z}^4(a,b)-\beta_{2z}\big(a,b,c_{z}^{\pm}(a,b)\big)\beta_{z}^4\big(a,b,c_{z}^{\pm}(a,b)\big)-2\gamma_{2z}(a,b)\gamma_{z}^2(a,b)\beta_{z}^2\big(a,b,c_{z}^{\pm}(a,b)\big).$$
	After tedious computations using \eqref{def cj}, we find
	\begin{equation}\label{exp fpmzc}
		\mathtt{f}^{\pm}(z,a,b)=\frac{1}{4\pi^3 z^3(b-a)^3}\Big(\pm A(z,a,b)\sqrt{\pi^2z^2(b-a)^2+1}+B(z,a,b)\Big),
	\end{equation}
	where
	\begin{align*}
		A(z,a,b)&\triangleq128\pi^{7}z^{7}(b-a)^{7}+232\pi^{5}z^{5}(b-a)^{5}+132\pi^{3}z^{3}(b-a)^{3}+24\pi z(b-a),\\
		B(z,a,b)&\triangleq128\pi^{8}z^{8}(b-a)^{8}+296\pi^6z^6(b-a)^6+232\pi^4z^4(b-a)^4+69\pi^2z^2(b-a)^2+6.
	\end{align*}
	Clearly, with this expression, we can conclude that
	$$\forall (z,a,b)\in[1,\infty)\times\mathbb{S},\quad \mathtt{f}^{+}(z,a,b)>0.$$
	Combined with \eqref{pitch-term1}, we deduce
	\begin{equation}\label{sgn pitch00+}
		\left\langle d_{(\check{r}_{+},\check{r}_{-})}^2F\big(a,b,c_{\mathbf{m}}^{+}(a,b),0,0\big)[\check{r}_{0,\mathbf{m},a,b}^{+},\theta_{0,\mathbf{m},a,b}^{+}]\,,\,g_{0,\mathbf{m},a,b}^{+}\right\rangle>0.
	\end{equation}
	Assume, for the sake of contradiction, that there exists $(z,a,b)\in[1,\infty)\times\mathbb{S}$ such that $\mathtt{f}^{-}(z,a,b)=0.$ Then this equation is equivalent to
	$$\sqrt{\pi^2z^2(b-a)^2+1}=\frac{B(z,a,b)}{A(z,a,b)}\cdot$$
	Taking the square of the previous expression, then, after simplifications, we end up with
	\begin{align*}
		0&=256\pi^{8}z^{8}(b-a)^{8}+672\pi^{6}z^{6}(b-a)^{6}+633\pi^{4}z^{4}(b-a)^{4}+252\pi^{2}z^{2}(b-a)^{2}+36,
	\end{align*}
	which is impossible since the right hand-side is positive. We deduce that 
	$$\forall(z,a,b)\in[1,\infty)\times\mathbb{S},\quad \mathtt{f}^{-}(z,a,b)\neq0.$$
	Now, we have the following asymptotic
	$$\mathtt{f}^{-}(z,a,b)\underset{z\to\infty}{\sim}16\pi^3z^3(b-a)^3.$$
	Hence, by a continuity argument, we infer
	$$\forall(z,a,b)\in[1,\infty)\times\mathbb{S},\quad \mathtt{f}^{-}(z,a,b)>0.$$
	Combined with \eqref{pitch-term1}, we deduce
	\begin{equation}\label{sgn pitch00-}
		\left\langle d_{(\check{r}_{+},\check{r}_{-})}^2F\big(a,b,c_{\mathbf{m}}^{-}(a,b),0,0\big)[\check{r}_{0,\mathbf{m},a,b}^{-},\theta_{0,\mathbf{m},a,b}^{-}]\,,\,g_{0,\mathbf{m},a,b}^{-}\right\rangle>0.
	\end{equation}
	Finally, putting together \eqref{param curv diff0}, \eqref{sgn cm-cond}, \eqref{sgn cm+cond}, \eqref{sgn pitch00+} and \eqref{sgn pitch00-}, we obtain
	$$\frac{d^2}{d\mathtt{s}^2}\Big(c_{\mathbf{m}}^{-}(\mathtt{s},a,b)\Big)\Big|_{\mathtt{s}=0}<0\qquad\textnormal{and}\qquad\frac{d^2}{d\mathtt{s}^2}\Big(c_{\mathbf{m}}^{+}(\mathtt{s},a,b)\Big)\Big|_{\mathtt{s}=0}>0,$$
	which corresponds to a "hyperbolic" bifurcation diagram. 
	This achieves the proof of Theorem \ref{thm E-states}-(i).
	\subsection{Bifurcation from $b$}
	In this subsection, we fix $a,c\in\mathbb{R}$ with $a\neq c$ and study the bifurcation with respect to the parameter $b>a.$ Observe that \eqref{determinant Mj per} implies
	\begin{equation}\label{def bjac}
		\Delta_{j}(a,b,c)=0\qquad\Leftrightarrow\qquad b=b_j(a,c)\triangleq c+\frac{1}{4\pi^2j^2(a-c)}\cdot
	\end{equation}
	Here and in the sequel, we shall use the following notation: for any $p\in\mathbb{R}^*$, we denote
	$$N_1(p)\triangleq1+ N_2(p),\qquad N_{2}(p)\triangleq\big\lfloor\tfrac{1}{2\pi|p|}\big\rfloor.$$
	The constraint $b>a$ requires the following restriction
	\begin{equation}\label{cond1}
		\begin{cases}
			4\pi^2j^2(a-c)^2>1, & \textnormal{if }a<c,\\
			4\pi^2j^2(a-c)^2<1, & \textnormal{if }a>c.
		\end{cases}
	\end{equation}
\texttt{Case 1 :} If $a<c$, then there is a countable family of potential bifurcation points, namely
$$b_{\mathbf{m}}(a,c),\qquad\mathbf{m}\geqslant N_1(c-a).$$
\texttt{Case 2 :} If $a>c$, then there are two options.
\begin{enumerate}[label=\textbullet]
	\item If $a\geqslant c+\tfrac{1}{2\pi}$, no bifurcation point exists.
	\item If $a<c+\tfrac{1}{2\pi}$, then there is a finite number of potential bifurcation points, namely
	$$b_{\mathbf{m}}(a,c),\qquad\mathbf{m}\in\big\llbracket1,N_2(a-c)\big\rrbracket.$$
\end{enumerate}
In what follows, the condition $a>c$ always refers to $c<a<c+\tfrac{1}{2\pi}\cdot$\\

\noindent $\blacktriangleright$ \textbf{One dimensional kernel condition :} For $a<c$ (resp. $a>c$) the sequence $\big(b_{j}(a,c)\big)_{j\in\mathbb{N}^*}$ (well-defined) is increasing (resp. decreasing) and tends to $c$ as $j\to\infty.$ In addition, one has, for any fixed $\mathbf{m}\in\mathbb{N}^*$ with $\mathbf{m}\geqslant N_1(c-a)$ (resp. $\mathbf{m}\in\llbracket 1,N_2(a-c)\rrbracket$), 
$$\Delta_{\mathbf{m}}\big(a,b_{\mathbf{m}}(a,c),c\big)=0$$
and, similarly to \eqref{sign det1},
\begin{equation}\label{detsgnac}
	\forall j\in\mathbb{N}\setminus\{0,1\},\quad\Delta_{j\mathbf{m}}\big(a,b_{\mathbf{m}}(a,c),c\big)=j^2-1>0.
\end{equation}
Thus, in both cases, the kernel of $d_{(\check{r}_{+},\check{r}_{-})}F\big(a,b_{\mathbf{m}}(a,c),c,0,0\big)$ is one dimensional, more precisely 
\begin{equation}\label{ker-ac}
	\begin{array}{l}
		\ker\Big(d_{(\check{r}_{+},\check{r}_{-})}F\big(a,b_{\mathbf{m}}(a,c),c,0,0\big)\Big)=\mathtt{span}(\check{r}_{0,\mathbf{m},a,c}),\vspace{0.2cm}\\
		\check{r}_{0,\mathbf{m},a,c}(x)\triangleq\begin{pmatrix}
		2\pi\mathbf{m}\big(a-c\big)-\tfrac{1}{2\pi\mathbf{m}\big(b_{\mathbf{m}}(a,c)-a\big)}\\
		-\tfrac{1}{2\pi\mathbf{m}\big(b_{\mathbf{m}}(a,c)-a\big)}
	\end{pmatrix}\cos(2\pi \mathbf{m}x).
\end{array}
\end{equation}
$\blacktriangleright$ \textbf{Range condition :} In both cases, $a\neq c$ and the definition of $b_{\mathbf{m}}(a,c)$ in \eqref{def bjac} implies that $b_{\mathbf{m}}(a,c)\neq c.$ So the Fredholmness property \eqref{Fredholmness} gives that the range $\mathcal{R}\Big(d_{(\check{r}_{+},\check{r}_{-})}F\big(a,b_{\mathbf{m}}(a,c),c,0,0\big)\Big)$ is closed and of codimension one in $Y_{\mathbf{m}}^{s-1,\sigma}.$ More precisely, arguying as in the previous subsection, we find
\begin{equation}\label{range-ac}
	\begin{array}{l}
		\mathcal{R}\Big(d_{(\check{r}_{+},\check{r}_{-})}F\big(a,b_{\mathbf{m}}(a,c),c,0,0\big)\Big)=\Big(\mathtt{span}( g_{0,\mathbf{m},a,c})\Big)^{\perp}=\ker\Big(g\mapsto\big\langle g,g_{0,\mathbf{m},a,c}\big\rangle\Big),\vspace{0.2cm}\\ g_{0,\mathbf{m},a,c}(x)\triangleq\begin{pmatrix}
		2\pi\mathbf{m}\big(a-c\big)-\tfrac{1}{2\pi\mathbf{m}\big(b_{\mathbf{m}}(a,c)-a\big)}\\
		\tfrac{1}{2\pi\mathbf{m}\big(b_{\mathbf{m}}(a,c)-a\big)}
	\end{pmatrix}\sin(2\pi\mathbf{m}x),
\end{array}
\end{equation}
where the orthogonal is understood in the sense of the scalar product defined in \eqref{scalar product per}.\\
$\blacktriangleright$ \textbf{Transversality condition :} From \eqref{lin op}-\eqref{lin op 2}, we have
$$\partial_{b}d_{(\check{r}_+,\check{r}_-)}F(a,b,c,0,0)=\begin{pmatrix}
	\partial_{x} & 0\\
	0 & 0
\end{pmatrix}+\frac{1}{(b-a)^2}\begin{pmatrix}
\partial_{x}^{-1} & -\partial_{x}^{-1}\\
\partial_{x}^{-1} & -\partial_{x}^{-1}
\end{pmatrix}.$$
Hence,
$$\partial_{b}d_{(\check{r}_{+},\check{r}_{-})}F\big(a,b_{\mathbf{m}}(a,c),c,0,0\big)[\check{r}_{0,\mathbf{m},a,c}]=h_{0,\mathbf{m},a,c}^{(2)}-h_{0,\mathbf{m},a,c}^{(1)},$$
where
\begin{align*}
	h_{0,\mathbf{m},a,c}^{(1)}(x)&\triangleq2\pi\mathbf{m}\begin{pmatrix}
		2\pi\mathbf{m}\big(a-c\big)-\tfrac{1}{2\pi\mathbf{m}\big(b_{\mathbf{m}}(a,c)-a\big)}\\
		0
	\end{pmatrix}\sin(2\pi\mathbf{m}x),\\
h_{0,\mathbf{m},a,c}^{(2)}(x)&\triangleq\frac{a-c}{\big(b_{\mathbf{m}}(a,c)-a\big)^2}\begin{pmatrix}
	1\\
	1
\end{pmatrix}\sin(2\pi\mathbf{m}x).
\end{align*}
On one hand
\begin{align*}
	\left\langle h_{0,\mathbf{m},a,c}^{(1)}\,,\,g_{0,\mathbf{m},a,c}\right\rangle=\pi\mathbf{m}\left(2\pi\mathbf{m}\big(a-c\big)-\tfrac{1}{2\pi\mathbf{m}\big(b_{\mathbf{m}}(a,c)-a\big)}\right)^2.
\end{align*}
On the other hand
\begin{align*}
	\left\langle h_{0,\mathbf{m},a,c}^{(2)}\,,\,g_{0,\mathbf{m},a,c}\right\rangle=\frac{\pi\mathbf{m}(a-c)^2}{\big(b_{\mathbf{m}}(a,c)-a\big)^2}\cdot
\end{align*}
Hence, using \eqref{def bjac}, we obtain
\begin{align*}
	&\left\langle\partial_{b}d_{(\check{r}_{+},\check{r}_{-})}F\big(a,b_{\mathbf{m}}(a,c),c,0,0\big)[\check{r}_{0,\mathbf{m},a,c}]\,,\,g_{0,\mathbf{m},a,c}\right\rangle=\left\langle h_{0,\mathbf{m},a,c}^{(2)}-h_{0,\mathbf{m},a,c}^{(1)}\,,\,g_{0,\mathbf{m},a,c}\right\rangle\\
	&=\pi\mathbf{m}\left(\frac{(a-c)^2}{\big(b_{\mathbf{m}}(a,c)-a\big)^2}-4\pi^2\mathbf{m}^2(a-c)^2+2\frac{a-c}{b_{\mathbf{m}}(a,c)-a}-\frac{1}{4\pi^2\mathbf{m}^2\big(b_{\mathbf{m}}(a,c)-a\big)^2}\right)\\
	&=\frac{\pi\mathbf{m}(a-c)^2}{\big(b_{\mathbf{m}}(a,c)-a\big)^2}\Big(1-4\pi^2\mathbf{m}^2(a-c)^2\Big).
\end{align*}
Therefore, the condition \eqref{cond1} gives
\begin{equation}\label{sign dbdF-b}
	\begin{cases}
		\left\langle\partial_{b}d_{(\check{r}_{+},\check{r}_{-})}F\big(a,b_{\mathbf{m}}(a,c),c,0,0\big)[\check{r}_{0,\mathbf{m},a,c}]\,,\,g_{0,\mathbf{m},a,c}\right\rangle<0, & \textnormal{if }a<c,\\
		\left\langle\partial_{b}d_{(\check{r}_{+},\check{r}_{-})}F\big(a,b_{\mathbf{m}}(a,c),c,0,0\big)[\check{r}_{0,\mathbf{m},a,c}]\,,\,g_{0,\mathbf{m},a,c}\right\rangle>0, & \textnormal{if }a>c.
	\end{cases}
\end{equation}
In both cases
$$\left\langle\partial_{b}d_{(\check{r}_{+},\check{r}_{-})}F\big(a,b_{\mathbf{m}}(a,c),c,0,0\big)[\check{r}_{0,\mathbf{m},a,c}]\,,\,g_{0,\mathbf{m},a,c}\right\rangle\neq0,$$
which means
\begin{equation}\label{transversality-ac}
	\partial_{b}d_{(\check{r}_{+},\check{r}_{-})}F\big(a,b_{\mathbf{m}}(a,c),c,0,0\big)[\check{r}_{0,\mathbf{m},a,c}]\not\in\mathcal{R}\Big(d_{(\check{r}_{+},\check{r}_{-})}F\big(a,b_{\mathbf{m}}(a,c),c,0,0\big)\Big).
\end{equation}
Finally, \eqref{trivial line a-c}, \eqref{regularity F}, \eqref{Fredholmness}, \eqref{ker-ac} and \eqref{transversality-ac} allow to apply the Crandall-Rabinowitz Theorem \ref{thm CR+S} which concludes the existence part of Theorem \ref{thm E-states}-(ii). In the sequel, we denote
\begin{equation}\label{def local curve bm}
	\mathscr{C}_{\textnormal{\tiny{local}}}^{\mathbf{m}}(a,c):\mathtt{s}\in(-\delta,\delta)\mapsto \Big(b_{\mathbf{m}}(\mathtt{s},a,c)\,,\,\check{r}_{\mathbf{m}}(\mathtt{s},a,c)\Big)\in\mathbb{R}\times X_{\mathbf{m}}^{s,\sigma},\qquad\delta>0
\end{equation}
the corresponding (real-analytic) local curve which satisfies
\begin{equation}\label{frst orders exp bm}
	b_{\mathbf{m}}(0,a,c)=b_{\mathbf{m}}(a,c),\qquad\frac{d}{d\mathtt{s}}\Big(\check{r}_{\mathbf{m}}(\mathtt{s},a,c)\Big)\Big|_{\mathtt{s}=0}=\check{r}_{0,\mathbf{m},a,c}.
\end{equation}
$\blacktriangleright$ \textbf{Pitchfork-type bifurcation :}
As in the previous subsection, we can write $$d^2_{(\check{r}_{+},\check{r}_{-})}F\big(a,b_{\mathbf{m}}(a,c),c,0,0\big)[\check{r}_{0,\mathbf{m},a,c},\check{r}_{0,\mathbf{m},a,c}]=d_{(\check{r}_{+},\check{r}_{-})}F\big(a,b_{\mathbf{m}}(a,c),c,0,0\big)[\theta_{0,\mathbf{m},a,c}],$$
with
$$\theta_{0,\mathbf{m},a,c}(x)\triangleq2\pi\mathbf{m}M_{2\mathbf{m}}^{-1}\big(a,b_{\mathbf{m}}(a,c),c)\big)\begin{pmatrix}
	\left(2\pi\mathbf{m}\big(a-c\big)-\tfrac{1}{2\pi\mathbf{m}\big(b_{\mathbf{m}}(a,c)-a\big)}\right)^2\\
	\tfrac{1}{4\pi^2\mathbf{m}^2\big(b_{\mathbf{m}}(a,c)-a\big)^2}
\end{pmatrix}\cos(4\pi\mathbf{m}x).$$
Therefore, Theorem \ref{thm CR+S} applies and gives
\begin{equation*}
	\frac{d}{d\mathtt{s}}\Big(b_{\mathbf{m}}(\mathtt{s},a,c)\Big)\Big|_{\mathtt{s}=0}=0
\end{equation*}
and
\begin{equation}\label{second deriv b}
	\frac{d^2}{d\mathtt{s}^2}\Big(b_{\mathbf{m}}(\mathtt{s},a,c)\Big)\Big|_{\mathtt{s}=0}=\frac{\left\langle d_{(\check{r}_{+},\check{r}_{-})}^2F\big(a,b_{\mathbf{m}}(a,c),c,0,0\big)[\check{r}_{0,\mathbf{m},a,c},\theta_{0,\mathbf{m},a,c}]\,,\,g_{0,\mathbf{m},a,c}\right\rangle}{\left\langle \partial_b d_{(\check{r}_{+},\check{r}_{-})}F\big(a,b_{\mathbf{m}}(a,c),c,0,0\big)[\check{r}_{0,\mathbf{m},a,c}]\,,\,g_{0,\mathbf{m},a,c}\right\rangle}\cdot
\end{equation}
In addition, 
$$\left\langle d_{(\check{r}_{+},\check{r}_{-})}^2F\big(a,b_{\mathbf{m}}(a,c),c,0,0\big)[\check{r}_{0,\mathbf{m},a,c},\theta_{0,\mathbf{m},a,c}]\,,\,g_{0,\mathbf{m},a,c}\right\rangle=\frac{\pi^2\mathbf{m}^2}{3}\mathtt{h}(\mathbf{m},a,c),$$
where 
\begin{align*}
	\mathtt{h}(\mathbf{m},a,c)&\triangleq \alpha_{2\mathbf{m}}\big(a,b_{\mathbf{m}}(a,c),c\big)\gamma_{\mathbf{m}}^4\big(a,b_{\mathbf{m}}(a,c)\big)-\beta_{2\mathbf{m}}\big(a,b_{\mathbf{m}}(a,c),c\big)\beta_{\mathbf{m}}^4\big(a,b_{\mathbf{m}}(a,c),c\big)\\
	&\quad-2\gamma_{2\mathbf{m}}\big(a,b_{\mathbf{m}}(a,c)\big)\gamma_{\mathbf{m}}^2\big(a,b_{\mathbf{m}}(a,c)\big)\beta_{\mathbf{m}}^2\big(a,b_{\mathbf{m}}(a,c),c\big).
\end{align*}
After tedious calculations using \eqref{def bjac}, we can write
$$\mathtt{h}(\mathbf{m},a,c)=\frac{4\pi^3\mathbf{m}^3(a-c)^3}{\big(1-4\pi^2\mathbf{m}^2(a-c)^2\big)^5}\Big(\mathtt{h}_1\big(4\pi^2\mathbf{m}^2(a-c)^2\big)+4\Big),$$
with
\begin{equation}\label{def h1}
	\forall x\geqslant0,\quad \mathtt{h}_1(x)\triangleq4x^6-3x^5-2x^3-3x.
\end{equation}
Observe that
$$\forall x\geqslant0,\quad \mathtt{h}_1'(x)=24x^5-15x^4-6x^2-3=3(x-1)(8x^4+3x^3+3x^2+x+1).$$
We deduce that the fonction $\mathtt{h}_1$ is decreasing on $[0,1]$, increasing on $[1,\infty)$ and admits a global minimum on $[0,\infty)$ at $x=1$ with value $\mathtt{h}_1(1)=-4.$ Therefore, together with the constraint \eqref{cond1}, we deduce that
$$\mathtt{h}(\mathbf{m},a,c)>0.$$
This implies in turn
\begin{equation}\label{sign d2F-b}
	\left\langle d_{(\check{r}_{+},\check{r}_{-})}^2F\big(a,b_{\mathbf{m}}(a,c),c,0,0\big)[\check{r}_{0,\mathbf{m},a,c},\theta_{0,\mathbf{m},a,c}]\,,\,g_{0,\mathbf{m},a,c}\right\rangle>0.
\end{equation}
Combining \eqref{second deriv b}, \eqref{sign dbdF-b} and \eqref{sign d2F-b}, we obtain
$$\begin{cases}
	\frac{d^2}{d\mathtt{s}^2}\Big(b_{\mathbf{m}}(\mathtt{s},a,c)\Big)\Big|_{\mathtt{s}=0}<0, & \textnormal{if }a<c,\vspace{0.2cm}\\
	\frac{d^2}{d\mathtt{s}^2}\Big(b_{\mathbf{m}}(\mathtt{s},a,c)\Big)\Big|_{\mathtt{s}=0}>0, & \textnormal{if }a>c.
\end{cases}$$
This concludes the proof of Theorem \ref{thm E-states}-(ii).
\subsection{Bifurcation from $a$}
In this subsection, we fix $b,c\in\mathbb{R}$ with $b\neq c$ and study the bifurcation with respect to the parameter $a<b.$ This subsection is very similar to the previous one, but cannot be reduced by a symmetry argument because of the transversality and pitchfork analysis. Observe that \eqref{determinant Mj per} implies
\begin{equation}\label{def ajbc}
	\Delta_{j}(a,b,c)=0\qquad\Leftrightarrow\qquad a=a_j(b,c)\triangleq c+\frac{1}{4\pi^2j^2(b-c)}\cdot
\end{equation}
The constraint $a<b$ requires the following restriction
\begin{equation}\label{cond2}
	\begin{cases}
		4\pi^2j^2(b-c)^2>1, & \textnormal{if }b>c,\\
		4\pi^2j^2(b-c)^2<1, & \textnormal{if }b<c.
	\end{cases}
\end{equation}
\texttt{Case 1 :} If $b>c$, then there is a countable family of potential bifurcation points, namely
$$a_{\mathbf{m}}(b,c),\qquad\mathbf{m}\geqslant N_1(b-c).$$
\texttt{Case 2 :} If $b<c$, then there are two options. 
\begin{enumerate}[label=\textbullet]
	\item If $b\leqslant c-\tfrac{1}{2\pi}$, no bifurcation point exists.
	\item If $b>c-\tfrac{1}{2\pi}$ and there is a finite number of potential bifurcation points, namely
	$$a_{\mathbf{m}}(b,c),\qquad\mathbf{m}\in\big\llbracket1,N_2(c-b)\big\rrbracket.$$
\end{enumerate}
In what follows, the condition $b<c$ always refers to $c-\tfrac{1}{2\pi}<b<c.$\\

\noindent $\blacktriangleright$ \textbf{One dimensional kernel condition :} For $b>c$ (resp. $b<c$), the sequence $\big(a_{j}(b,c)\big)_{j\in\mathbb{N}^*}$ (well-defined) is decreasing (resp. increasing) and tends to $c$ as $j\to\infty.$ In addition, one has, for any fixed $\mathbf{m}\in\mathbb{N}^*$ with $\mathbf{m}\geqslant N_1(b-c)$ (resp. $\mathbf{m}\in\llbracket 1,N_2(c-b)\rrbracket$),
$$\Delta_{\mathbf{m}}\big(a_{\mathbf{m}}(b,c),b,c\big)=0$$
and, similarly to \eqref{sign det1},
\begin{equation}\label{detsgnbc}
	\forall j\in\mathbb{N}\setminus\{0,1\},\quad\Delta_{j\mathbf{m}}\big(a_{\mathbf{m}}(b,c),b,c\big)=j^2-1.
\end{equation}
Thus, in both cases, the kernel of $d_{(\check{r}_{+},\check{r}_{-})}F\big(a_{\mathbf{m}}(b,c),b,c,0,0\big)$ is one dimensional, more precisely 
\begin{equation}\label{ker-bc}
	\begin{array}{l}
		\ker\Big(d_{(\check{r}_{+},\check{r}_{-})}F\big(a_{\mathbf{m}}(b,c),b,c,0,0\big)\Big)=\mathtt{span}(\check{r}_{0,\mathbf{m},b,c}),\vspace{0.2cm}\\
		\check{r}_{0,\mathbf{m},b,c}(x)\triangleq\begin{pmatrix}
			2\pi\mathbf{m}\big(a_{\mathbf{m}}(b,c)-c\big)-\tfrac{1}{2\pi\mathbf{m}\big(b-a_{\mathbf{m}}(b,c)\big)}\\
			-\tfrac{1}{2\pi\mathbf{m}\big(b-a_{\mathbf{m}}(b,c)\big)}
		\end{pmatrix}\cos(2\pi \mathbf{m}x).
	\end{array}
\end{equation}
$\blacktriangleright$ \textbf{Range condition :} In both cases, $b\neq c$ and the definition of $a_{\mathbf{m}}(b,c)$ in \eqref{def ajbc} implies that $a_{\mathbf{m}}(b,c)\neq c.$ So the Fredholmness property \eqref{Fredholmness} gives that the range $\mathcal{R}\Big(d_{(\check{r}_{+},\check{r}_{-})}F\big(a_{\mathbf{m}}(b,c),b,c,0,0\big)\Big)$ is closed and of codimension one in $Y_{\mathbf{m}}^{s-1,\sigma}.$ More precisely, arguying as in the subsection \ref{sec vel bif}, we find
\begin{equation}\label{range-bc}
	\begin{array}{l}
		\mathcal{R}\Big(d_{(\check{r}_{+},\check{r}_{-})}F\big(a_{\mathbf{m}}(b,c),b,c,0,0\big)\Big)=\Big(\mathtt{span}( g_{0,\mathbf{m},b,c})\Big)^{\perp}=\ker\Big(g\mapsto\big\langle g,g_{0,\mathbf{m},b,c}\big\rangle\Big),\vspace{0.2cm}\\ g_{0,\mathbf{m},b,c}(x)\triangleq\begin{pmatrix}
			2\pi\mathbf{m}\big(a_{\mathbf{m}}(b,c)-c\big)-\tfrac{1}{2\pi\mathbf{m}\big(b-a_{\mathbf{m}}(b,c)\big)}\\
			\tfrac{1}{2\pi\mathbf{m}\big(b-a_{\mathbf{m}}(b,c)\big)}
		\end{pmatrix}\sin(2\pi\mathbf{m}x),
	\end{array}
\end{equation}
where the orthogonal is understood in the sense of the scalar product defined in \eqref{scalar product per}.\\
$\blacktriangleright$ \textbf{Transversality condition :} From \eqref{lin op}-\eqref{lin op 2}, we have
$$\partial_{a}d_{(\check{r}_+,\check{r}_-)}F(a,b,c,0,0)=\begin{pmatrix}
	0 & 0\\
	0 & \partial_{x}
\end{pmatrix}+\frac{1}{(b-a)^2}\begin{pmatrix}
	-\partial_{x}^{-1} & \partial_{x}^{-1}\\
	-\partial_{x}^{-1} & \partial_{x}^{-1}
\end{pmatrix}.$$
Hence,
$$\partial_{a}d_{(\check{r}_{+},\check{r}_{-})}F\big(a_{\mathbf{m}}(b,c),b,c,0,0\big)[\check{r}_{0,\mathbf{m},b,c}]=h_{0,\mathbf{m},b,c}^{(1)}-h_{0,\mathbf{m},b,c}^{(2)},$$
where
\begin{align*}
	h_{0,\mathbf{m},b,c}^{(1)}(x)&\triangleq\begin{pmatrix}
		0\\
		\tfrac{1}{b-a_{\mathbf{m}}(b,c)}
	\end{pmatrix}\sin(2\pi\mathbf{m}x),\qquad h_{0,\mathbf{m},b,c}^{(2)}(x)\triangleq\frac{a_{\mathbf{m}}(b,c)-c}{\big(b-a_{\mathbf{m}}(b,c)\big)^2}\begin{pmatrix}
		1\\
		1
	\end{pmatrix}\sin(2\pi\mathbf{m}x).
\end{align*}
We have
\begin{align*}
	\left\langle h_{0,\mathbf{m},b,c}^{(1)}\,,\,g_{0,\mathbf{m},b,c}\right\rangle=\frac{1}{4\pi\mathbf{m}\big(b-a_{\mathbf{m}}(b,c)\big)^2},\qquad\left\langle h_{0,\mathbf{m},b,c}^{(2)}\,,\,g_{0,\mathbf{m},b,c}\right\rangle=\frac{\pi\mathbf{m}\big(a_{\mathbf{m}}(b,c)-c\big)^2}{\big(b-a_{\mathbf{m}}(b,c)\big)^2}\cdot
\end{align*}
Hence, using \eqref{def ajbc}, we obtain
\begin{align*}
	&\left\langle\partial_{a}d_{(\check{r}_{+},\check{r}_{-})}F\big(a_{\mathbf{m}}(b,c),b,c,0,0\big)[\check{r}_{0,\mathbf{m},b,c}]\,,\,g_{0,\mathbf{m},b,c}\right\rangle=\left\langle h_{0,\mathbf{m},b,c}^{(1)}-h_{0,\mathbf{m},b,c}^{(2)}\,,\,g_{0,\mathbf{m},b,c}\right\rangle\\
	&=\frac{1}{4\pi\mathbf{m}\big(b-a_{\mathbf{m}}(b,c)\big)^2}\left(1-4\pi^2\mathbf{m}^2\big(a_{\mathbf{m}}(b,c)-c\big)^2\right)\\
	&=\frac{1}{16\pi^3\mathbf{m}^3(b-c)^2\big(b-a_{\mathbf{m}}(b,c)\big)^2}\Big(4\pi^2\mathbf{m}^2(b-c)^2-1\Big).
\end{align*}
Therefore, the condition \eqref{cond2} gives
\begin{equation}\label{sign dadF-a}
	\begin{cases}
		\left\langle\partial_{a}d_{(\check{r}_{+},\check{r}_{-})}F\big(a_{\mathbf{m}}(b,c),b,c,0,0\big)[\check{r}_{0,\mathbf{m},b,c}]\,,\,g_{0,\mathbf{m},b,c}\right\rangle>0, & \textnormal{if }b>c,\\
		\left\langle\partial_{a}d_{(\check{r}_{+},\check{r}_{-})}F\big(a_{\mathbf{m}}(b,c),b,c,0,0\big)[\check{r}_{0,\mathbf{m},b,c}]\,,\,g_{0,\mathbf{m},b,c}\right\rangle<0, & \textnormal{if }b<c.
	\end{cases}
\end{equation}
In both cases
$$\left\langle\partial_{a}d_{(\check{r}_{+},\check{r}_{-})}F\big(a_{\mathbf{m}}(b,c),b,c,0,0\big)[\check{r}_{0,\mathbf{m},b,c}]\,,\,g_{0,\mathbf{m},b,c}\right\rangle\neq0,$$
which means
\begin{equation}\label{transversality-bc}
	\partial_{a}d_{(\check{r}_{+},\check{r}_{-})}F\big(a_{\mathbf{m}}(b,c),b,c,0,0\big)[\check{r}_{0,\mathbf{m},b,c}]\not\in\mathcal{R}\Big(d_{(\check{r}_{+},\check{r}_{-})}F\big(a_{\mathbf{m}}(b,c),b,c,0,0\big)\Big).
\end{equation}
Finally, \eqref{trivial line a-c}, \eqref{regularity F}, \eqref{Fredholmness}, \eqref{ker-bc} and \eqref{transversality-bc} allow to apply the Crandall-Rabinowitz Theorem \ref{thm CR+S} which concludes the existence part of Theorem \ref{thm E-states}-(iii). In the sequel, we denote
\begin{equation}\label{def local curve am}
	\mathscr{C}_{\textnormal{\tiny{local}}}^{\mathbf{m}}(b,c):\mathtt{s}\in(-\delta,\delta)\mapsto \Big(a_{\mathbf{m}}(\mathtt{s},b,c)\,,\,\check{r}_{\mathbf{m}}(\mathtt{s},b,c)\Big)\in\mathbb{R}\times X_{\mathbf{m}}^{s,\sigma},\qquad\delta>0
\end{equation}
the corresponding (real-analytic) local curve which satisfies
\begin{equation}\label{frst orders exp am}
	a_{\mathbf{m}}(0,b,c)=a_{\mathbf{m}}(b,c),\qquad\frac{d}{d\mathtt{s}}\Big(\check{r}_{\mathbf{m}}(\mathtt{s},b,c)\Big)\Big|_{\mathtt{s}=0}=\check{r}_{0,\mathbf{m},b,c}.
\end{equation}
$\blacktriangleright$ \textbf{Pitchfork-type bifurcation :} As in the subsection \ref{sec vel bif}, we can write $$d^2_{(\check{r}_{+},\check{r}_{-})}F\big(a_{\mathbf{m}}(b,c),b,c,0,0\big)[\check{r}_{0,\mathbf{m},b,c},\check{r}_{0,\mathbf{m},b,c}]=d_{(\check{r}_{+},\check{r}_{-})}F\big(a_{\mathbf{m}}(b,c),b,c,0,0\big)[\theta_{0,\mathbf{m},b,c}],$$
with
$$\theta_{0,\mathbf{m},b,c}(x)\triangleq2\pi\mathbf{m}M_{2\mathbf{m}}^{-1}\big(a_{\mathbf{m}}(b,c),b,c)\big)\begin{pmatrix}
	\left(2\pi\mathbf{m}\big(a_{\mathbf{m}}(b,c)-c\big)-\tfrac{1}{2\pi\mathbf{m}\big(b-a_{\mathbf{m}}(b,c)\big)}\right)^2\\
	\tfrac{1}{4\pi^2\mathbf{m}^2\big(b-a_{\mathbf{m}}(b,c)\big)^2}
\end{pmatrix}\cos(4\pi\mathbf{m}x).$$
Therefore, Theorem \ref{thm CR+S} applies and gives
\begin{equation*}
	\frac{d}{d\mathtt{s}}\Big(a_{\mathbf{m}}(\mathtt{s},b,c)\Big)\Big|_{\mathtt{s}=0}=0
\end{equation*}
and
\begin{equation}\label{second deriv a}
	\frac{d^2}{d\mathtt{s}^2}\Big(a_{\mathbf{m}}(\mathtt{s},b,c)\Big)\Big|_{\mathtt{s}=0}=\frac{\left\langle d_{(\check{r}_{+},\check{r}_{-})}^2F\big(a_{\mathbf{m}}(b,c),b,c,0,0\big)[\check{r}_{0,\mathbf{m},b,c},\theta_{0,\mathbf{m},b,c}]\,,\,g_{0,\mathbf{m},b,c}\right\rangle}{\left\langle \partial_a d_{(\check{r}_{+},\check{r}_{-})}F\big(a_{\mathbf{m}}(b,c),b,c,0,0\big)[\check{r}_{0,\mathbf{m},b,c}]\,,\,g_{0,\mathbf{m},b,c}\right\rangle}\cdot
\end{equation}
In addition, 
$$\left\langle d_{(\check{r}_{+},\check{r}_{-})}^2F\big(a_{\mathbf{m}}(b,c),b,c,0,0\big)[\check{r}_{0,\mathbf{m},b,c},\theta_{0,\mathbf{m},b,c}]\,,\,g_{0,\mathbf{m},b,c}\right\rangle=\frac{\pi^2\mathbf{m}^2}{3}\widetilde{\mathtt{h}}(\mathbf{m},b,c),$$
where 
\begin{align*}
	\widetilde{\mathtt{h}}(\mathbf{m},b,c)&\triangleq \alpha_{2\mathbf{m}}\big(a_{\mathbf{m}}(b,c),b,c\big)\gamma_{\mathbf{m}}^4\big(a_{\mathbf{m}}(b,c),b\big)-\beta_{2\mathbf{m}}\big(a_{\mathbf{m}}(b,c),b,c\big)\beta_{\mathbf{m}}^4\big(a_{\mathbf{m}}(b,c),b,c\big)\\
	&\quad-2\gamma_{2\mathbf{m}}\big(a_{\mathbf{m}}(b,c),b\big)\gamma_{\mathbf{m}}^2\big(a_{\mathbf{m}}(b,c),b\big)\beta_{\mathbf{m}}^2\big(a_{\mathbf{m}}(b,c),b,c\big).
\end{align*}
After tedious calculations using \eqref{def ajbc}, we can write
$$\widetilde{\mathtt{h}}(\mathbf{m},b,c)=\frac{1}{64\pi^5\mathbf{m}^5(b-c)^5\big(4\pi^2\mathbf{m}^2(b-c)^2-1\big)^5}\Big(\mathtt{h}_1\big(4\pi^2\mathbf{m}^2(b-c)^2\big)+4\Big),$$
with $\mathtt{h}_1$ as in \eqref{def h1}. Therefore, together with the constraint \eqref{cond2} and the variations of $\mathtt{h}_1$ obtained in the previous subsection, we deduce that
$$\widetilde{\mathtt{h}}(\mathbf{m},a,c)>0.$$
This implies in turn
\begin{equation}\label{sign d2F-a}
	\left\langle d_{(\check{r}_{+},\check{r}_{-})}^2F\big(a_{\mathbf{m}}(b,c),b,c,0,0\big)[\check{r}_{0,\mathbf{m},b,c},\theta_{0,\mathbf{m},b,c}]\,,\,g_{0,\mathbf{m},b,c}\right\rangle>0.
\end{equation}
Combining \eqref{second deriv a}, \eqref{sign dadF-a} and \eqref{sign d2F-a}, we obtain
$$\begin{cases}
	\frac{d^2}{d\mathtt{s}^2}\Big(a_{\mathbf{m}}(\mathtt{s},b,c)\Big)\Big|_{\mathtt{s}=0}>0, & \textnormal{if }b>c,\vspace{0.2cm}\\
	\frac{d^2}{d\mathtt{s}^2}\Big(a_{\mathbf{m}}(\mathtt{s},b,c)\Big)\Big|_{\mathtt{s}=0}<0, & \textnormal{if }b<c.
\end{cases}$$
This concludes the proof of Theorem \ref{thm E-states}-(iii).
	\subsection{Area bifurcation from symmetric flat strips}
	In this subsection, we look for solutions close to the symmetric flat strip $S_{\textnormal{\tiny{flat}}}(-a,a).$ Hence the couple $(a,b)$ is replaced by the couple $(-a,a)$ with $a>0$ and the functional $F$ is replaced by
	$$G:(0,\infty)\times\mathbb{R}\times X_{\mathbf{m}}^{s,\sigma}\rightarrow Y_{\mathbf{m}}^{s-1,\sigma},\qquad G(a,c,\check{r}_+,\check{r}_-)\triangleq F(-a,a,c,\check{r}_+,\check{r}_-).$$
	In the sequel, refering to the notations \eqref{def Mj per} and \eqref{determinant Mj per}, we denote
	\begin{align}
		\widetilde{M}_{j}(a,c)&\triangleq M_j(-a,a,c)=\begin{pmatrix}
			2\pi j(a-c)+\tfrac{1}{4\pi ja} & -\tfrac{1}{4\pi ja}\\
			\tfrac{1}{4\pi ja} & -2\pi j(a+c)-\tfrac{1}{4\pi ja}
		\end{pmatrix}\triangleq\begin{pmatrix}
		\widetilde{\alpha}_{j}(a,c) & -\widetilde{\gamma}_{j}(a)\\
		\widetilde{\gamma}_{j}(a) & -\widetilde{\beta}_{j}(a,c)
	\end{pmatrix},\label{def Mj tilde}\\
	\widetilde{\Delta}_{j}(a,c)&\triangleq\Delta_{j}(-a,a,c)=-4\pi^2j^2(a^2-c^2)-1.\label{deter tilde}
	\end{align}
	Notice that the velocity bifurcation is a consequence of Section \ref{sec vel bif} but the bifurcation from $a$ differs from the previous studies. Remark that $2a$ corresponds to area of the flat strip $S_{\textnormal{\tiny{flat}}}(-a,a).$ We fix $c\in\mathbb{R}$ and study the bifurcation with respect to the parameter $a>0.$ Observe that for $c=0$, the relation \eqref{deter tilde} implies that for any $j\in\mathbb{N}^*,$ the matrix $\widetilde{M}_{j}(a,0)$ is always invertible whatever the value of $a>0.$ Hence, in the sequel, we shall restrict the discussion to the case $c\in\mathbb{R}^*.$ According to \eqref{deter tilde}, we have
	\begin{equation}\label{def aj}
		\widetilde{\Delta}_{j}(a,c)=0\qquad\Leftrightarrow\qquad a=a_{j}(c)\triangleq\sqrt{\frac{4\pi^{2}j^{2}c^{2}-1}{4\pi^{2}j^2}}\cdot
	\end{equation}
	Notice that $a_j(c)$ is well-defined only for
	$$j\geqslant N_1(c).$$
	$\blacktriangleright$ \textbf{One dimensional kernel condition :} The sequence $\big(a_{j}(c)\big)_{j\geqslant N_1(c)}$ is positive, increasing and converges to $|c|$ when $j\to\infty.$  In addition, for any fixed $\mathbf{m}\in\mathbb{N}^*$ with $\mathbf{m}\geqslant N_1(c),$ we have
	$$\widetilde{\Delta}_{\mathbf{m}}\big(a_{\mathbf{m}}(c),c\big)=0$$
	and
	\begin{equation}\label{detsgna}
		\forall j\in\mathbb{N}\setminus\{0,1\},\quad\widetilde{\Delta}_{j\mathbf{m}}\big(a_{\mathbf{m}}(c),c\big)=j^2-1>0.
	\end{equation}
	Thus, the kernel of $d_{(\check{r}_{+},\check{r}_{-})}G\big(a_{\mathbf{m}}(c),c,0,0\big)$ is one dimensional, more precisely 
	\begin{equation}\label{ker-c}
		\begin{array}{l}
			\ker\Big(d_{(\check{r}_{+},\check{r}_{-})}G\big(a_{\mathbf{m}}(c),c,0,0\big)\Big)=\mathtt{span}\big( \check{r}_{0,\mathbf{m},c}\big),\vspace{0.2cm}\\
			\check{r}_{0,\mathbf{m},c}(x)\triangleq\begin{pmatrix}
				2\pi\mathbf{m}\big(a_{\mathbf{m}}(c)+c\big)+\tfrac{1}{4\pi\mathbf{m}a_{\mathbf{m}}(c)}\\
				\tfrac{1}{4\pi\mathbf{m}a_{\mathbf{m}}(c)}
			\end{pmatrix}\cos(2\pi \mathbf{m}x).
		\end{array}
	\end{equation}
	$\blacktriangleright$ \textbf{Range condition :} Notice that the monotonicity and the convergence of the sequence $\big(a_{j}(c)\big)_{j\geqslant N_1(c)}$ imply that $|c|\neq a_{\mathbf{m}}(c).$ So the Fredholmness property \eqref{Fredholmness} gives that the range $\mathcal{R}\Big(d_{(\check{r}_{+},\check{r}_{-})}G\big(a_{\mathbf{m}}(c),c,0,0\big)\Big)$ is closed and of codimension one in $Y_{\mathbf{m}}^{s-1,\sigma}.$ More precisely, arguying as in Section \ref{sec vel bif}, we find
	\begin{equation}\label{range-c}
		\begin{array}{l}
			\mathcal{R}\Big(d_{(\check{r}_{+},\check{r}_{-})}G\big(a_{\mathbf{m}}(c),c,0,0\big)\Big)=\Big(\mathtt{span}\big( g_{0,\mathbf{m},c}\big)\Big)^{\perp}=\ker\Big(g\mapsto\big\langle g,g_{0,\mathbf{m},c}\big\rangle\Big),\vspace{0.2cm}\\
			g_{0,\mathbf{m},c}(x)\triangleq\begin{pmatrix}
				2\pi\mathbf{m}\big(a_{\mathbf{m}}(c)+c\big)+\tfrac{1}{4\pi\mathbf{m}a_{\mathbf{m}}(c)}\\
				-\tfrac{1}{4\pi\mathbf{m}a_{\mathbf{m}}(c)}
			\end{pmatrix}\sin(2\pi\mathbf{m}x),
		\end{array} 
	\end{equation}
	where the orthogonal is understood in the sense of the scalar product defined in \eqref{scalar product per}.\\ $\blacktriangleright$ \textbf{Transversality condition :} From \eqref{lin op}-\eqref{lin op 2}, we have
	$$\partial_{a}d_{(\check{r}_+,\check{r}_-)}G(a,c,0,0)=\begin{pmatrix}
		\partial_{x} & 0\\
		0 & -\partial_{x}
	\end{pmatrix}+\tfrac{1}{2a^2}\begin{pmatrix}
		\partial_{x}^{-1} & -\partial_{x}^{-1}\\
		\partial_{x}^{-1} & -\partial_{x}^{-1}
	\end{pmatrix}.$$
	Hence, direct calculations using in particular \eqref{inv dx} give
	$$\partial_{a}d_{(\check{r}_{+},\check{r}_{-})}G\big(a_{\mathbf{m}}(c),c,0,0\big)[\check{r}_{0,\mathbf{m},c}]=h_{0,\mathbf{m},c}-2\pi\mathbf{m}\,g_{0,\mathbf{m},c},$$
	where
	$$h_{0,\mathbf{m},c}(x)\triangleq \frac{a_{\mathbf{m}}(c)+c}{2a_{\mathbf{m}}^{2}(c)}\begin{pmatrix}
		1\\
		1
	\end{pmatrix}\sin(2\pi\mathbf{m}x).$$
	Hence, one has the equivalence
	\begin{align}\label{cond range}
		&\partial_{a}d_{(\check{r}_{+},\check{r}_{-})}G\big(a_{\mathbf{m}}(c),c,0,0\big)[\check{r}_{0,\mathbf{m},c}]\in\mathcal{R}\Big(d_{(\check{r}_{+},\check{r}_{-})}G\big(a_{\mathbf{m}}(c),c,0,0\big)\Big)\nonumber\\
		\quad\Leftrightarrow\quad &2\pi\mathbf{m}\big\langle g_{0,\mathbf{m},c}\,,\,g_{0,\mathbf{m},c}\big\rangle=\big\langle h_{0,\mathbf{m},c}\,,\,g_{0,\mathbf{m},c}\big\rangle.
	\end{align}
	Now, on one hand, we have
	\begin{align*}
		\big\langle h_{0,\mathbf{m},c}\,,\,g_{0,\mathbf{m},c}\big\rangle&=\frac{a_{\mathbf{m}}(c)+c}{4a_{\mathbf{m}}^2(c)}\left\langle\begin{pmatrix}
			1\\
			1
		\end{pmatrix},\begin{pmatrix}
			2\pi\mathbf{m}\big(a_{\mathbf{m}}(c)+c\big)+\tfrac{1}{4\pi\mathbf{m}a_{\mathbf{m}}(c)}\\
			-\tfrac{1}{4\pi\mathbf{m}a_{\mathbf{m}}(c)}
		\end{pmatrix}\right\rangle_{\mathbb{R}^2}\\
		&=\frac{\pi\mathbf{m}\big(a_{\mathbf{m}}(c)+c\big)^2}{2a_{\mathbf{m}}^2(c)}\cdot
	\end{align*}
	On the other hand,
	\begin{align*}
		2\pi\mathbf{m}\big\langle g_{0,\mathbf{m},c}\,,\,g_{0,\mathbf{m},c}\big\rangle&=\pi\mathbf{m}\left[\left(2\pi\mathbf{m}\big(a_{\mathbf{m}}(c)+c\big)+\frac{1}{4\pi\mathbf{m}a_{\mathbf{m}}(c)}\right)^2+\left(\frac{1}{4\pi\mathbf{m}a_{\mathbf{m}}(c)}\right)^2\right]\\
		&=\frac{\pi\mathbf{m}}{a_{\mathbf{m}}^{2}(c)}\left[4\pi^{2}\mathbf{m}^2a_{\mathbf{m}}^{2}(c)\big(a_{\mathbf{m}}(c)+c\big)^2+a_{\mathbf{m}}(c)\big(a_{\mathbf{m}}(c)+c\big)+\frac{1}{8\pi^{2}\mathbf{m}^2}\right].
	\end{align*}
	Therefore, 
	\begin{align}\label{inter eq per a}
		\big\langle& h_{0,\mathbf{m},c}-2\pi\mathbf{m}g_{0,\mathbf{m},c}\,,\,g_{0,\mathbf{m},c}\big\rangle\nonumber\\
		&=\frac{\pi\mathbf{m}}{a_{\mathbf{m}}^{2}(c)}\left[\frac{1}{2}\big(a_{\mathbf{m}}(c)+c\big)^2-4\pi^{2}\mathbf{m}^2a_{\mathbf{m}}^{2}(c)\big(a_{\mathbf{m}}(c)+c\big)^2-a_{\mathbf{m}}(c)\big(a_{\mathbf{m}}(c)+c\big)-\frac{1}{8\pi^{2}\mathbf{m}^2}\right]\nonumber\\
		&=-\frac{\pi\mathbf{m}}{a_{\mathbf{m}}^{2}(c)}\left[4\pi^{2}\mathbf{m}^{2}a_{\mathbf{m}}^{2}(c)\big(a_{\mathbf{m}}(c)+c\big)^{2}+\frac{a_{\mathbf{m}}^{2}(c)}{2}-\frac{c^2}{2}+\frac{1}{8\pi^{2}\mathbf{m}^{2}}\right].
	\end{align}
	Now, the definition \eqref{def aj} implies
	$$\frac{a_{\mathbf{m}}^{2}(c)}{2}-\frac{c^2}{2}+\frac{1}{8\pi^{2}\mathbf{m}^{2}}=0.$$
	Plugging this identity into \eqref{inter eq per a} and using the fact that $a_{\mathbf{m}}(c)\neq|c|$, yields
	\begin{align}\label{sgn dadFag}
		\big\langle\partial_{a}d_{(\check{r}_{+},\check{r}_{-})}G\big(a_{\mathbf{m}}(c),c,0,0\big)[\check{r}_{0,\mathbf{m},c}]\,,\,g_{0,\mathbf{m},c}\big\rangle&=\big\langle h_{0,\mathbf{m},c}-2\pi\mathbf{m}\,g_{0,\mathbf{m},c}\,,\,g_{0,\mathbf{m},c}\big\rangle\nonumber\\
		&=-4\pi^3\mathbf{m}^3\big(c+a_{\mathbf{m}}(c)\big)^2<0.
	\end{align}
	Consequently, the relation \eqref{cond range} is not satisfied, that is
	\begin{equation}\label{transversality-c}
		\partial_{a}d_{(\check{r}_{+},\check{r}_{-})}G\big(a_{\mathbf{m}}(c),c,0,0\big)[\check{r}_{0,\mathbf{m},c}]\not\in\mathcal{R}\Big(d_{(\check{r}_{+},\check{r}_{-})}G\big(a_{\mathbf{m}}(c),c,0,0\big)\Big).
	\end{equation}
	Finally, \eqref{trivial line a-c}, \eqref{regularity F}, \eqref{Fredholmness}, \eqref{ker-c} and \eqref{transversality-c} allow to apply the Crandall-Rabinowitz Theorem \ref{thm CR+S} which concludes the existence part of Theorem \ref{thm E-states}-(iv). In the sequel, we denote
	$$\mathscr{C}_{\textnormal{\tiny{local}}}^{\mathbf{m}}(c):\mathtt{s}\in(-\delta,\delta)\mapsto \Big(a_{\mathbf{m}}(\mathtt{s},c)\,,\,\check{r}_{\mathbf{m}}(\mathtt{s},c)\Big)\in\mathbb{T}\times X_{\mathbf{m}}^{s,\sigma},\qquad\delta>0$$
	the corresponding (real-analytic) local curve which satisfies
	$$a_{\mathbf{m}}(0,c)=a_{\mathbf{m}}(c),\qquad\frac{d}{d\mathtt{s}}\Big(\check{r}_{\mathbf{m}}(\mathtt{s},c)\Big)\Big|_{\mathtt{s}=0}=\check{r}_{0,\mathbf{m},c}.$$
	$\blacktriangleright$ \textbf{Pitchfork-type bifurcation :} Proceeding as in subsection \ref{sec vel bif}, we can write
	$$d_{(\check{r}_{+},\check{r}_{-})}G\big(a_{\mathbf{m}}(c),c,0,0\big)[\theta_{0,\mathbf{m},c}]=d_{(\check{r}_+,\check{r}_-)}^{2}G\big(a_{\mathbf{m}}(c),c,0,0\big)[\check{r}_{0,\mathbf{m},c},\check{r}_{0,\mathbf{m},c}],$$
	with
	\begin{align*}
		&\theta_{0,\mathbf{m},c}(x)\triangleq2\pi\mathbf{m}\widetilde{M}_{2\mathbf{m}}^{-1}\big(a_{\mathbf{m}}(c),c\big)\begin{pmatrix}
			\Big(2\pi\mathbf{m}\big(c+a_{\mathbf{m}}(c)\big)+\tfrac{1}{4a_{\mathbf{m}}(c)\pi\mathbf{m}}\Big)^2\\
			\tfrac{1}{16a_{\mathbf{m}}^2(c)\pi^2\mathbf{m}^2}
		\end{pmatrix}\cos(4\pi\mathbf{m}x).
	\end{align*}
	Then, according to Theorem \ref{thm CR+S}, and \eqref{diff 2 and 3}, we have
	$$\frac{d}{d\mathtt{s}}\Big(a_{\mathbf{m}}(\mathtt{s},c)\Big)\Big|_{\mathtt{s}=0}=0$$
	and
	\begin{equation}\label{param curv diff01}
		\frac{d^2}{d\mathtt{s}^2}\Big(a_{\mathbf{m}}(\mathtt{s},c)\Big)\Big|_{\mathtt{s}=0}=\frac{\left\langle d_{(\check{r}_{+},\check{r}_{-})}^2G\big(a_{\mathbf{m}}(c),c,0,0\big)[\check{r}_{0,\mathbf{m},c},\theta_{0,\mathbf{m},c}]\,,\,g_{0,\mathbf{m},c}\right\rangle}{\left\langle\partial_c d_{(\check{r}_{+},\check{r}_{-})}G\big(a_{\mathbf{m}}(c),c,0,0\big)[\check{r}_{0,\mathbf{m},c}]\,,\,g_{0,\mathbf{m},c}\right\rangle}\cdot
	\end{equation}
	Proceeding as in the subsection \ref{sec vel bif}, we obtain
	\begin{equation}\label{d2rtga}
		\left\langle d_{(\check{r}_{+},\check{r}_{-})}^2G\big(a_{\mathbf{m}}(c),c,0,0\big)[\check{r}_{0,\mathbf{m},c},\theta_{0,\mathbf{m},c}]\,,\,g_{0,\mathbf{m},c}\right\rangle=\frac{\pi^2\mathbf{m}^2}{3}\mathtt{f}(\mathbf{m},c),
	\end{equation}
	where $\mathtt{f}$ is a well-defined continuous fonction on $D_{\mathtt{f}}\triangleq\{(z,c)\in[1,\infty)\times\mathbb{R}^*\quad\textnormal{s.t.}\quad z\geqslant N_1(c)\}$ given by
	$$\mathtt{f}(z,c)\triangleq\Big[\widetilde{\alpha}_{2z}\big(a_{z}(c),c\big)\widetilde{\gamma}_{z}^4\big(a_{z}(c)\big)+\widetilde{\beta}_{2z}\big(a_{z}(c),c\big)\widetilde{\beta}_{z}^4\big(a_{z}(c),c\big)-2\widetilde{\gamma}_{2z}\big(a_{z}(c)\big)\widetilde{\gamma}_{z}^2\big(a_{z}(c)\big)\widetilde{\beta}_{z}^2\big(a_{z}(c),c\big)\Big].$$
	We warn the reader about the change of notation for the coefficients, which explains the different shape of $\mathtt{f}$ compared to the other subsections. After tedious calculations, we can write
	\begin{equation}\label{exp fzc}
		\mathtt{f}(z,c)=\frac{1}{128\pi^5 z^5a_{z}^5(c)}\Big(C(z,c)a_z(c)+D(z,c)\Big),
	\end{equation}
	where
	\begin{align*}
		C(z,c)&\triangleq131072\pi^{10}z^{10}c^9-71680\pi^8z^8c^7+13056\pi^6z^6c^5-896\pi^4z^4c^3+16\pi^2z^2c,\\
		D(z,c)&\triangleq131072\pi^{10}z^{10}c^{10}-88064\pi^8z^8c^8+20992\pi^6z^6c^6-2096\pi^4z^4c^4+80\pi^2z^2c^2-1.
	\end{align*}
	Assume, for the sake of contradiction, that there exists $(z,c)\in D_{\mathtt{f}}$ such that $\mathtt{f}(z,c)=0.$ Then this equation is equivalent to
	$$a_z(c)=-\frac{D(z,a)}{C(z,a)}\cdot$$
	Taking the square of the previous expression and using \eqref{def aj}, we end up with
	\begin{align*}
		&1048576\pi^{12}z^{12}c^{12}+284928\pi^{8}z^{8}c^{8}+3168\pi^{4}z^{4}c^{4}+1=884736\pi^{10}z^{10}c^{10}+43520\pi^{6}z^{6}c^{6}+96\pi^{2}z^{2}c^{2}.
	\end{align*}
	Now by construction $z\geqslant N_1(c)$, which implies in particular $2\pi|c|z\geqslant1.$ Consequently, the previous equation cannot be satisfied and we obtain
	$$\forall(z,c)\in D_{\mathtt{f}},\quad \mathtt{f}(z,c)\neq 0.$$
	Besides, the convergence of $a_{z}(c)$ to $|c|$ when $z$ is large together with \eqref{exp fzc} provide the following asymptotics
	$$\forall c>0,\quad \mathtt{f}(z,c)\underset{z\to\infty}{\sim}2048\pi^{5}z^5c^5\qquad\textnormal{and}\qquad\forall c<0,\quad \mathtt{f}(z,c)\underset{z\to\infty}{\sim}-128\pi^{3}z^{3}|c|^{3}.$$
	Therefore, by a continuity argument, we obtain
	$$\forall(z,c)\in D_{\mathtt{f}},\quad \begin{cases}
		\mathtt{f}(z,c)>0, & \textnormal{if }c>0,\\
		\mathtt{f}(z,c)<0, & \textnormal{if }c<0.
	\end{cases}$$
	Added to \eqref{d2rtga} and \eqref{detsgna}, we obtain
	\begin{equation}\label{sgnd2Fag}
		\begin{cases}
			\left\langle d_{(\check{r}_{+},\check{r}_{-})}^2G\big(a_{\mathbf{m}}(c),c,0,0\big)[\check{r}_{0,\mathbf{m},c},\theta_{0,\mathbf{m},c}]\,,\,g_{0,\mathbf{m},c}\right\rangle>0, & \textnormal{if }c>0,\vspace{0.2cm}\\
			\left\langle d_{(\check{r}_{+},\check{r}_{-})}^2G\big(a_{\mathbf{m}}(c),c,0,0\big)[\check{r}_{0,\mathbf{m},c},\theta_{0,\mathbf{m},c}]\,,\,g_{0,\mathbf{m},c}\right\rangle<0, & \textnormal{if }c<0.
		\end{cases}
	\end{equation}
	Plugging \eqref{sgn dadFag} and \eqref{sgnd2Fag} into \eqref{param curv diff01} yields
	$$\forall c>0,\quad\frac{d^2}{d\mathtt{s}^2}\Big(a_{\mathbf{m}}(\mathtt{s},c)\Big)\Big|_{\mathtt{s}=0}<0\qquad\textnormal{and}\qquad\forall c<0,\quad\frac{d^2}{d\mathtt{s}^2}\Big(a_{\mathbf{m}}(\mathtt{s},c)\Big)\Big|_{\mathtt{s}=0}>0.$$
	This ends the proof of Theorem \ref{thm E-states}-(iv).
	\section{Large amplitude solutions}
	The scope of this section is to implement the analytic global bifurcation Theorem \ref{thm BT} in order to continue the local branches constructed in the previous section. We first discuss some qualitative properties for a solution and then implement the global bifurcation theory.
	\subsection{Qualitative properties of generic solutions}\label{sec quali}
	Let us fix $(a,b,c)\in\mathbb{S}\times\mathbb{R}$ and consider $(\check{r}_+,\check{r}_-)$ a real-analytic non-trivial solution to the system \eqref{system rpm per}. Notice that, in order to have an electron layer, we must have
	$$\forall(t,x)\in\mathbb{R}_+\times\mathbb{T},\quad v_+(t,x)>v_-(t,x),$$
	which is equivalent to
	\begin{equation}\label{electron layer constraint}
		\forall x\in\mathbb{T},\quad\check{r}_+(x)+b-c>\check{r}_-(x)+a-c.
	\end{equation}
		Substracting both equations in \eqref{system rpm per} yields
		$$\big(\check{r}_{-}(x)+a-c\big)\partial_{x}\check{r}_-(x)=\big(\check{r}_{+}(x)+b-c\big)\partial_{x}\check{r}_+(x).$$
		From the previous two relations, we deduce that for any $\overline{x},\underline{x}\in\mathbb{T},$ we have
		\begin{align*}
			&\check{r}_+(\overline{x})+b-c=0\quad\Rightarrow\quad \Big(\check{r}_-(\overline{x})+a-c<0\quad\textnormal{ and }\quad\partial_{x}\check{r}_-(\overline{x})=0\Big),\\
			&\check{r}_-(\underline{x})+a-c=0\quad\Rightarrow\quad \Big(\check{r}_+(\underline{x})+b-c>0\quad\textnormal{ and }\quad\partial_{x}\check{r}_+(\underline{x})=0\Big).
		\end{align*}
	Let us assume that there exists $\overline{x}\in\mathbb{T}$ such that $\check{r}_+(\overline{x})+b-c=0$ and consider the following integers (finite by real-analyticity and zero average condition) 
	\begin{align*}
		\overline{q}_0^{+}&\triangleq\min\big\{q\in\mathbb{N}^*\quad\textnormal{s.t.}\quad\partial_{x}^{q}\check{r}_+(\overline{x})\neq 0\big\}\in\mathbb{N}^*,\\
		\overline{q}_0^{-}&\triangleq\min\big\{q\in\mathbb{N}^*\quad\textnormal{s.t.}\quad\partial_{x}^{q}\check{r}_-(\overline{x})\neq 0\big\}\in\mathbb{N}\setminus\{0,1\}.
	\end{align*}
Using Leibniz rule, we obtain for any $q\in\mathbb{N},$
\begin{equation}\label{system diff}
	\begin{cases}
		\big(\check{r}_+(x)+b-c\big)\partial_{x}^{q+1}\check{r}_+(x)+\displaystyle\sum_{k=1}^{q}\binom{q}{k}\partial_{x}^{k}\check{r}_+(x)\partial_{x}^{q-k+1}\check{r}_+(x)-\tfrac{1}{b-a}\partial_{x}^{q-1}\big(\check{r}_+(x)-\check{r}_-(x)\big)=0,\\
		\big(\check{r}_-(x)+a-c\big)\partial_{x}^{q+1}\check{r}_-(x)+\displaystyle\sum_{k=1}^{q}\binom{q}{k}\partial_{x}^{k}\check{r}_-(x)\partial_{x}^{q-k+1}\check{r}_-(x)-\tfrac{1}{b-a}\partial_{x}^{q-1}\big(\check{r}_+(x)-\check{r}_-(x)\big)=0.
	\end{cases}
\end{equation}
Substracting both equations in \eqref{system diff}, we obtain 
$$\forall q\in\mathbb{N}^*,\quad\partial_{x}^{q+1}\check{r}_-(\overline{x})=\frac{1}{\check{r}_-(\overline{x})+a-c}\sum_{k=1}^{q}\binom{q}{k}\Big[\partial_{x}^{k}\check{r}_+(\overline{x})\partial_{x}^{q-k+1}\check{r}_+(\overline{x})-\partial_{x}^{k}\check{r}_-(\overline{x})\partial_{x}^{q-k+1}\check{r}_-(\overline{x})\Big].$$
From this, we readily infer
$$\overline{q}_0^-=\begin{cases}
	\overline{q}_0^++1, & \textnormal{if }\overline{q}_0^+\equiv1[2],\\
	\overline{q}_0^++2, & \textnormal{if }\overline{q}_0^+\equiv0[2].
\end{cases}$$
By a similar argument, for any $\underline{x}\in\mathbb{T}$ with $\check{r}_-(\underline{x})+a-c=0,$ then denoting
\begin{align*}
	\underline{q}_0^{+}&\triangleq\min\big\{q\in\mathbb{N}^*\quad\textnormal{s.t.}\quad\partial_{x}^{q}\check{r}_+(\underline{x})\neq 0\big\}\in\mathbb{N}\setminus\{0,1\},\\
	\underline{q}_0^{-}&\triangleq\min\big\{q\in\mathbb{N}^*\quad\textnormal{s.t.}\quad\partial_{x}^{q}\check{r}_-(\underline{x})\neq 0\big\}\in\mathbb{N}^*,
\end{align*}
we have
$$\underline{q}_0^+=\begin{cases}
	\underline{q}_0^-+1, & \textnormal{if }\underline{q}_0^-\equiv1[2],\\
	\underline{q}_0^-+2, & \textnormal{if }\underline{q}_0^-\equiv0[2].
\end{cases}$$
	\subsection{Global bifurcation}
	Now, we prove the Theorem \ref{thm global E-states}. We denote
	$$\mathtt{m}(a,b)\triangleq\min_{x\in\mathbb{T}}\big|\check{r}_+(x)-\check{r}_{-}(x)+b-a\big|,\qquad\mathtt{m}_{+}(b,c)\triangleq\min_{x\in\mathbb{T}}\big|\check{r}_+(x)+b-c\big|,\qquad\mathtt{m}_{-}(a,c)\triangleq\min_{x\in\mathbb{T}}\big|\check{r}_-(x)+a-c\big|.$$
	Due to the similarity of the argument, we shall mainly focus on proving the global bifurcation from the velocity parameter $c$ when the other parameters $a$ and $b$ are fixed. We introduce the following open set
	$$U(a,b)\triangleq\Big\{(c,\check{r}_+,\check{r}_-)\in\mathbb{R}\times X_{\mathbf{m}}^{s,\sigma}\quad\textnormal{s.t.}\quad\min\big(\mathtt{m}(a,b),\mathtt{m}_{+}(b,c),\mathtt{m}_{-}(a,c)\big)>0\Big\}.$$
	In the sequel, we shall denote for any $r>0$
	$$B_{\mathbf{m}}^{s,\sigma}(r)\triangleq \Big\{f\in X_{\mathbf{m}}^{s,\sigma}\quad\textnormal{s.t.}\quad\|f\|_{s,\sigma}\leqslant r\Big\}.$$
	Let us consider the following closed and bounded set defined for any $n\in\mathbb{N}^*$ by
	$$F_{n}(a,b)\triangleq\Big\{(c,\check{r}_+,\check{r}_-)\in[-n,n]\times B_{\mathbf{m}}^{s,\sigma}(n)\quad\textnormal{s.t.}\quad\min\big(\mathtt{m}(a,b),\mathtt{m}_{+}(b,c),\mathtt{m}_{-}(a,c)\big)\geqslant\tfrac{1}{n}\Big\}.$$
	Obviously, one has
	$$U(a,b)=\bigcup_{n\in\mathbb{N}^*}F_n(a,b).$$
	We denote
	$$\mathscr{S}_{n}(a,b)\triangleq\Big\{(c,\check{r}_+,\check{r}_-)\in F_n(a,b)\quad\textnormal{s.t.}\quad F(a,b,c,\check{r}_+,\check{r}_-)=0\Big\}.$$
	\begin{lem}\label{lem compactness global bif}
		Let $a<b$, $s>\tfrac{3}{2}$ and $\sigma>0.$ The following properties hold true.
		\begin{enumerate}[label=(\roman*)]
			\item We have the inclusion $\mathscr{C}_{\textnormal{\tiny{local}}}^{\pm,\mathbf{m}}(a,b)\subset U(a,b).$
			\item For any $(c,\check{r}_+,\check{r}_-)\in U(a,b)$ with $F(a,b,c,\check{r}_+,\check{r}_-)=0,$ the operator $d_{(\check{r}_+,\check{r}_-)}F(a,b,c,\check{r}_+,\check{r}_-):X_{\mathbf{m}}^{s,\sigma}\rightarrow Y_{\mathbf{m}}^{s-1,\sigma}$ is Fredholm with index zero.
			\item For any $n\in\mathbb{N}^*$, the set $\mathscr{S}_n(a,b)$ is compact in $\mathbb{R}\times X_{\mathbf{m}}^{s,\sigma}.$
		\end{enumerate}
	\end{lem}
\begin{proof}
	\textbf{(i)} Let $\kappa\in\{+,-\}.$ Since $a<b$ and $c_{\mathbf{m}}^{\kappa}(a,b)\not\in\{a,b\},$ then, up to taking $\delta$ small enough we get
	\begin{align*}
		&\left|\big(\check{r}_{\mathbf{m}}^{\kappa}(\mathtt{s},a,b)\big)_+(x)-\big(\check{r}_{\mathbf{m}}^{\kappa}(\mathtt{s},a,b)\big)_-(x)+b-a\right|\geqslant\left|b-a\right|-C\delta>0,\\
		&\left|\big(\check{r}_{\mathbf{m}}^{\kappa}(\mathtt{s},a,b)\big)_+(x)+b-c_{\mathbf{m}}^{\kappa}(\mathtt{s},a,b)\right|\geqslant\left|b-c_{\mathbf{m}}^{\kappa}(a,b)\right|-C\delta>0,\\
		&\left|\big(\check{r}_{\mathbf{m}}^{\kappa}(\mathtt{s},a,b)\big)_-(x)+a-c_{\mathbf{m}}^{\kappa}(\mathtt{s},a,b)\right|\geqslant \left|a-c_{\mathbf{m}}^{\kappa}(a,b)\right|-C\delta>0.
	\end{align*}
This proves the inclusion $\mathscr{C}_{\textnormal{\tiny{local}}}^{\kappa,\mathbf{m}}(a,b)\subset U(a,b).$\\
	\textbf{(ii)} Let $(c,\check{r}_+,\check{r}_-)\in U(a,b)$ with $F(a,b,c,\check{r}_+,\check{r}_-)=0.$ Differentiating \eqref{system rpm per}, we can write
	$$d_{(\check{r}_+,\check{r}_-)}F(a,b,c,\check{r}_+,\check{r}_-)=I_{(\check{r}_+,\check{r}_-)}+K_{(\check{r}_+,\check{r}_-)},$$
	where
	$$I_{(\check{r}_+,\check{r}_-)}\triangleq\begin{pmatrix}
		(\check{r}_++b-c)\partial_{x} & 0\\
		0 & (\check{r}_-+a-c)\partial_{x}
	\end{pmatrix}$$
and
$$K_{(\check{r}_+,\check{r}_-)}\triangleq M_{(\check{r}_+,\check{r}_-)}+K_{(0,0)},\qquad M_{(\check{r}_+,\check{r}_-)}\triangleq\begin{pmatrix}
	\partial_{x}\check{r}_+ & 0\\
	0 & \partial_x\check{r}_-
\end{pmatrix}.$$
Since $(c,\check{r}_+,\check{r}_-)\in U(a,b)$, then in particular
$$\forall x\in\mathbb{T},\quad\check{r}_+(x)+b-c\neq0\qquad\textnormal{and}\qquad\check{r}_-(x)+a-c\neq0.$$
This implies that the operator $I_{(\check{r}_+,\check{r}_-)}:X_{\mathbf{m}}^{s,\sigma}\rightarrow Y_{\mathbf{m}}^{s-1,\sigma}$ is an isomorphism. Now, recall that the compactness of $K_{(0,0)}:X_{\mathbf{m}}^{s,\sigma}\rightarrow Y_{\mathbf{m}}^{s-1,\sigma}$ has already been proved at the beginning of Section \ref{sec local}. Let us now prove the compactness of $M_{(\check{r}_+,\check{r}_-)}:X_{\mathbf{m}}^{s,\sigma}\rightarrow Y_{\mathbf{m}}^{s-1,\sigma}.$ For this aim, consider $\left(x_+^{[k]},x_-^{[k]}\right)_{k\in\mathbb{N}}$ bounded in $X_{\mathbf{m}}^{s,\sigma}.$ There exist $(k_m)_{m\in\mathbb{N}}\in\mathbb{N}^{\mathbb{N}}$ increasing and $\left(x_{+}^{[\infty]},x_{-}^{[\infty]}\right)\in X_{\mathbf{m}}^{s-1,\sigma}$ such that
$$\left(x_{+}^{[k_m]},x_{-}^{[k_m]}\right)\underset{m\to\infty}{\longrightarrow}\left(x_{+}^{[\infty]},x_{-}^{[\infty]}\right)\quad \textnormal{in } X_{\mathbf{m}}^{s-1,\sigma}.$$
We denote, for any $m\in\mathbb{N},$
$$\left(y_{+}^{[m]},y_{-}^{[m]}\right)\triangleq M_{(\check{r}_+,\check{r}_-)}\left(x_{+}^{[k_m]},x_{-}^{[k_m]}\right)=\left(\partial_{x}\check{r}_+\,x_{+}^{[k_m]},\partial_{x}\check{r}_-\,x_{-}^{[k_m]}\right)\in Y_{\mathbf{m}}^{s-1,\sigma}.$$
Our purpose is to prove the convergence of the sequence $\left(y_{+}^{[m]},y_{-}^{[m]}\right)_{m\in\mathbb{N}}$ in $Y_{\mathbf{m}}^{s-1,\sigma}.$ Let $p,q\in\mathbb{N}$, then using that $s>\tfrac{3}{2}$, we have
\begin{align*}
	\left\|y_{\pm}^{[p]}-y_{\pm}^{[q]}\right\|_{s-1,\sigma}&=\left\|\partial_{x}r_{\pm}\left(x_{\pm}^{[p]}-x_{\pm}^{[q]}\right)\right\|_{s-1,\sigma}\\
	&\lesssim\left\|\check{r}_{\pm}\right\|_{s,\sigma}\left\|x_{\pm}^{[k_p]}-x_{\pm}^{[k_q]}\right\|_{s-1,\sigma}.
\end{align*}
Since the sequence $\left(x_{+}^{[k_m]},x_{-}^{[k_m]}\right)_{m\in\mathbb{N}}$ is convergent in $X_{\mathbf{m}}^{s-1,\sigma},$ then it is of Cauchy-type in $X_{\mathbf{m}}^{s-1,\sigma}.$ We deduce that the sequence $\left(y_{+}^{[m]},y_{-}^{[m]}\right)_{m\in\mathbb{N}}$ is of Cauchy-type (thus convergent) in the Banach space $Y_{\mathbf{m}}^{s-1,\sigma}.$ Consequently, the operator $d_{(\check{r}_+,\check{r}_-)}F(a,b,c,\check{r}_+,\check{r}_-):X_{\mathbf{m}}^{s,\sigma}\rightarrow Y_{\mathbf{m}}^{s-1,\sigma}$ is a compact perturbation of an isomorphism. Therefore, it is a Fredholm operator with index zero.\\
	\textbf{(iii)} Let $\left(c^{[k]},\check{r}_+^{[k]},\check{r}_-^{[k]}\right)_{k\in\mathbb{N}}\in\big(\mathscr{S}_n(a,b)\big)^{\mathbb{N}}.$ Since $\big(c^{[k]}\big)_{k\in\mathbb{N}}\in[-n,n]^{\mathbb{N}}$, then by Bolzano-Weierstrass Theorem, there exists $c^{[\infty]}\in[-n,n]$ and a subsequence $(k_m)_{m\in\mathbb{N}}$ such that
	$$\lim_{m\to\infty}c^{[k_m]}=c^{[\infty]}.$$
	Since $\left(\check{r}_+^{[k]},\check{r}_-^{[k]}\right)_{k\in\mathbb{N}}$ is bounded in $X_{\mathbf{m}}^{s,\sigma},$ then there exists $\left(\check{r}_+^{[\infty]},\check{r}_-^{[\infty]}\right)\in X_{\mathbf{m}}^{s,\sigma}$ such that, up to an other extraction,
	$$\left(\check{r}_+^{[k_m]},\check{r}_-^{[k_m]}\right)\underset{m\to\infty}{\rightharpoonup}\left(\check{r}_+^{[\infty]},\check{r}_-^{[\infty]}\right)\quad\textnormal{in }X_{\mathbf{m}}^{s,\sigma}$$
	and
	$$\forall\, \tfrac{3}{2}<s'<s,\quad\left(\check{r}_+^{[k_m]},\check{r}_-^{[k_m]}\right)\underset{m\to\infty}{\longrightarrow}\left(\check{r}_+^{[\infty]},\check{r}_-^{[\infty]}\right)\quad\textnormal{in }X_{\mathbf{m}}^{s',\sigma}.$$
	Since, $s'>\tfrac{3}{2},$ by pointwise convergence, we obtain
	$$F\left(a,b,c^{[\infty]},\check{r}_+^{[\infty]},\check{r}_-^{[\infty]}\right)=0.$$
	We shall now prove that the sequence $\left(\check{r}_+^{[k_m]},\check{r}_-^{[k_m]}\right)_{m\in\mathbb{N}}$ is a Cauchy sequence (and thus convergent) in the Banach space $X_{\mathbf{m}}^{s,\sigma}.$ Let $p,q\in\mathbb{N}$, since 
	$$F\left(a,b,c^{[k_p]},\check{r}_+^{[k_p]},\check{r}_-^{[k_p]}\right)=0=F\left(a,b,c^{[k_q]},\check{r}_+^{[k_q]},\check{r}_-^{[k_q]}\right),$$
	then substracting, we obtain
	$$\partial_{x}\left(\check{r}_{+}^{[k_p]}-\check{r}_+^{[k_q]}\right)=\mathcal{I}_{1}^{\,+,p,q}+\mathcal{I}_{2}^{\,+,p,q},\qquad\partial_{x}\left(\check{r}_{-}^{[k_p]}-\check{r}_-^{[k_q]}\right)=\mathcal{I}_{1}^{\,-,p,q}+\mathcal{I}_{2}^{\,-,p,q},$$
	where
	$$\begin{array}{ll}
		\mathcal{I}_{1}^{\,+,p,q}\triangleq\frac{\left(c^{[k_p]}-c^{[k_q]}+\check{r}_+^{[k_q]}-\check{r}_+^{[k_p]}\right)\partial_{x}\check{r}_+^{[k_q]}}{r_+^{[k_p]}+b-c^{[k_p]}},\hspace{1cm} & \mathcal{I}_{1}^{\,-,p,q}\triangleq\frac{\left(c^{[k_p]}-c^{[k_q]}+\check{r}_-^{[k_q]}-\check{r}_-^{[k_p]}\right)\partial_{x}\check{r}_-^{[k_q]}}{r_-^{[k_p]}+a-c^{[k_p]}},\vspace{0.2cm}\\
		\mathcal{I}_{2}^{\,+,p,q}\triangleq \frac{\partial_{x}^{-1}\left(\check{r}_+^{[k_p]}-\check{r}_+^{[k_q]}-\check{r}_-^{[k_p]}+\check{r}_-^{[k_q]}\right)}{(b-a)\left(r_+^{[k_p]}+b-c^{[k_p]}\right)}, & \mathcal{I}_{2}^{\,-,p,q}\triangleq \frac{\partial_{x}^{-1}\left(\check{r}_+^{[k_p]}-\check{r}_+^{[k_q]}-\check{r}_-^{[k_p]}+\check{r}_-^{[k_q]}\right)}{(b-a)\left(r_+^{[k_p]}+a-c^{[k_p]}\right)}\cdot
	\end{array}$$
Since 
$$\forall p\in\mathbb{N},\quad\min\left(\min_{x\in\mathbb{T}}\left|\check{r}_+^{[k_p]}(x)+b-c^{[k_p]}\right|\,,\,\min_{x\in\mathbb{T}}\left|\check{r}_-^{[k_p]}(x)+a-c^{[k_p]}\right|\right)\geqslant\tfrac{1}{n},$$
then
	\begin{align*}
		\left\|\mathcal{I}_{1}^{\,\pm,p,q}\right\|_{s-1,\sigma}&\lesssim_{n,a,b}\left(\left|c^{[k_p]}-c^{[k_q]}\right|+\left\|\check{r}_{\pm}^{[k_p]}-\check{r}_{\pm}^{[k_q]}\right\|_{s-1,\sigma}\right)\left\|\partial_{x}\check{r}_{\pm}^{[k_q]}\right\|_{s-1,\sigma}\\
		&\lesssim_{n,a,b} \left(\left|c^{[k_p]}-c^{[k_q]}\right|+\left\|\check{r}_{\pm}^{[k_p]}-\check{r}_{\pm}^{[k_q]}\right\|_{s',\sigma}\right)\left\|\check{r}_{\pm}^{[k_q]}\right\|_{s,\sigma}\\
		&\lesssim_{n,a,b}\left(\left|c^{[k_p]}-c^{[k_q]}\right|+\left\|\check{r}_{\pm}^{[k_p]}-\check{r}_{\pm}^{[k_q]}\right\|_{s',\sigma}\right)
	\end{align*}
and
	\begin{align*}
		\left\|\mathcal{I}_{2}^{\,\pm,p,q}\right\|_{s-1,\sigma}&\lesssim_{n,a,b}\left\|\partial_{x}^{-1}\left(\check{r}_+^{[k_p]}-\check{r}_+^{[k_q]}+\check{r}_-^{[k_p]}-\check{r}_-^{[k_q]}\right)\right\|_{s-1,\sigma}\\
		&\lesssim_{n,a,b}\left(\left\|\check{r}_+^{[k_p]}-\check{r}_+^{[k_q]}\right\|_{s',\sigma}+\left\|\check{r}_-^{[k_p]}-\check{r}_-^{[k_q]}\right\|_{s',\sigma}\right).
	\end{align*}
	Since the sequences $\big(c^{[k_m]}\big)_{m\in\mathbb{N}}$ and $\big(\check{r}_+^{[k_m]},\check{r}_-^{[k_m]}\big)$ and convergent in $\mathbb{R}$ are $X_{\mathbf{m}}^{s',\sigma}$ respectively, they are in particular of Cauchy-type in the corresponding spaces. This gives the desired result, i.e.
	$$\left(\check{r}_+^{[k_m]},\check{r}_-^{[k_m]}\right)\underset{m\to\infty}{\longrightarrow}\left(\check{r}_+^{[\infty]},\check{r}_-^{[\infty]}\right)\quad\textnormal{in }X_{\mathbf{m}}^{s,\sigma}.$$
	Thus, for any $n\in\mathbb{N}^*$, the set $\mathscr{S}_n(a,b)$ is compact in $\mathbb{R}\times X_{\mathbf{m}}^{s,\sigma}.$ This ends the proof of Lemma \ref{lem compactness global bif}.
\end{proof}
Now, we can conclude.
\begin{proof}[Proof of Theorem \ref{thm global E-states}-(i)] Let $\kappa\in\{+,-\}.$ The Lemma \ref{lem compactness global bif} allows to apply the Theorem \ref{thm BT} which provides the existence of a global continuation curve $\mathscr{C}_{\textnormal{\tiny{global}}}^{\kappa,\mathbf{m}}(a,b)$ satisfying
	$$\mathscr{C}_{\textnormal{\tiny{local}}}^{\kappa,\mathbf{m}}(a,b)\subset\mathscr{C}_{\textnormal{\tiny{global}}}^{\kappa,\mathbf{m}}(a,b)\triangleq\Big\{\big(c_{\mathbf{m}}^{\kappa}(\mathtt{s},a,b),\check{r}_{\mathbf{m}}^{\kappa}(\mathtt{s},a,b)\big),\quad\mathtt{s}\in\mathbb{R}\Big\}\subset U(a,b)\cap F(a,b,\cdot,\cdot,\cdot)^{-1}\big(\{0\}\big).$$
	Moreover, the curve $\mathscr{C}_{\textnormal{\tiny{global}}}^{\kappa,\mathbf{m}}(a,b)$ admits locally around each of its points a real-analytic reparametrization. In addition, one of the following alternatives occurs
	\begin{itemize}
		\item [$(A1)$] There exists $T_{\mathbf{m}}^{\kappa}(a,b)>0$ such that 
		$$\forall\mathtt{s}\in\mathbb{R},\quad c_{\mathbf{m}}^{\kappa}\big(\mathtt{s}+T_{\mathbf{m}}^{\kappa}(a,b),a,b\big)=c_{\mathbf{m}}^{\kappa}(\mathtt{s},a,b)\qquad\textnormal{and}\qquad \check{r}_{\mathbf{m}}^{\pm}\big(\mathtt{s}+T_{\mathbf{m}}^{\kappa}(a,b),a,b\big)=\check{r}_{\mathbf{m}}^{\pm}(\mathtt{s},a,b).$$
		\item [$(A2)$] One one the following limits holds (possibly simultaneously)
		\begin{enumerate}[label=\textbullet]
			\item (Blow-up) $\displaystyle\lim_{\mathtt{s}\to\pm\infty}\frac{1}{1+\left|c_{\mathbf{m}}^{\kappa}(\mathtt{s},a,b)\right|+\|\check{r}_{\mathbf{m}}^{\kappa}(\mathtt{s},a,b)\|_{s,\sigma}}=0.$
			\item (Collision of the boundaries) $\displaystyle\lim_{\mathtt{s}\to\pm\infty}\min_{x\in\mathbb{T}}\left|\big(\check{r}_{\mathbf{m}}^{\kappa}\big)_{+}(\mathtt{s},a,b)(x)-\big(\check{r}_{\mathbf{m}}^{\kappa}\big)_{-}(\mathtt{s},a,b)(x)+b-a\right|=0.$
			\item (Degeneracy $+$) $\displaystyle\lim_{\mathtt{s}\to\pm\infty}\min_{x\in\mathbb{T}}\left|\big(\check{r}_{\mathbf{m}}^{\kappa}\big)_{+}(\mathtt{s},a,b)(x)+b-c_{\mathbf{m}}^{\kappa}(\mathtt{s},a,b)\right|=0.$
			\item (Degeneracy $-$) $\displaystyle\lim_{\mathtt{s}\to\pm\infty}\min_{x\in\mathbb{T}}\left|\big(\check{r}_{\mathbf{m}}^{\kappa}\big)_{-}(\mathtt{s},a,b)(x)+a-c_{\mathbf{m}}^{\kappa}(\mathtt{s},a,b)\right|=0.$
		\end{enumerate}
	\end{itemize}
This gives the desired result.
\end{proof}
	Now, to prove the other items of Theorem \ref{thm global E-states}, we proceed similarly by replacing $U(a,b)$ by one of the following open sets
	\begin{align*}
		V(a,c)&\triangleq\Big\{(b,\check{r}_+,\check{r}_-)\in(a,\infty)\times X_{\mathbf{m}}^{s,\sigma}\quad\textnormal{s.t.}\quad\min\big(\mathtt{m}(a,b),\mathtt{m}_{+}(b,c),\mathtt{m}_{-}(a,c)\big)>0\Big\},\\
		W(b,c)&\triangleq\Big\{(a,\check{r}_+,\check{r}_-)\in(-\infty,b)\times X_{\mathbf{m}}^{s,\sigma}\quad\textnormal{s.t.}\quad\min\big(\mathtt{m}(a,b),\mathtt{m}_{+}(b,c),\mathtt{m}_{-}(a,c)\big)>0\Big\},\\
		Z(c)&\triangleq\Big\{(a,\check{r}_+,\check{r}_-)\in(0,\infty)\times X_{\mathbf{m}}^{s,\sigma}\quad\textnormal{s.t.}\quad\min\big(\mathtt{m}(-a,a),\mathtt{m}_{+}(a,c),\mathtt{m}_{-}(-a,c)\big)>0\Big\}
	\end{align*}
 and $F_n(a,b)$ by one of the following closed and bounded sets
 \begin{align*}
 	G_{n}(a,c)&\triangleq\Big\{(b,\check{r}_+,\check{r}_-)\in[a+\tfrac{1}{n},a+n]\times B_{\mathbf{m}}^{s,\sigma}(n)\quad\textnormal{s.t.}\quad\min\big(\mathtt{m}(a,b),\mathtt{m}_{+}(b,c),\mathtt{m}_{-}(a,c)\big)\geqslant\tfrac{1}{n}\Big\},\\
 	H_{n}(b,c)&\triangleq\Big\{(a,\check{r}_+,\check{r}_-)\in[b-n,b-\tfrac{1}{n}]\times B_{\mathbf{m}}^{s,\sigma}(n)\quad\textnormal{s.t.}\quad\min\big(\mathtt{m}(a,b),\mathtt{m}_{+}(b,c),\mathtt{m}_{-}(a,c)\big)\geqslant\tfrac{1}{n}\Big\},\\
 	I_{n}(c)&\triangleq\Big\{(a,\check{r}_+,\check{r}_-)\in[\tfrac{1}{n},n]\times B_{\mathbf{m}}^{s,\sigma}(n)\quad\textnormal{s.t.}\quad\min\big(\mathtt{m}(-a,a),\mathtt{m}_{+}(a,c),\mathtt{m}_{-}(-a,c)\big)\geqslant\tfrac{1}{n}\Big\}.
 \end{align*}
The only difference is that in the alternative $(A2)$, one has to add, for instance in the case of the bifurcation in the parameter $b,$
\begin{enumerate}[label=\textbullet]
	\item (Vanishing degeneracy) $\displaystyle\lim_{\mathtt{s}\to\pm\infty}b_{\mathbf{m}}(\mathtt{s},a,c)=a.$
\end{enumerate}
But this situation can be included in the "Collision of the boundaries" alternative.
	
	\appendix
	\section{Elements of bifurcation theory}
	The purpose of this appendix is to expose the theoretical bifurcation results used in this work. We first start by the classical local bifurcation theorem of Crandall-Rabinowitz \cite{CR71}, see also \cite[p. 15]{K11}. The version presented here is in the analytic setting which fits more with our goal. We also add to the statement the required conditions to get a pitchfork-type bifurcation. For more details, we refer the reader to the works of Shi \cite{S99} and Liu-Shi \cite{LS22}.
	\begin{theo}\label{thm CR+S}
		\textbf{(Analytic local bifurcation + pitchfork property)}\\
		Let $X$ and $Y$ be two Banach spaces. Let $(\lambda_0,u_0)\in\mathbb{R}\times X$ and $U$ be a neighborhood of $(\lambda_0,u_0)$ in $\mathbb{R}\times X.$ Consider a real-analytic function $F:U\rightarrow Y$ such that
		\begin{itemize}
			\item [$(L1)$] $\forall(\lambda,u_0)\in U,\quad F(\lambda,u_0)=0.$
			\item [$(L2)$] $d_uF(\lambda_0,u_0)$ is a Fredholm operator with
			$$\dim\Big(\ker\big(d_uF(\lambda_0,u_0)\big)\Big)=1=\textnormal{codim}\Big(R\big(d_uF(\lambda_0,u_0)\big)\Big),\qquad\ker\big(d_uF(\lambda_0,u_0)\big)=\mathtt{span}(w_0).$$
			\item [$(L3)$] Transversality:
			$$\partial_{\lambda}d_uF(\lambda_0,u_0)\not\in R\big(d_uF(\lambda_0,u_0)\big).$$
		\end{itemize}
		If we decompose
		$$X=\mathtt{span}(w_0)\oplus Z,$$
		then there exist two real-analytic functions
		$$\lambda:(-\delta,\delta)\rightarrow\mathbb{R}\qquad\textnormal{and}\qquad z:(-\delta,\delta)\rightarrow Z,\qquad\textnormal{with}\qquad\delta>0,$$
		such that
		$$\lambda(0)=\lambda_0,\qquad z(0)=0$$
		and the set of zeros of $F$ in $U$ is the union of two curves 
		$$\big\{(\lambda,u)\in U\quad\textnormal{s.t.}\quad F(\lambda,u)=0\big\}=\big\{(\lambda,u_0)\in U\big\}\cup\mathscr{C}_{\textnormal{\tiny{local}}},\qquad\mathscr{C}_{\textnormal{\tiny{local}}}\triangleq\big\{\big(\lambda(\mathtt{s}),u_0+\mathtt{s}w_0+\mathtt{s}z(\mathtt{s})\big),\quad|\mathtt{s}|<\delta\big\}.$$
		Assume in addition that
		$$d_u^2F(\lambda_0,u_0)[w_0,w_0]\in R\big(d_uF(\lambda_0,u_0)\big).$$
		Then $\lambda'(0)=0$ and if we denote
		$$R\big(d_uF(\lambda_0,u_0)\big)=\ker(l)\qquad\textnormal{ for some }l\in Y^*,$$
		then we have
		$$\lambda''(0)=\frac{3\big\langle l\,,\,d_u^2F(\lambda_0,u_0)[w_0,\theta_0]\big\rangle-\big\langle l\,,\,d_u^3F(\lambda_0,u_0)[w_0,w_0,w_0]\big\rangle}{3\big\langle l\,,\,\partial_\lambda d_uF(\lambda_0,u_0)\big\rangle},$$
		where $\theta_0$ is solution of
		$$d_uF(\lambda_0,u_0)[\theta_0]=d_u^2F(\lambda_0,u_0)[w_0,w_0].$$
		If $\lambda''(0)\neq 0,$ we say that the bifurcation is of pitchfork-type. More precisely, the condition $\lambda''(0)>0$ (resp. $\lambda''(0)<0$) is called supercritical (resp. subcritical) bifurcation.
	\end{theo}
	Finally, we present the classical global bifurcation theorem of Dancer \cite{D73}, Buffoni-Toland \cite[Thm. 9.1.1]{BT03}. The version given here is taken from \cite[Thm. 4]{CSV16}.
	\begin{theo}\label{thm BT}
		\textbf{(Analytic global bifurcation)}
		Let $X$ and $Y$ be two Banach spaces. Let $U$ be an open subset of $\mathbb{R}\times X.$ Consider a real-analytic function $F:U\rightarrow Y$ satisfying $(L1)$, $(L2)$, $(L3)$ and the following additional properties.
		\begin{itemize}
			\item [$(G1)$] For any $(\lambda,u)\in U$ such that $F(\lambda,u)=0$, the operator $d_uF(\lambda,u)$ is a Fredholm operator of index $0.$
			\item [$(G2)$] Assume that we can write
			$$U=\bigcup_{n\in\mathbb{N}}F_n,$$
			where for any $n\in\mathbb{N},$ the set $F_n$ is bounded and closed in $\mathbb{R}\times X.$ Suppose that for any $n\in\mathbb{N},$ the set 
			$$\mathscr{S}_n\triangleq\Big\{(\lambda,u)\in F_n\quad\textnormal{s.t.}\quad F(\lambda,u)=0\Big\}$$
			is compact in $\mathbb{R}\times X.$  
		\end{itemize}
		Then there exists a unique (up to reparametrization) continuous curve $\mathscr{C}_{\textnormal{\tiny{global}}}$ such that
		$$\mathscr{C}_{\textnormal{\tiny{local}}}\subset\mathscr{C}_{\textnormal{\tiny{global}}}\triangleq \Big\{\big(\lambda(\mathtt{s}),u(\mathtt{s})\big),\quad\mathtt{s}\in\mathbb{R}\Big\}\subset U\cap F^{-1}\big(\{0\}\big).$$
		Moreover, $\mathscr{C}_{\textnormal{\tiny{global}}}$ admits locally around each of its points a real-analytic parametrization. In addition, one of the following alternatives occurs
		\begin{itemize}
			\item [$(A1)$] there exists $T>0$ such that
			$$\forall\mathtt{s}\in\mathbb{R},\quad\lambda(\mathtt{s}+T)=\lambda(\mathtt{s})\qquad\textnormal{and}\qquad u(\mathtt{s}+T)=u(\mathtt{s}).$$
			\item [$(A2)$] for any $n\in\mathbb{N},$ there exists $\mathtt{s}_n>0$ such that 
			$$\forall\mathtt{s}>\mathtt{s}_n,\quad\big(\lambda(\mathtt{s}),u(\mathtt{s})\big)\not\in F_n.$$
		\end{itemize}
	\end{theo}
\section{A useful lemma in linear algebra}
We present here a simple lemma used in our analysis to describe the range of the linearized operator.
\begin{lem}\label{lem codim1}
	Let $E$ be a vector space over a field $\mathbb{K}.$ Let $V_1$ and $V_2$ be two subspaces of $E$ of codimension one in $E$ and such that $V_2\subset V_1.$ Then $V_2=V_1.$
\end{lem}
\begin{proof}
	By definition, since $V_1$ and $V_2$ are of codimension one in $E$, then there exists $(u_1,u_2)\in E^2$ such that $u_1\not\in V_1$, $u_2\not\in V_2$ and
	\begin{equation}\label{split V1V2}
		V_1\oplus\langle u_1\rangle=E=V_2\oplus\langle u_2\rangle.
	\end{equation}
	Let $v_1\in V_1.$ According to \eqref{split V1V2}, there exists $(v_2,\lambda_u)\in V_2\times\mathbb{K}$ such that
	$$v_1=v_2+\lambda_u u_2.$$
	Assume for the sake of contradiction that $\lambda_u\neq0.$ Then, using the fact that $V_1$ is a vector space together with $v_1\in V_1$ and $v_2\in V_2\subset V_1,$ we have
	$$u_2=\lambda_u^{-1}(v_1-v_2)\in V_1.$$
	Combined with the hypothesis $V_2\subset V_1$, we deduce that
	$$V_2\oplus\langle u_2\rangle\subset V_1.$$
	This contradicts \eqref{split V1V2}. Thus, $\lambda_u=0$ and $v_1=v_2\in V_2.$ This achieves the proof of the~lemma.
\end{proof}
We mention that this lemma is just a reformulation of the following classical result on hyperplanes/linear forms.
\begin{lem}
	Let $E$ be a vector space over a field $\mathbb{K}$. Then,
	\begin{enumerate}
		\item The hyperplanes (subspaces of codimension one) of $E$ are exactly the kernels of non-zero linear forms on $E$.
		\item Let $\varphi$ be a non-zero linear form on $E$. Then, any linear form vanishing on  the hyperplane $\ker(\varphi)$ is proportional to $\varphi$.
		\item Two non-zero linear forms on $E$ have the same kernel if and only if they are proportional. 
	\end{enumerate}
\end{lem}
	
\end{document}